\documentclass[onefignum,onetabnum]{siamart251216}
\headers{SFFT}
{W. Hou, K. Jiang, J. Tang}

\title{Mathematical foundation of symmetric fast Fourier transform}

\author{
Weifeng Hou\thanks{
Hunan Key Laboratory for Computation and Simulation in Science and Engineering,
Key Laboratory of Intelligent Computing and Information Processing of Ministry of Education,
School of Mathematics and Computational Science,
Xiangtan University, Xiangtan, Hunan, 411105, China
(\email{houwf@smail.xtu.edu.cn},
\email{tangjf@smail.xtu.edu.cn}).
}
\and
Kai Jiang\thanks{
Wuhan Institute for Math \& AI,
Wuhan University, Wuhan, Hubei, 430072, China
(\email{kaijiang@whu.edu.cn}).
}
\and
Jingfei Tang\footnotemark[1]
}
\usepackage{mathrsfs,amsmath,amssymb,bm}
\usepackage{lipsum}
\usepackage{amsfonts}
\usepackage{graphicx, subfigure}
\usepackage{epstopdf}
\usepackage{makecell, rotating, bbding}
\usepackage{multirow}
\usepackage{algorithmic, algorithm}
\usepackage{enumerate}
\usepackage{booktabs}
\usepackage[misc]{ifsym}
\usepackage{mathtools}
\usepackage[draft]{changes}
\usepackage{amsopn}
\usepackage{aligned-overset}
\usepackage{nicematrix}
\usepackage{tikz}
\usepackage{dashbox}
\usetikzlibrary{calc}
\usepackage{hyperref}
\usepackage{enumitem}
\usepackage{cases}
\usepackage{longtable}

\newtheorem{thm}{Theorem}[section]

\newtheorem{remark}[thm]{Remark}

\usepackage[draft]{changes}
\usepackage{lipsum}% <- For dummy text

\newsiamremark{hypothesis}{Hypothesis}
\crefname{hypothesis}{Hypothesis}{Hypotheses}
\newsiamthm{claim}{Claim}

\usepackage{amsopn}

\makeatletter
\newcommand*{\addFileDependency}[1]{% argument=file name and extension
  \typeout{(#1)}% latexmk will find this if $recorder=0 (however, in that case, it will ignore #1 if it is a .aux or .pdf file etc and it exists! if it doesn't exist, it will appear in the list of dependents regardless)
  \@addtofilelist{#1}% if you want it to appear in \listfiles, not really necessary and latexmk doesn't use this
  \IfFileExists{#1}{}{\typeout{No file #1.}}% latexmk will find this message if #1 doesn't exist (yet)
}
\makeatother

\usepackage{xr}

\begin{document}

\maketitle

% arXiv submission source
	
\begin{abstract}
    In this work, we study the fast Fourier transform (FFT) of discrete data with crystallographic symmetry, for which standard FFT algorithms incur substantial redundant computation. We establish a rigorous mathematical theory of the symmetric fast Fourier transform (SFFT) for $n$-dimensional crystallographic groups and propose the SFFT and symmetric inverse fast Fourier transform (SIFFT) algorithms. We derive symmetry-reduction formulas for both linear and translational symmetries and unify them into a recursive framework. We analyze the arithmetic complexity to obtain theoretical speedup factors relative to the standard FFT, and we confirm the accuracy and efficiency of the proposed methods by numerical experiments on representative three-dimensional space groups.
\end{abstract}



\section{Introduction}
    \label{sec:intr}
    The discrete Fourier transform (DFT) is a fundamental linear transform that connects discrete data in physical space with their spectral representations. As a core tool in modern scientific computing, the DFT is widely used in digital signal processing, crystallographic diffraction-data processing, and Fourier pseudospectral methods for periodic partial differential equations \cite{oppenheim1999discrete,gonzalez2008digital,bricogne2001fourier,shen2011fourier}. For data with total grid size $N$, a direct evaluation of the DFT has computational complexity $O(N^2)$, whereas Cooley--Tukey-type fast Fourier transform (FFT) algorithms reduce this cost to $O(N\log N)$ by decomposing the underlying indices \cite{cooley1965algorithm}. Together with highly optimized implementations such as FFTW, these algorithms have become standard tools in scientific computing, image processing, machine learning, and data analysis \cite{frigo1998fftw,gonzalez2008digital,mathieu2014fast,brillinger2001time}. Nevertheless, the essential mechanism of the standard Cooley--Tukey FFT is to exploit algebraic decompositions of the underlying cyclic index groups, and it does not explicitly make use of possible geometric symmetries in the input data. Consequently, in highly symmetric settings such as crystal-structure computations, standard FFT algorithms still operate on the full grid, including many degrees of freedom that are not independent but constrained by group actions. This leads to substantial redundant computation and storage overhead.
    
    In crystallography, fast algorithms that exploit space-group symmetry—i.e., the symmetry of three-dimensional crystallographic groups (\cref{def:crystallographic_group})—to accelerate discrete Fourier transforms are usually referred to as crystallographic fast Fourier transforms (CFFTs). Ten Eyck was among the first to incorporate crystallographic symmetry into the algebraic decomposition of Cooley--Tukey-type FFTs, thereby achieving significant complexity reductions for certain space groups \cite{ten1973crystallographic}. Later, Bricogne gave a systematic formulation of the crystallographic discrete Fourier transform in the \emph{International Tables for Crystallography} in terms of group actions on finite real-space and reciprocal-space index sets, clarifying the coupling between space-group actions and multidimensional Cooley--Tukey decompositions \cite{bricogne2001fourier}. Building on this foundation, Rowicka, Kudlicki, and Otwinowski further developed algorithms based on one-step symmetry reduction, centering-translation reduction, and recursive reduction at special positions, ultimately producing a complete CFFT framework applicable to all 230 three-dimensional space groups \cite{rowicka2002crystallographicI,rowicka2003crystallographicII,rowicka2003crystallographicIII,kudlicki2007crystallographic}. More recently, CFFTs have also been applied to pseudospectral computations in self-consistent field theory for block copolymers, where they have demonstrated substantial savings in both time and memory \cite{qiang2020accelerated}.
    
    Although existing work has established a mature CFFT methodology for three-dimensional space groups, there remains room for further abstraction and systematization from the viewpoint of a unified mathematical formulation. Existing constructions typically depend rather strongly on the classification of specific space groups, on particular factorizations of the computational grid, and on choices of asymmetric units. When a single reduction step is insufficient to eliminate all symmetry redundancy, or when centering translations and special positions occur on the grid, additional dedicated reduction procedures and recursive strategies are often required. It is therefore natural to ask whether the constraints induced by the linear point-group part and by the fractional translations can be treated within a single algebraic framework, and whether that framework can be coupled with the recursive structure of Cooley--Tukey decompositions.
    
    In this paper, we propose a unified SFFT/SIFFT framework based on discrete group actions for periodic discrete data with finite affine symmetries. The framework applies to general $n$-dimensional crystallographic groups and combines the hierarchical construction of finite group actions with symmetry-reduction formulas. By introducing $G$-admissible grids, we induce the affine action of a crystallographic group as a finite permutation action on a discrete index group $X_{\bm A}$. This action is then combined with the subgrid decomposition and orbit--stabilizer structure arising from multidimensional Cooley--Tukey factorizations, thereby reducing the global DFT/IDFT to local DFTs/IDFTs on orbit representatives together with finite-dimensional matrix recombinations. Compared with existing CFFT methods, which are mainly organized around the classification of three-dimensional space groups, explicit grid constructions, and the selection of asymmetric units, the present formulation unifies a broad class of geometric symmetry constraints as actions of finite quotient groups on discrete index sets. It therefore provides a compact algebraic framework for layer-by-layer symmetry reduction for general $n$-dimensional crystallographic groups.

    The main contributions of this paper are summarized as follows:
    \begin{itemize}
        \item \textbf{An arithmetic representation of crystallographic groups and a discrete computational framework.}
        We prove that, in suitable coordinates, a crystallographic group can be represented arithmetically with an integral linear part, a rational translational part, and a finite quotient group, and we establish the invariance of this representation under conventional bases. Using the composition series of the finite quotient and the one- or two-generator structure of its simple factors, we prove that lifted generators construct all symmetry operations and recover the complete orbit of a representative point. For three-dimensional crystallographic groups, solvability reduces this construction to one lifted generator per factor. We then introduce $G$-admissible grids together with algebraic admissibility criteria, providing a unified foundation for discretizing crystallographic symmetries on finite sampling grids.
    
        \item \textbf{Two fundamental symmetry-reduction formulas.}
        For the cases of nontrivial linear actions and pure translational actions, respectively, we derive SFFT/SIFFT reduction formulas, give their corresponding matrix representations, and prove the invertibility of the core constant matrices as well as the effective frequency-selection mechanism in the pure translation case.
    
        \item \textbf{A unified recursive framework for symmetry reduction.}
        We formally define well-formed SFFT problems and prove that the subproblems generated by nontrivial linear reductions inherit the space-group structure and the corresponding discrete symmetries. This allows linear reductions, translation reductions, and standard FFT calls to be organized into a single recursive procedure. A concrete construction is presented for the space group $Im\bar{3}m$.
    
        \item \textbf{Complexity analysis and numerical validation.}
        We estimate the arithmetic complexity and data-access complexity of SFFT-LIN/SIFFT-LIN and SFFT-TRN/SIFFT-TRN, derive theoretical speedup factors relative to standard FFT/IFFT algorithms, and validate the accuracy and acceleration of the proposed method for the three space groups $Fm\bar{3}m$, $Im\bar{3}m$, and $Ia\bar{3}d$.
    \end{itemize}

    The remainder of this paper is organized as follows. \Cref{sec:arit} introduces the geometric definition of crystallographic groups, their arithmetic representations, invariance under conventional bases, and the quotient-group computational framework. \Cref{sec:solv} reviews composition series and solvability, proves the hierarchical reconstruction of finite group actions, and presents composition-series examples for representative three-dimensional crystallographic groups. \Cref{sec:math} develops the mathematical foundations of SFFT, including torus representations, discrete index groups, $G$-admissible grids, induced discrete actions, and multidimensional Cooley--Tukey decompositions. \Cref{sec:algor} presents the core algorithmic framework, derives the SFFT/SIFFT reduction formulas for both nontrivial linear actions and pure translations, and constructs the unified recursive symmetry-reduction framework. \Cref{sec:complexity} analyzes the algorithmic complexity. \Cref{sec:algorithm_example} uses the space group $Im\bar{3}m$ to illustrate the concrete implementation of the recursive SFFT. \Cref{sec:numerical_results} reports numerical experiments validating the accuracy and efficiency of the proposed algorithms. The conclusion and outlook of this paper are given in \Cref{sec:conclusion}. The appendix provides explicit expressions for the constant matrices and phase matrices used in the recursive example.

\section*{Notations and conventions}
	The main notations and conventions used in this paper are summarized in the following table.

	\begin{longtable}{l >{\raggedright\arraybackslash}p{0.75\textwidth}}
		\caption{Summary of notations and conventions used in this paper}\label{tab:notation}\\
		\toprule
		\textbf{Symbol}                                  & \textbf{Description}                                         \\
		\midrule
		\endfirsthead
		\toprule
		\textbf{Symbol}                                  & \textbf{Description}                                         \\
		\midrule
		\endhead
		\midrule
		\multicolumn{2}{r}{Continued on next page}                                                                     \\
		\midrule
		\endfoot
		\bottomrule
		\endlastfoot

		\multicolumn{2}{l}{\textbf{Basic objects and parameters}} \\
		\midrule
		$n$ & Spatial dimension\\
	$\mathbb{R}^n, \mathbb{Q}^n$   & $n$-dimensional real and rational vector spaces\\
	$\mathbb{Z}^n$                 & Set of $n$-dimensional integer vectors\\
	$\mathbb{C}$                   & Field of complex numbers\\
	$\mathbb{Z}_{\geq 1}$          & Set of positive integers \\
	$\mathbb{T}^n$                 & $n$-dimensional torus $\mathbb{R}^n / \mathbb{Z}^n$                \\
	$\bm I_n$                      & Identity matrix of order $n$                                             \\
	$D$                            & Total grid size $D=\det(\bm A)=|X_{\bm A}|=\prod_{j=1}^n N_j$                 \\
	$N_j$                               & Diagonal entries of $\bm A=\operatorname{diag}(N_1,\dots,N_n)$ \\
	$D_0, D_1$                     & Subgrid sizes after Cooley--Tukey decomposition (satisfying $D=D_0D_1$)    \\
	$N_{0,j},N_{1,j}$ & Diagonal entries of the factors $\bm A_0,\bm A_1$, with $N_j=N_{0,j}N_{1,j}$ \\
	$M$                            & Order of the group $M\coloneqq |G_{\bm A}|$                               \\
	$d$                            & Number of effective frequencies $|X_{\bm A_T^1}|/M$                                \\
		\midrule

		\multicolumn{2}{l}{\textbf{Affine groups, matrix groups, and crystallographic groups}} \\
		\midrule
	$\operatorname{Aff}(\mathbb{R}^n)$             & $n$-dimensional affine transformation group                                           \\
	$\operatorname{Isom}(\mathbb{R}^n)$            & Isometry group of $n$-dimensional Euclidean space                                     \\
	$O(n)$                                         & $n$-dimensional orthogonal group                                              \\
	$\operatorname{GL}^{\pm 1}(n,\mathbb{Z})$      & $n$-dimensional unimodular matrix group     \\
	$\Gamma$                                       & $n$-dimensional crystallographic group \\
	$\mathcal{G}$                                  & Arithmetic representation of an $n$-dimensional crystallographic group                                         \\
	$T$                                            & Integer translation subgroup $\{(\bm I_n,\bm z):\bm z\in\mathbb Z^n\}$ in $\mathcal{G}$                                 \\
	$T_{\mathcal G}$                     & Full translation subgroup $\mathcal G\cap\mathbb R^n$ \\
	$G$                                            & Quotient group $\mathcal{G}/T$                                  \\
	$P$                                            & Point group                                      \\
		\midrule

		\multicolumn{2}{l}{\textbf{Group actions and related notations}} \\
		\midrule
	$g = (\bm R_g, \bm t_g)$                       & Space group element, comprising a linear part $\bm R_g$ and a translation part $\bm t_g$               \\
	$r$                                            & Homomorphism, mapping $g=(\bm R_g,\bm t_g)$ to $r(g)=\bm R_g$                          \\
	$\odot$                                        & Induced action of $G$ on $\mathbb{T}^n$ \\
	$\Psi_{\bar g}$                                & Discrete action on $X_{\bm A}$ induced by $\bar g\in G$        \\
	$\bm{M}_{\bar g}, \bm{z}_{\bar g}$             & Linear and translation parts satisfying $\Psi_{\bar g}([\bm x])=[\bm M_{\bar g}\bm x] + \bm z_{\bar g}$                    \\
	$G_{\bm A}$                                    & Discrete symmetry group $\{\Psi_{\bar g} : \bar g \in G\}$                  \\
	$\operatorname{Sym}(S)$                        & Permutation group on the set $S$                                     \\
	$O_{[\bm y]}$                                  & Orbit of the set $X^{[\bm y]}$ under the group action                                 \\
	$G_{[\bm y]}$ & Stabilizer subgroup of the set $X^{[\bm y]}$                                  \\
	$\mathcal{R}([\bm y])$                         & Set of coset representatives for the stabilizer subgroup                                           \\
	$Y_{\mathrm{ref}}$                             & Index set of orbit representatives                                              \\
		\midrule

		\multicolumn{2}{l}{\textbf{Lattices and discretized objects}} \\
		\midrule
	$L, \bm{b}_i$                                  & Lattice and its primitive basis vectors                                             \\
	$\Omega$                                       & Compact fundamental domain          \\
	$\bm A$                                        & Diagonal discretization matrix                                           \\
	$\bm s$                                        & Grid shift vector, $\bm s\in [0,1)^n$                                              \\
	$\mathbb{T}^n_{\bm A,\bm s}$                   & Shifted grid $\{[\bm A^{-1}(\bm k+\bm s)]:\bm k\in\mathbb{Z}^n\}\subset\mathbb{T}^n$ \\
	$\Lambda_{\bm A}$                              & Finite discrete subgroup $\bm A^{-1}\mathbb{Z}^n/\mathbb{Z}^n\subset\mathbb{T}^n$ \\
	$X_{\bm A}$                                    & Discrete index group $\mathbb{Z}^n / \bm A \mathbb{Z}^n$         \\
	$X_R^{[\bm y]},X_T^{[\bm y]}$        & Layer of the grid indexed by the class $[\bm y]$, in the nontrivial linear-action and the pure translation case, respectively \\
	$\mathcal{D}_{\bm A}$                          & Set of standard integer representatives of $X_{\bm A}$                                \\
	$\kappa_j$                           & Least common multiple of the denominators of the $j$-th coordinates of the pure translation vectors; it provides the canonical choice $\bm A_T^1=\operatorname{diag}(\kappa_1,\dots,\kappa_n)$ \\
	$\iota_{\bm A,\bm s}$                          & Parameterized bijection (from $X_{\bm A}$ to $\mathbb{T}^n_{\bm A,\bm s}$): $[\bm x]\mapsto[\bm A^{-1}(\bm x+\bm s)]$ \\
	$\mathcal I^{\mathrm{lex}}$          & Lexicographically ordered index sequence \\
	$[\,\cdot\,]$                        & Congruence class in the relevant quotient group \\
		\midrule

		\multicolumn{2}{l}{\textbf{SFFT/SIFFT related notations}} \\
		\midrule
	$\mathcal{P}$                                  & SFFT problem instance                                           \\
	$\langle[\bm h],[\bm x]\rangle$                & Exponential kernel $e^{2\pi i\,\bm h^\mathsf{T}\bm A^{-1}\bm x}$\\
	$f_{\bm A, \bm s}$                             & Discretized function defined on $X_{\bm A}$                                \\
	$\hat{f}_{\bm A, \bm s}$                       & Fourier coefficients of $f_{\bm A, \bm s}$                                         \\
	$\hat{V}_{\bm y}$                              & Local Fourier coefficients in the nontrivial linear-action reduction                                      \\
	$\hat{\bm V}$ & Vector composed of $\hat{V}_{\bm y}$ \\
	$V_{\bm h_1}$                     & Local function in the pure translation reduction             \\
	$\bm F$                            & Vector of function values in the pure translation reduction        \\
	$\bm V$                            & Local transform vector in the pure translation reduction                                       \\
	$\hat{\bm F}$                    & Global Fourier coefficient vector                           \\
	$\bm A_R^0, \bm A_R^1$                         & Cooley--Tukey decomposition factors in the nontrivial linear-action case                             \\
	$\bm A_T^0, \bm A_T^1$                         & Cooley--Tukey decomposition factors in the pure translation case                             \\
	$\Theta_R, \Theta_T$                           & Families of subsets induced by the decomposition                       \\
	$\bm T_{c,r}, \bm T_{s,r}$                     & Constant matrix and phase matrix in the nontrivial linear-action reduction                                       \\
	$\bm T_{c,t}, \bm T_{s,t}$                     & Constant matrix and phase matrix in the pure translation reduction                                       \\
	$H_{\chi}$                                     & Effective frequency index set under the pure translation reduction                                        \\
	$\chi$                                         & Sum of phase factors \\
	superscript $(k)$                    & A quantity carrying the superscript $(k)$ refers to the $k$-th level of the recursive reduction, e.g.\ $\bm A^{(k)},\bm s^{(k)},\mathcal G^{(k)},G_{\bm A^{(k)}},Y_{\mathrm{ref}}^{(k)}$ \\
	\end{longtable}
	
	\section{Arithmetic representation of crystallographic groups and computational framework for quotient groups}
	\label{sec:arit}

	\subsection{Geometric definition and arithmetic representation of crystallographic groups}

In this section we review the geometric definition of crystallographic groups and their standard arithmetic representation, which underpins the computational framework developed here. The classical arithmetic representation reviewed here originates from the Bieberbach theorems \cite{Bieberbach1911,Bieberbach1912,leonard1986bieberbach} and the arithmetic formulation by Zassenhaus \cite{zassenhaus1948} (later recast in algorithmic form by Brown \cite{brown1969algorithm}). Namely, a crystallographic group is affinely equivalent to a subgroup of $\operatorname{Aff}(\mathbb{R}^n)$ with unimodular integer linear parts $\bm{R}_g\in\operatorname{GL}^{\pm 1}(n,\mathbb{Z})$ and a full-rank integer translation lattice $T\cong\mathbb{Z}^n$.

Let $\operatorname{Aff}(\mathbb{R}^n)$ denote the affine transformation group on $\mathbb{R}^n$.
Any $g \in \operatorname{Aff}(\mathbb{R}^n)$ can be uniquely represented as $g = (\bm{R}_g, \bm{t}_g)$, where $\bm{R}_g \in \operatorname{GL}(n, \mathbb{R})$ is the linear part, and $\bm{t}_g \in \mathbb{R}^n$ is the translation part.
This paper adopts the left action convention, meaning that for any $\bm{x}\in\mathbb{R}^n$,
\[
	g \cdot \bm{x} = \bm{R}_g \bm{x} + \bm{t}_g.
\]
The corresponding group multiplication rule is
\[
	g_1 g_2=(\bm{R}_{g_1},\bm{t}_{g_1})(\bm{R}_{g_2},\bm{t}_{g_2}) =(\bm{R}_{g_1}\bm{R}_{g_2},\,\bm{R}_{g_1}\bm{t}_{g_2}+\bm{t}_{g_1}).
\]
Its identity element is $e=(\bm{I}_n,\bm{0})$. The inverse of $g$ is given by
\[
	g^{-1}=(\bm{R}_g^{-1},\,-\bm{R}_g^{-1}\bm{t}_g).
\]
To simplify the notation, we identify the additive group $\mathbb{R}^n$ as the pure translation subgroup of $\operatorname{Aff}(\mathbb{R}^n)$ via the embedding $\bm{t}\mapsto (\bm{I}_n,\bm{t})$, and we directly regard $\bm{t}\in\mathbb{R}^n$ as the corresponding translation $(\bm{I}_n,\bm{t})$ when no confusion arises.

In classical crystallography, a crystallographic group is typically defined as a specific subgroup of the isometry group of a Euclidean space. Let $\operatorname{Isom}(\mathbb{R}^n) \cong O(n) \ltimes \mathbb{R}^n$ denote the isometry group of the $n$-dimensional Euclidean space.

\begin{definition}[$n$-dimensional crystallographic group, {\cite[p.~4]{leonard1986bieberbach}}]\label{def:crystallographic_group}
	A subgroup $\Gamma$ of $\operatorname{Isom}(\mathbb{R}^n)$ is called an $n$-dimensional crystallographic group if it satisfies the following conditions.
	\begin{enumerate}
		\item Discreteness. $\Gamma$ is a discrete subgroup of $\operatorname{Isom}(\mathbb{R}^n)$. Equivalently, there exists a neighborhood $U$ of the identity $e$ in $\operatorname{Isom}(\mathbb{R}^n)$ such that $U\cap\Gamma=\{e\}$.
		\item Cocompactness. The orbit space $\mathbb{R}^n/\Gamma$ is compact. This implies that the fundamental domain has a finite volume, namely, there exists a bounded domain $\Omega \subset \mathbb{R}^n$ such that the regions obtained by the action of distinct elements of $\Gamma$ on $\Omega$ do not overlap and collectively cover the entire space.
	\end{enumerate}
\end{definition}

\begin{remark}
	Cocompactness means that the orbit space $\mathbb{R}^n / \Gamma$ is compact as a topological space. Equivalently, there exists a compact fundamental domain $\Omega$ such that
	\[
		\mathbb{R}^n = \bigcup_{\gamma \in \Gamma} \gamma \cdot \Omega,
	\]
	and these domains are pairwise disjoint except at their boundaries.
\end{remark}

\begin{theorem}[Bieberbach's first theorem, {\cite[p.~17]{leonard1986bieberbach}}]
	Let $\Gamma$ be a crystallographic subgroup of $\operatorname{Isom}(\mathbb{R}^n)$. Define the homomorphism $r: \operatorname{Isom}(\mathbb{R}^n) \to O(n)$ by $r(g) = \bm{R}_g$. Then
	\begin{enumerate}
		\item The point group $r(\Gamma)$ is a finite group.
		\item The translation subgroup $L = \Gamma \cap \mathbb{R}^n$ is a full-rank lattice, i.e., $L \cong \mathbb{Z}^n$ and it spans the entire $\mathbb{R}^n$.
	\end{enumerate}
\end{theorem}

\begin{remark}
	Here, the lattice $L$ is a discrete, full-rank additive subgroup of $\mathbb{R}^n$, which can be generated by a set of linearly independent primitive translation vectors $\{\bm{b}_1, \dots, \bm{b}_n\}$, namely,
	\[
		L = \left\{ \sum_{i=1}^n z_i \bm{b}_i : z_i \in \mathbb{Z} \right\}.
	\]
	Note that the lattice $L$ itself is independent of the choice of coordinate basis, but its generators and corresponding coordinate representations will change with the basis.
\end{remark}
\begin{theorem}[Bieberbach's second theorem, {\cite[p.~19]{leonard1986bieberbach}}]
	Let $\Gamma$ and $\Gamma'$ be two crystallographic subgroups of $\operatorname{Isom}(\mathbb{R}^n)$. If there exists a group isomorphism $f: \Gamma \to \Gamma'$, then there exists an $\alpha \in \operatorname{Aff}(\mathbb{R}^n)$ such that for any $\beta \in \Gamma$,
	\[
		f(\beta) = \alpha \beta \alpha^{-1}.
	\]
	In other words, if two crystallographic groups are isomorphic, they must be conjugate under an affine transformation.
\end{theorem}
\begin{remark}
	$\Gamma$ and $\Gamma'$ are said to be affinely conjugate. The map $g\mapsto\alpha g\alpha^{-1}$ is a topological automorphism of $\operatorname{Aff}(\mathbb{R}^n)$, and its restriction $f:\Gamma\to\Gamma'$ is a topological group isomorphism for the subspace topologies. Moreover, the affine homeomorphism $\alpha:\mathbb{R}^n\to\mathbb{R}^n$ induces a homeomorphism $\bar\alpha:\mathbb{R}^n/\Gamma\to\mathbb{R}^n/\Gamma'$ by $\bar\alpha([\bm x]_\Gamma)=[\alpha\cdot\bm x]_{\Gamma'}$. Consequently, affine conjugation preserves discreteness and cocompactness.
\end{remark}
\begin{definition}
	An $n\times n$ matrix $\bm{U}=(u_{ij})$ is called a unimodular matrix if its elements $u_{ij}$ are all integers and satisfy $\det(\bm{U})=\pm 1$.
	The set of all $n\times n$ unimodular matrices forms a group under matrix multiplication, denoted by $\operatorname{GL}^{\pm 1}(n,\mathbb{Z})$.
\end{definition}

\begin{definition}[Arithmetic representation of an $n$-dimensional crystallographic group]\label{def:space_group}
	Let $\mathcal{G} \le \operatorname{Aff}(\mathbb{R}^n)$. $\mathcal{G}$ is said to be an arithmetic representation of an $n$-dimensional crystallographic group if it satisfies the following conditions.
	\begin{enumerate}
		\item For any $g = (\bm{R}_g, \bm{t}_g) \in \mathcal{G}$, $\bm{R}_g \in \operatorname{GL}^{\pm 1}(n,\mathbb{Z})$ and $\bm{t}_g \in \mathbb{Q}^n$.
		\item $\mathcal{G}$ contains the integer translation subgroup
		      \[
			      T = \{(\bm{I}_n, \bm{t}) : \bm{t} \in \mathbb{Z}^n\}.
		      \]
		\item The quotient group $\mathcal{G}/T$ is a finite group.
	\end{enumerate}
\end{definition}

\begin{remark}
	For any $g\in\mathcal{G}$ and $(\bm{I}_n,\bm{z})\in T$, we have
	\[
		g(\bm{I}_n,\bm{z})g^{-1}=(\bm{I}_n,\bm{R}_g\bm{z}).
	\]
	Since $\bm{R}_g\in\operatorname{GL}^{\pm 1}(n,\mathbb{Z})$, it follows that $\bm{R}_g\bm{z}\in\mathbb{Z}^n$, hence $T\triangleleft\mathcal{G}$, making $\mathcal{G}/T$ well-defined.

	From conditions 2 and 3, $\mathcal{G}$ is a union of finitely many $T$-cosets. For any $g\in\mathcal G$, we have
	\[
		gT=\{(\bm R_g,\bm t_g+\bm R_g\bm z):\bm z\in\mathbb Z^n\},
	\]
	thus each coset is discrete in $\operatorname{Aff}(\mathbb R^n)$. 
	
	To see that the finite union is discrete, let $K\coloneqq\mathcal G\cap\mathbb R^n\supseteq T$ denote the pure-translation subgroup of $\mathcal G$. Since $K/T$ embeds into the finite group $\mathcal G/T$, the quotient $K/T$ is finite; as $T=\mathbb Z^n$ is a full-rank lattice, $K$ is a finite union of $T$-cosets and hence a full-rank lattice as well. Moreover $R(K)=K$ for every linear part $R\in P$, and writing $\mathcal G=\bigcup_{R\in P}\{\bm R\}\times(\bm t_R+R(K))$, the finitely many fibers are pairwise separated in their linear-part coordinates, each fiber being discrete; hence $\mathcal G$ is discrete. 
	
	Furthermore, condition 3 implies that the point group $P = \{\bm{R}_g : g \in \mathcal{G}\}$ is finite.
\end{remark}

\begin{theorem}[Geometric-arithmetic equivalence of crystallographic groups]\label{thm:main_theorem}
	Let $\Gamma\leq\operatorname{Isom}(\mathbb{R}^n)$. Then the following are equivalent.
	\begin{enumerate}
		\item $\Gamma$ is an $n$-dimensional crystallographic group.
		\item There exists $\alpha\in\operatorname{Aff}(\mathbb{R}^n)$ such that $\mathcal{G}=\alpha\Gamma\alpha^{-1}$ satisfies all conditions in \cref{def:space_group}.
	\end{enumerate}
	In particular, $\Gamma\cong\mathcal{G}$. Therefore, an $n$-dimensional crystallographic group is isomorphic to its arithmetic representation in a group-theoretic sense.
\end{theorem}

\begin{remark}
	The isomorphism in \cref{thm:main_theorem} is given by the conjugation map $\varphi_\alpha: \Gamma\to\mathcal{G}$, where $\varphi_\alpha(g)=\alpha g\alpha^{-1}$.
\end{remark}

\subsection{Proof of the geometric-arithmetic equivalence theorem}

We now prove \cref{thm:main_theorem}. The proof is divided into two parts.

\textbf{(1)$\Rightarrow$(2)}. Let $\Gamma \subset \operatorname{Isom}(\mathbb{R}^n)$ be a crystallographic group. By Bieberbach's first theorem, the translation subgroup $L = \Gamma \cap \mathbb{R}^n$ of $\Gamma$ is a full-rank lattice.

\begin{lemma}\label{lma:2.1}
	Let $L \subset \mathbb{R}^n$ be a full-rank lattice in $\mathbb{R}^n$. Then there exists an invertible linear transformation $\bm{A} \in \operatorname{GL}(n,\mathbb{R})$ that maps $L$ to the standard integer lattice, i.e.,
	\[
		\bm{A} L = \mathbb{Z}^n.
	\]
\end{lemma}

By \cref{lma:2.1}, we can choose $\bm{A} \in \operatorname{GL}(n,\mathbb{R})$ such that $\bm{A} L = \mathbb{Z}^n$. Denote
\[
	\mathcal{G}_0 \coloneqq \bm{A} \Gamma \bm{A}^{-1} \subset \operatorname{Aff}(\mathbb{R}^n),
	\qquad
	T \coloneqq \{(\bm{I}_n,\bm{z}) : \bm{z}\in\mathbb{Z}^n\}.
\]
Then $T$ is precisely the pure translation subgroup of $\mathcal{G}_0$. Since $L\triangleleft \Gamma$, it follows that $T\triangleleft \mathcal{G}_0$. For any $g=(\bm{R}_g,\bm{t}_g)\in\mathcal{G}_0$ and any $\bm{z}\in\mathbb{Z}^n$, we have
\[
	g(\bm{I}_n,\bm{z})g^{-1}=(\bm{I}_n,\bm{R}_g\bm{z})\in T.
\]
Thus, $\bm{R}_g\mathbb{Z}^n\subseteq\mathbb{Z}^n$. Applying the same argument to $g^{-1}=(\bm{R}_g^{-1},-\bm{R}_g^{-1}\bm{t}_g)\in\mathcal{G}_0$ yields
\[
	\bm{R}_g^{-1}\mathbb{Z}^n\subseteq\mathbb{Z}^n.
\]
Therefore,
\[
	\mathbb{Z}^n=\bm{R}_g\bm{R}_g^{-1}\mathbb{Z}^n\subseteq \bm{R}_g\mathbb{Z}^n\subseteq\mathbb{Z}^n,
\]
which implies $\bm{R}_g\mathbb{Z}^n=\mathbb{Z}^n$. Hence, $\bm{R}_g$ is an automorphism of the additive group $\mathbb{Z}^n$, corresponding to a unimodular integer matrix under the standard basis, i.e., $\bm{R}_g\in\operatorname{GL}^{\pm 1}(n,\mathbb{Z})$. Denote
\[
	P\coloneqq\{\bm{R}_g:g\in\mathcal{G}_0\}.
\]
Since $P=\bm{A}\,r(\Gamma)\,\bm{A}^{-1}$, and $r(\Gamma)$ is a finite group by Bieberbach's first theorem, $P$ is also a finite group. Furthermore, the kernel of the map
\[
	\pi:\mathcal{G}_0\to P,\qquad (\bm{R},\bm{t})\mapsto \bm{R}
\]
is exactly $T$, thus $\mathcal{G}_0/T\cong P$ is a finite group.

It remains to show that, after an affine conjugation, the translation parts can all be reduced to rational vectors. For each $\bm{R}\in P$, choose a representative
\[
	\gamma_{\bm{R}}=(\bm{R},\bm{t}_{\bm{R}})\in\mathcal{G}_0,
\]
with the convention that $\gamma_{\bm{I}_n}=e$. Then any $g\in\mathcal{G}_0$ can be uniquely written as
\[
	g=(\bm{I}_n,\bm{z})\gamma_{\bm{R}},\qquad \bm{z}\in\mathbb{Z}^n,\ \bm{R}\in P.
\]
For any $\bm{R},\bm{S}\in P$, since $T=\ker\pi$, there exists a unique $\bm{z}(\bm{R},\bm{S})\in\mathbb{Z}^n$ such that
\[
	\gamma_{\bm{R}}\gamma_{\bm{S}}=(\bm{I}_n,\bm{z}(\bm{R},\bm{S}))\gamma_{\bm{R}\bm{S}}.
\]
Comparing the translation parts yields
\[
	\bm{t}_{\bm{R}}+\bm{R}\bm{t}_{\bm{S}}-\bm{t}_{\bm{R}\bm{S}}=\bm{z}(\bm{R},\bm{S})\in\mathbb{Z}^n.
\]

Now, we conjugate by a translation to simultaneously transform the translation parts of these representatives into rational vectors. Denote
\[
	\bm{u}\coloneqq -\frac{1}{|P|}\sum_{\bm{S}\in P}\bm{t}_{\bm{S}}\in\mathbb{R}^n,
	\qquad
	\tau\coloneqq(\bm{I}_n,\bm{u}).
\]
For a fixed $\bm{R}\in P$, summing the above equation over all $\bm{S}\in P$, and noting that as $\bm{S}$ runs through $P$, $\bm{R}\bm{S}$ also runs through $P$, we obtain
\[
	\sum_{\bm{S}\in P}\bm{z}(\bm{R},\bm{S})
	= |P|\bm{t}_{\bm{R}} + \bm{R}\sum_{\bm{S}\in P}\bm{t}_{\bm{S}} - \sum_{\bm{S}\in P}\bm{t}_{\bm{S}}.
\]
Rearranging the terms gives
\[
	\bm{t}_{\bm{R}}+\bm{u}-\bm{R}\bm{u}
	= \frac{1}{|P|}\sum_{\bm{S}\in P}\bm{z}(\bm{R},\bm{S})\in\mathbb{Q}^n.
\]
Hence, for each $\bm{R}\in P$,
\[
	\tau\gamma_{\bm{R}}\tau^{-1}
	= (\bm{R},\bm{t}_{\bm{R}}+\bm{u}-\bm{R}\bm{u})
	\in \operatorname{GL}^{\pm 1}(n,\mathbb{Z})\ltimes\mathbb{Q}^n.
\]

Let
\[
	\mathcal{G}\coloneqq \tau\mathcal{G}_0\tau^{-1}.
\]
Since $\tau$ is a pure translation, it commutes with all integer translations, which implies $\tau T\tau^{-1}=T$. Moreover, the linear part of $\mathcal{G}$ is identical to that of $\mathcal{G}_0$, so it remains $P$. Any $g\in\mathcal{G}$ can be uniquely written as
\[
	g=(\bm{I}_n,\bm{z})\,\tau\gamma_{\bm{R}}\tau^{-1},
	\qquad \bm{z}\in\mathbb{Z}^n,\ \bm{R}\in P.
\]
Therefore, for any $g=(\bm{R},\bm{t})\in\mathcal{G}$, we have $\bm{R}\in\operatorname{GL}^{\pm 1}(n,\mathbb{Z})$ and $\bm{t}\in\mathbb{Q}^n$. Additionally, since $T\subset\mathcal{G}$ and
\[
	\mathcal{G}/T\cong \mathcal{G}_0/T\cong P
\]
is a finite group, $\mathcal{G}$ satisfies all the conditions in \cref{def:space_group}.

Finally, by construction,
\[
	\mathcal{G}=\tau\mathcal{G}_0\tau^{-1}=(\tau\bm{A})\Gamma(\tau\bm{A})^{-1}.
\]
Setting $\alpha\coloneqq\tau\bm{A}\in\operatorname{Aff}(\mathbb{R}^n)$, we obtain
\[
	\mathcal{G}=\alpha\Gamma\alpha^{-1},
\]
which establishes (1)$\Rightarrow$(2).

\textbf{(2)$\Rightarrow$(1)}. Suppose there exists $\alpha\in\operatorname{Aff}(\mathbb{R}^n)$ such that $\mathcal{G}=\alpha\Gamma\alpha^{-1}$ satisfies \cref{def:space_group}. Denote
\[
	T=\{(\bm{I}_n,\bm{k}):\bm{k}\in\mathbb{Z}^n\}.
\]
By definition, $T\triangleleft\mathcal{G}$ and $\mathcal{G}/T$ is finite. Thus, $\mathcal{G}$ can be written as the union of finitely many $T$-cosets,
\[
	\mathcal{G}=\bigcup_{i=1}^m g_iT,
\]
where $g_i=(\bm{R}_i,\bm{t}_i)\in\mathcal{G}$. Since $\bm{R}_i\in\operatorname{GL}^{\pm 1}(n,\mathbb{Z})$, every element in each coset takes the form
\[
	g_i(\bm{I}_n,\bm{k})=(\bm{R}_i,\bm{t}_i+\bm{R}_i\bm{k}),
	\qquad \bm{k}\in\mathbb{Z}^n,
\]
and since $\bm{R}_i\in\operatorname{GL}^{\pm 1}(n,\mathbb{Z})$ implies $\bm{R}_i\mathbb{Z}^n=\mathbb{Z}^n$, we have
\[
	g_iT=\{(\bm{R}_i,\bm{t}_i+\bm{z}) : \bm{z}\in\mathbb{Z}^n\}.
\]
This indicates that the linear part of each coset $g_iT$ is fixed to $\bm{R}_i$, while the translation part is a translate of the integer lattice $\mathbb{Z}^n$. The coordinate map $g=(\bm R_g,\bm t_g)\mapsto(\bm R_g,\bm t_g)$ identifies the topological space $\operatorname{Aff}(\mathbb{R}^n)$ with
\[
    \operatorname{GL}(n,\mathbb{R})\times\mathbb{R}^n.
\]
The group law is that of the semidirect product $\operatorname{GL}(n,\mathbb{R})\ltimes\mathbb{R}^n$. Under this coordinate map, each coset $g_iT$ corresponds to
\[
    \{\bm{R}_i\}\times(\bm{t}_i+\mathbb{Z}^n),
\]
which is a closed discrete subset of the product space. A finite union of closed discrete subsets is again closed and discrete. Therefore, $\mathcal{G}$ is a discrete subgroup of $\operatorname{Aff}(\mathbb{R}^n)$.

Next, let $\Omega=[0,1]^n$. For any $\bm{x}\in\mathbb{R}^n$, there always exists $\bm{k}\in\mathbb{Z}^n$ such that $\bm{x}-\bm{k}\in\Omega$, hence $T\cdot\Omega=\mathbb{R}^n$. Since $T\subset\mathcal{G}$, we also have $\mathcal{G}\cdot\Omega=\mathbb{R}^n$. Let
\[
	q:\mathbb{R}^n\to \mathbb{R}^n/\mathcal{G}
\]
be the quotient map. From $\mathcal{G}\cdot\Omega=\mathbb{R}^n$, it follows that $q(\Omega)=\mathbb{R}^n/\mathcal{G}$. Since $q$ is continuous and $\Omega$ is a compact set, its image $q(\Omega)$ is also compact, meaning that $\mathbb{R}^n/\mathcal{G}$ is compact. Thus, $\mathcal{G}$ is cocompact.

Because $\mathcal{G}=\alpha\Gamma\alpha^{-1}$ and affine conjugation preserves compactness, $\mathbb{R}^n/\Gamma$ is also compact. On the other hand, the conjugation map $g\mapsto \alpha g\alpha^{-1}$ does not alter the discreteness of subgroups; therefore, the discreteness of $\mathcal{G}$ in $\operatorname{Aff}(\mathbb{R}^n)$ implies that $\Gamma$ is also discrete in $\operatorname{Aff}(\mathbb{R}^n)$. Noting further that $\Gamma\leq\operatorname{Isom}(\mathbb{R}^n)$, $\Gamma$ remains discrete in $\operatorname{Isom}(\mathbb{R}^n)$. Therefore, $\Gamma$ satisfies both discreteness and cocompactness, making it an $n$-dimensional crystallographic group. This completes the proof of (2)$\Rightarrow$(1).

\subsection{Arithmetic representation in conventional coordinate systems}

\begin{definition}[Conventional basis and conventional lattice]
	Let $\Gamma$ be an $n$-dimensional crystallographic group, and denote its translation lattice by $L = \Gamma \cap \mathbb{R}^n$. Let $(\bm{a}_1, \dots, \bm{a}_n)$ be a basis for $\mathbb{R}^n$, and let the lattice generated by it be
	\[
		L_{\mathrm{conv}} \coloneqq \left\{ \sum_{i=1}^n m_i \bm{a}_i : m_i \in \mathbb{Z} \right\} \subset \mathbb{R}^n.
	\]
	This basis is termed a conventional basis, and $L_{\mathrm{conv}}$ is called the corresponding conventional lattice, if the following conditions are satisfied.
	\begin{enumerate}
		\item Let $P \coloneqq \{\bm{R}_\gamma : \gamma \in \Gamma\}$ be the point group of $\Gamma$. Then for any $\bm{R} \in P$, we have $\bm{R}(L_{\mathrm{conv}}) = L_{\mathrm{conv}}$.
		\item $L_{\mathrm{conv}} \subseteq L$, and the index $[L : L_{\mathrm{conv}}] < \infty$.
	\end{enumerate}
\end{definition}

\begin{remark}
	The conventional basis corresponds to the coordinate basis chosen for the conventional cell in crystallography, where the conventional basis vectors are exactly the basis vectors of this cell, typically selected according to the standard conventions for Bravais lattices. In the classical two- and three-dimensional setting, if the Bravais lattice is centered, then the volume of the conventional cell is strictly greater than that of the primitive cell; equivalently, $L_{\mathrm{conv}}$ is a proper finite-index sublattice of $L$.
\end{remark}

\begin{proposition}\label{prop:arith-conventional}
	Let $\Gamma$ be an $n$-dimensional crystallographic group. If its coordinate representation under some basis satisfies the conditions of \cref{def:space_group}, then its coordinate representation with respect to any conventional basis also satisfies \cref{def:space_group}.
\end{proposition}

\begin{remark}
	Note that in the conventional basis, the integer translation subgroup $T$ corresponds to the conventional lattice $L_{\mathrm{conv}}$, which may be a proper sublattice of the full translation lattice. The arithmetic representation requires $T\subseteq\mathcal G$ but does not require equality with the full translation subgroup. For example, in the body-centred cubic case, the body-centring translation $(1/2,1/2,1/2)$ belongs to the full translation lattice but not to $T$. The arithmetic representation remains valid because translation parts are allowed to lie in $\mathbb Q^n$.
\end{remark}

\begin{proof}
	First, choose a basis such that the coordinate representation of $\Gamma$ satisfies \cref{def:space_group}, and denote its pure translation subgroup by
	\[
		L\coloneqq\Gamma\cap\mathbb{R}^n.
	\]
	Since $\bm t\in\mathbb{Q}^n$ for any $(\bm I_n,\bm t)\in\Gamma$, we have $L\subset\mathbb{Q}^n$. Take a $\mathbb{Z}$-basis $\{\bm\ell_1,\dots,\bm\ell_n\}$ of $L$, and let
	\[
		\bm A\coloneqq[\bm\ell_1\ \cdots\ \bm\ell_n]\in\operatorname{GL}(n,\mathbb{Q}),
		\qquad
		\bm B_0\coloneqq\bm A^{-1}.
	\]
	Then $L=\bm A\mathbb{Z}^n$. Conjugating $\Gamma$ by $(\bm B_0,\bm 0)$, we denote the resulting group by
	\[
		\Gamma_0\coloneqq(\bm B_0,\bm 0)\Gamma(\bm B_0,\bm 0)^{-1}.
	\]
	Its pure translation subgroup is denoted by
	\[
		L_0\coloneqq\Gamma_0\cap\mathbb R^n.
	\]
	For any $g=(\bm R_g,\bm t_g)\in\Gamma$, its corresponding element in $\Gamma_0$ is
	\[
		g_0=(\bm R_{g,0},\bm t_{g,0})
		=\bigl(\bm B_0\bm R_g\bm B_0^{-1},\,\bm B_0\bm t_g\bigr).
	\]
	Moreover, for any $\bm \ell\in L$, since $(\bm I_n,\bm \ell)\in\Gamma$, we have
	\[
		g(\bm I_n,\bm \ell)g^{-1}=(\bm I_n,\bm R_g\bm \ell)\in\Gamma\cap\mathbb R^n=L.
	\]
	Thus $\bm R_g(L)\subseteq L$. Applying the same argument to $g^{-1}$ yields $\bm R_g^{-1}(L)\subseteq L$, implying $\bm R_g(L)=L$. Consequently,
	\[
		\bm R_{g,0}\mathbb Z^n
		=\bm B_0\bm R_g(L)
		=\bm B_0L
		=\mathbb Z^n,
	\]
	hence $\bm R_{g,0}\in\operatorname{GL}^{\pm 1}(n,\mathbb Z)$. Also, since $\bm B_0\in\operatorname{GL}(n,\mathbb Q)$ and $\bm t_g\in\mathbb Q^n$, it follows that $\bm t_{g,0}\in\mathbb Q^n$. In addition,
	\[
		L_0=\bm B_0L=\mathbb Z^n.
	\]
	From the remark following \cref{def:space_group}, the point group $P=\{\bm R_g:g\in\Gamma\}$ is a finite group, so $P_0\coloneqq\bm B_0P\bm B_0^{-1}$ is also finite. The kernel of the linear part map
	\[
		r_0:\Gamma_0\to P_0,
		\qquad
		(\bm R,\bm t)\mapsto\bm R
	\]
	is exactly $L_0$, and thus $\Gamma_0/L_0\cong P_0$ is a finite group. This shows that $\Gamma_0$ still satisfies \cref{def:space_group}. Therefore, it suffices to proceed with the remainder of the proof in this intermediate coordinate system.

	Now we introduce the conventional basis in the coordinate system of $\Gamma_0$. Let the coordinate column vectors of the conventional basis vectors in this system be $\bm a_1,\dots,\bm a_n$, and let
	\[
		\bm C\coloneqq[\bm a_1\ \cdots\ \bm a_n].
	\]
	By definition, the conventional lattice in this coordinate system is expressed as
	\[
		\Lambda\coloneqq L_{\mathrm{conv}}=\bm C\mathbb Z^n.
	\]
	Since $L_0=\mathbb Z^n$ here, and the conventional basis satisfies $L_{\mathrm{conv}}\subseteq L_0$ and $[L_0:L_{\mathrm{conv}}]<\infty$, we have
	\[
		\Lambda\subseteq\mathbb Z^n,
		\qquad
		[\mathbb Z^n:\Lambda]<\infty.
	\]
	Taking
	\[
		\bm B\coloneqq\bm C^{-1}\in\operatorname{GL}(n,\mathbb Q),
	\]
	we get $\bm B\Lambda=\mathbb Z^n$. Conjugating $\Gamma_0$ by $(\bm B,\bm 0)$ yields the group representation under the conventional basis
	\[
		\Gamma_{\mathrm{conv}}\coloneqq(\bm B,\bm 0)\Gamma_0(\bm B,\bm 0)^{-1}.
	\]
	For any $g_0=(\bm R_{g,0},\bm t_{g,0})\in\Gamma_0$, its counterpart under the conventional basis is
	\[
		g_{\mathrm{conv}}=(\bm R_{g,\mathrm{conv}},\bm t_{g,\mathrm{conv}})
		=\bigl(\bm B\bm R_{g,0}\bm B^{-1},\,\bm B\bm t_{g,0}\bigr).
	\]
	We now verify that $\Gamma_{\mathrm{conv}}$ satisfies \cref{def:space_group}.

	First, by the definition of the conventional basis, $\bm R_{g,0}(\Lambda)=\Lambda$ for any $g_0\in\Gamma_0$. Thus
	\[
		\bm R_{g,\mathrm{conv}}\mathbb Z^n
		=(\bm B\bm R_{g,0}\bm B^{-1})\mathbb Z^n
		=\bm B\bm R_{g,0}(\Lambda)
		=\bm B\Lambda
		=\mathbb Z^n,
	\]
	so $\bm R_{g,\mathrm{conv}}\in\operatorname{GL}^{\pm 1}(n,\mathbb Z)$.

	Second, since $\bm B\in\operatorname{GL}(n,\mathbb Q)$ and $\bm t_{g,0}\in\mathbb Q^n$, we have
	\[
		\bm t_{g,\mathrm{conv}}=\bm B\bm t_{g,0}\in\mathbb Q^n.
	\]

	Third, let
	\[
		T_\Lambda\coloneqq\{(\bm I_n,\bm v):\bm v\in\Lambda\}.
	\]
	Since $\Lambda\subseteq L_0$ and $L_0$ is the pure translation subgroup of $\Gamma_0$, we have $T_\Lambda\subseteq\Gamma_0$. Moreover, for any $\bm z\in\mathbb Z^n$, setting $\bm v=\bm C\bm z\in\Lambda$, we obtain
	\[
		(\bm B,\bm 0)(\bm I_n,\bm v)(\bm B^{-1},\bm 0)
		=(\bm I_n,\bm B\bm v)
		=(\bm I_n,\bm z).
	\]
	Thus
	\[
		T\coloneqq\{(\bm I_n,\bm z):\bm z\in\mathbb Z^n\}\subseteq\Gamma_{\mathrm{conv}}.
	\]

	Finally, from $[\mathbb Z^n:\Lambda]<\infty$ and the finiteness of $\Gamma_0/L_0$, we obtain
	\[
		[\Gamma_0:T_\Lambda]=[\Gamma_0:L_0]\,[L_0:T_\Lambda]<\infty.
	\]
	Since the conjugation map sends $T_\Lambda$ to $T$, it follows that
	\[
		[\Gamma_{\mathrm{conv}}:T]=[\Gamma_0:T_\Lambda]<\infty.
	\]
	Hence $\Gamma_{\mathrm{conv}}/T$ is a finite group.

	In summary, the linear parts of $\Gamma_{\mathrm{conv}}$ belong to $\operatorname{GL}^{\pm 1}(n,\mathbb Z)$, the translation parts belong to $\mathbb Q^n$, and $\Gamma_{\mathrm{conv}}$ contains the standard integer translation subgroup while satisfying that $\Gamma_{\mathrm{conv}}/T$ is finite. 	Therefore, the coordinate representation of $\Gamma$ under a conventional basis still satisfies \cref{def:space_group}.
\end{proof}

\paragraph{Convention for the rest of the paper}
Based on the geometric--arithmetic equivalence theorem and the preservation of the arithmetic representation under conventional bases, we work with arithmetic crystallographic groups for the remainder of this paper. This involves no loss of generality: every isomorphism class of $n$-dimensional crystallographic groups admits an arithmetic representative, and affine conjugacy in this framework corresponds to abstract group isomorphism.

\subsection{Quotient group structure and discrete computational framework}

For compatibility with the standard discrete Fourier transform, the computational framework employs a coordinate system defined by the basis vectors of the conventional crystal cell. Unless otherwise specified, all subsequent matrices and vectors are expressed in this coordinate system.

By the remark following \cref{def:space_group}, we know that $T\triangleleft\mathcal{G}$. We thus define the quotient group
\[
	G\coloneqq\mathcal{G}/T,
\]
whose elements are denoted by $\bar g=gT$.

\begin{remark}
	The relationship between the quotient group $G$ and the point group $P$ depends on the integer translation subgroup $T$, whose precise form is dictated by the choice of the coordinate basis. Define the map $\phi:G\to P$ by
	\[
		\phi(\bar g)=\bm{R}_g.
	\]
	This map is well-defined and surjective. Its kernel consists of cosets whose linear part is the identity matrix,
	\[
		\ker(\phi)=\{\bar g\in G:\bm{R}_g=\bm{I}_n\}\cong T_{\mathcal{G}}/T,
	\]
	where $T_{\mathcal{G}}\coloneqq\{(\bm{I}_n,\bm{t})\in\mathcal{G}\}$ is the pure translation subgroup of $\mathcal{G}$.

	\textbf{Primitive basis case.} Under a primitive basis, the full translation lattice has integer coordinates, so $T_{\mathcal G}=T$. Thus, $\ker(\phi)$ is trivial and $G\cong P$.

	\textbf{Conventional basis case.} Under a conventional basis, $T$ is given by the integer translations generated by the conventional basis vectors. In the classical two- and three-dimensional setting, a centered Bravais lattice satisfies $[L:L_{\mathrm{conv}}] > 1$. More generally, the same conclusion holds whenever the chosen conventional lattice is a proper finite-index sublattice of $L$. In either case, there exist pure translation elements whose coordinates are non-integral in the conventional basis, leading to $T \subsetneq T_{\mathcal{G}}$. Therefore, $\ker(\phi)\cong T_{\mathcal{G}}/T$ is a nontrivial finite group.
\end{remark}

\section{Composition series and reconstruction of symmetry orbits}
\label{sec:solv}
Let us first review several basic concepts in group theory.

\begin{definition}[Normal Subgroup]
	A subgroup $N$ of a group $G$ is called a normal subgroup, denoted as $N \trianglelefteq G$, if $gN=Ng$ for all $g\in G$.
\end{definition}

\begin{definition}[Simple Group]
	A nontrivial group $G$ is called a simple group if its only normal subgroups are $\{e\}$ and $G$.
\end{definition}

\begin{definition}[Index]
	The index of a subgroup $H$ in a group $G$, denoted by $[G:H]$, is defined as the number of distinct cosets of $H$ in $G$.
\end{definition}

\begin{theorem}[Lagrange's Theorem]
	For any subgroup $H$ of a finite group $G$, we have $|G|=|H|[G:H]$.
\end{theorem}

For a finite group action, the map from the left cosets of a point stabilizer to the corresponding orbit is a bijection. Together with Lagrange's theorem, this gives the orbit--stabilizer formula, which will be used below to determine the number of distinct points recovered from a representative.

\begin{definition}[Subnormal Series]
	A subnormal series of a group $G$ is a subgroup chain $G=G_0\rhd G_1\rhd \cdots\rhd G_n$, where $G_{j+1}$ is a normal subgroup of $G_j$ \textup{(}$0\leq j<n$\textup{)}. The quotient groups $G_j/G_{j+1}$ are called the factors. If all $G_j$ are normal subgroups of $G$, it is called a normal series.
\end{definition}

\begin{definition}[Composition Series]
	A composition series is a subnormal series $G=G_0\rhd G_1\rhd \cdots\rhd G_n=\{ e \}$ in which every factor $G_j/G_{j+1}$ is a simple group.
\end{definition}

\begin{definition}[Solvable Series]
	A solvable series is a subnormal series $G=G_0\rhd G_1\rhd \cdots\rhd G_n=\{ e \}$ in which every factor $G_j/G_{j+1}$ is an abelian group. If $G$ possesses such a series, $G$ is called a solvable group.
\end{definition}

\begin{theorem}
	A finite group $G$ is a solvable group if and only if all factors of its composition series are cyclic groups of prime order.
\end{theorem}

\begin{proof}
If $G$ is solvable, by definition it has a subnormal series
\[
    G=G_0\triangleright G_1\triangleright\cdots
    \triangleright G_n=\{e\}
\]
with abelian factors. Since $G$ is finite, each factor admits a composition series. Taking inverse images of these series refines the original series into a composition series of $G$. Its factors are finite abelian simple groups, and hence cyclic groups of prime order. The Jordan--H\"older theorem then shows that every composition series of $G$ has the same factors up to isomorphism and order.

Conversely, if a composition series of $G$ has cyclic prime-order factors, all its factors are abelian. This series is therefore a solvable series, and $G$ is solvable.
\end{proof}

We write $\mathrm{Sym}_m$ for the symmetric group on $m$ symbols,
$\mathrm{Alt}_m$ for its alternating subgroup, and
$\mathrm{Cyc}_m\cong\mathbb Z/m\mathbb Z$ for the cyclic group of order $m$.
The Klein four-group is denoted by
$V_4\cong\mathrm{Cyc}_2\times\mathrm{Cyc}_2$.

\begin{theorem}\label{thm:three_dimensional_solvability}
    Every three-dimensional crystallographic group is solvable. In particular,
    its finite symmetry quotient $G=\mathcal G/T$ is solvable.
\end{theorem}

\begin{proof}
    We first consider the point group $P$ of $\mathcal G$.
    By the classification of three-dimensional crystallographic point groups,
    $P$ is isomorphic to a subgroup of either the full octahedral point group
    $m\bar3m$ or the full hexagonal point group $6/mmm$
    \cite{Aroyo2016ITC_A}. These groups satisfy
    \[
        m\bar3m\cong\mathrm{Sym}_4\times\mathrm{Cyc}_2,
        \qquad
        6/mmm\cong\mathrm{Sym}_3\times\mathrm{Cyc}_2\times\mathrm{Cyc}_2.
    \]
    For the first isomorphism, the orientation-preserving symmetries of the
    cube act faithfully on its four body diagonals. Since this group has
    order $24$, the action identifies it with $\mathrm{Sym}_4$. Central inversion supplies the additional
    $\mathrm{Cyc}_2$ factor. For the second isomorphism, the
    orientation-preserving hexagonal symmetries form the dihedral group of
    order $12$. If $r$ is a sixfold rotation and $s$ is a twofold rotation
    about an axis in the basal plane, then
    $\langle r^2,s\rangle\cong\mathrm{Sym}_3$ and
    $\langle r^3\rangle\cong\mathrm{Cyc}_2$ form a direct product;
    central inversion supplies the remaining $\mathrm{Cyc}_2$ factor.

    The series
    \[
        \{e\}\triangleleft\mathrm{Cyc}_2\triangleleft V_4
        \triangleleft\mathrm{Alt}_4\triangleleft\mathrm{Sym}_4,
        \qquad
        \{e\}\triangleleft\mathrm{Alt}_3\triangleleft\mathrm{Sym}_3
    \]
    have cyclic prime-order factors. Thus $\mathrm{Sym}_4$ and
    $\mathrm{Sym}_3$ are solvable. Since solvability is preserved under
    direct products and passage to subgroups, $P$ is solvable.

    Let $T_{\mathcal G}$ denote the full translation subgroup of
    $\mathcal G$. The linear-part homomorphism gives the exact sequence
    \[
        1\longrightarrow T_{\mathcal G}\longrightarrow\mathcal G
        \longrightarrow P\longrightarrow1,
    \]
    where $T_{\mathcal G}\cong\mathbb Z^3$ is abelian. Since $P$ is
    solvable, some term of the derived series of $\mathcal G$ lies in
    $T_{\mathcal G}$, and its next derived subgroup is trivial.
    Hence $\mathcal G$ is solvable. Its quotient $\mathcal G/T$ is
    therefore also solvable. Here $T$ is the integer translation subgroup
    selected in the arithmetic representation; it may be properly
    contained in $T_{\mathcal G}$ when conventional centred coordinates
    are used.
\end{proof}

Let $\mathcal G$ be an $n$-dimensional crystallographic group, and let $T$ be the translation subgroup fixed in the preceding section. Although $\mathcal G$ is generally infinite, the quotient $G=\mathcal G/T$ is finite. Hence, the following standard result applies to $G$:

\begin{theorem}\cite[Theorem 8.4]{Hungerford1974Algebra}
	Every finite group has a composition series.
\end{theorem}

Thus, $G$ admits a finite hierarchy
\[
	G=G_0\triangleright G_1\triangleright\cdots\triangleright G_m=\{e\},
\]
whose factors $G_j/G_{j+1}$ are finite simple groups. Each factor has a generating set of cardinality at most two:

\begin{theorem}
	Every finite simple group can be generated by at most two elements.
\end{theorem}

For every nonabelian finite simple group, \cite[Theorem 1.1(1)]{GuralnickMalle2012} gives a generating triple whose product is the identity, so its first two elements generate the group. Every abelian finite simple group is cyclic of prime order and hence has one generator.

The composition series and the generating sets of its factors provide a
hierarchical construction of the symmetry operations. The following theorem
connects this group-theoretic structure with orbit reconstruction, which is
used in the subsequent symmetry reductions.

\begin{theorem}[Hierarchical reconstruction of finite group actions]
\label{thm:hierarchical_group_action}
    Let $G$ be a finite group acting on a set $X$, and let
    \[
        \{e\}=H_0\triangleleft H_1\triangleleft\cdots
        \triangleleft H_m=G
    \]
    be a composition series. For each $i$, choose generators
    $\bar a_{i,1},\ldots,\bar a_{i,d_i}$ of $Q_i=H_i/H_{i-1}$,
    where $d_i\leq2$, and choose lifts
    $a_{i,1},\ldots,a_{i,d_i}\in H_i$. Then
    \[
        H_i=\langle H_{i-1},a_{i,1},\ldots,a_{i,d_i}\rangle.
    \]
    Words in these lifts and their inverses give a set $R_i$ of left coset
    representatives of $H_{i-1}$ in $H_i$, with $|R_i|=|Q_i|$.
    Every $g\in G$ has a unique representation
    \begin{equation}\label{eq:hierarchical_group_normal_form}
        g=r_mr_{m-1}\cdots r_1,\qquad r_i\in R_i,
        \qquad |G|=\prod_{i=1}^m|Q_i|.
    \end{equation}
    For $x\in X$, define
    \begin{equation}\label{eq:hierarchical_orbit_reconstruction}
        O_0(x)=\{x\},\qquad
        O_i(x)=\bigcup_{r\in R_i}r\cdot O_{i-1}(x).
    \end{equation}
    Then $O_i(x)=H_i\cdot x$ for every $i$, and in particular
    $O_m(x)=G\cdot x$. If $G$ is solvable, each factor $Q_i$ is cyclic
    of prime order $p_i$, and one may take
    \begin{equation}\label{eq:cyclic_factor_transversal}
        R_i=\{e,a_i,\ldots,a_i^{p_i-1}\},
    \end{equation}
    where $a_iH_{i-1}$ generates $Q_i$.
\end{theorem}

\begin{proof}
    Let $\pi_i:H_i\to Q_i$ be the quotient homomorphism. The subgroup
    generated by $H_{i-1}$ and the chosen lifts contains $\ker\pi_i$
    and maps onto $Q_i$, so it equals $H_i$. For each element of $Q_i$,
    choose one word in its generators and their inverses, and evaluate
    this word in the lifts. This selects one representative of each left
    coset and gives the disjoint decomposition
    \[
        H_i=\bigsqcup_{r\in R_i}rH_{i-1}.
    \]
    Applying these decompositions successively proves the existence and
    uniqueness of the representation in
    \eqref{eq:hierarchical_group_normal_form}. Lagrange's theorem gives
    $|H_i|=|Q_i||H_{i-1}|$, and therefore the stated order formula.
    Applying the same decompositions to $x$ yields
    \[
        H_i\cdot x
        =\bigcup_{r\in R_i}r\cdot(H_{i-1}\cdot x).
    \]
    Induction now proves \eqref{eq:hierarchical_orbit_reconstruction}.
    Finally, if $G$ is solvable, its composition factors are cyclic of
    prime order. The cosets of the powers in
    \eqref{eq:cyclic_factor_transversal} are distinct and exhaust $Q_i$,
    which proves the last assertion.
\end{proof}

\begin{remark}\label{rem:factor_action_and_lifts}
    The factor $Q_i$ acts naturally on the set of $H_{i-1}$-orbits in $X$
    by
    \[
        (gH_{i-1})\cdot(H_{i-1}\cdot x)=H_{i-1}\cdot(g\cdot x).
    \]
    This action is well-defined because $H_{i-1}$ is normal in $H_i$.
    The lifts specify the corresponding action on individual points.
    No splitting of the group extension is required: $R_i$ need not be a
    subgroup, and a lift $a_i$ may have order larger than $p_i$;
    only $a_i^{p_i}\in H_{i-1}$ is required. Distinct group operations
    can also produce the same point. The number of distinct points
    recovered from $x$ is
    \[
        |G\cdot x|=\frac{|G|}{|G_x|},\qquad
        G_x=\{g\in G:g\cdot x=x\}.
    \]
    For $G$-invariant data, one value at $x$ determines the values on
    this entire orbit.
\end{remark}

\begin{corollary}[One-generator reconstruction in three dimensions]
\label{cor:three_dimensional_reconstruction}
    Let $G=\mathcal G/T$ be the finite symmetry quotient of a
    three-dimensional crystallographic group. Every subgroup and every
    homomorphic image of $G$ admits a composition series with cyclic
    prime-order factors. Consequently, its operations and symmetry
    orbits can be reconstructed using one lifted generator at each
    composition step.
\end{corollary}

\begin{proof}
    By \cref{thm:three_dimensional_solvability}, $G$ is solvable, and
    solvability is preserved under passage to subgroups and homomorphic
    images. The conclusion follows from
    \cref{thm:hierarchical_group_action}. More explicitly, for a subgroup
    $K\leq G$ and a composition series $(H_i)$ of $G$, set
    $K_i=K\cap H_i$. Then $K_{i-1}\trianglelefteq K_i$, and the map
    $kK_{i-1}\mapsto kH_{i-1}$ embeds $K_i/K_{i-1}$ into
    $H_i/H_{i-1}$. Since the latter has prime order, each nontrivial
    factor of the intersection series has the same prime order.
    Removing repeated terms therefore gives a composition series of
    $K$ with one generator per factor.
\end{proof}

This construction supplies the symmetry operations and orbit representatives
needed by SFFT. A Fourier reduction additionally uses a compatible grid
decomposition and the reduction formulas developed in
\cref{sec:math,sec:algor}. Thus the composition steps construct the group
action, whereas the Fourier recursion is organized by the orbits and
stabilizers of compatible grid layers. These two hierarchies need not have
the same number of steps.

\subsection{Composition series in three-dimensional crystallographic groups}
\label{sec:Solvable-Series}

We next illustrate composition series in the seven crystal systems. In each example, $G$ denotes the finite quotient $\mathcal G/T$. Group elements are represented by affine pairs $(\bm R,\bm t)$ modulo the integer translations of the stated conventional cell. Lagrange's theorem determines the orders of successive subgroups.

\noindent\textbf{Triclinic system.}
In the standard primitive conventional cell, $T=T_{\mathcal G}$.
For space group No.~1 ($P1$), $G=\mathcal G/T$ is trivial.
For No.~2 ($P\bar1$), $G\cong\mathrm{Cyc}_2$, with composition series
$\{e\}\triangleleft G$.

\noindent \textbf{Monoclinic system.} Consider space group No. 12 ($C12/m1$) as an example.
\begin{itemize}
	\item Let $G=\mathcal G/T$ for $\mathcal G=C12/m1$, so that $|G|=8$.
	\item Define $\bm R_\alpha=\begin{bmatrix} -1&0&0\\0&-1&0\\0&0&-1\end{bmatrix}$, $\bm R_\beta=\begin{bmatrix} 1&0&0\\0&-1&0\\0&0&1\end{bmatrix}$, and $\bm R_e=\bm I_3$.
	\item Let $M_1=\{( \bm I_3,\bm 0),( \bm R_\alpha,\bm 0)\}$, $M_2=\{( \bm I_3,\bm 0),( \bm R_\beta,\bm 0)\}$, $M_3=\{( \bm I_3,\bm 0),( \bm I_3,(\frac{1}{2},\frac{1}{2},0))\}$.
	\item Set $H_0=\{e\}$ and $H_j=\langle M_1,\ldots,M_j\rangle$ for $1\leq j\leq3$. Then $|H_j|=2^j$ and $H_3=G$. Since each inclusion has index two, $H_{j-1}\trianglelefteq H_j$. Hence
	\[
		\{e\}=H_0\triangleleft H_1\triangleleft H_2\triangleleft H_3=G
	\]
	is a solvable series, with all factors isomorphic to $\mathrm{Cyc}_2$.
\end{itemize}

\noindent \textbf{Orthorhombic system.} Consider space group No. 69 ($Fmmm$) as an example.
\begin{itemize}
	\item Let $G=\mathcal G/T$ for $\mathcal G=Fmmm$, so that $|G|=32$.
	\item Define $\bm R_\alpha=\begin{bmatrix} -1&0&0\\0&1&0\\0&0&1\end{bmatrix}$, $\bm R_\beta=\begin{bmatrix} 1&0&0\\0&-1&0\\0&0&1\end{bmatrix}$, $\bm R_\omega=\begin{bmatrix} 1&0&0\\0&1&0\\0&0&-1\end{bmatrix}$, and $\bm R_e=\bm I_3$.
	\item Let $M_1=\{( \bm I_3,\bm 0),( \bm R_\alpha,\bm 0)\}$, $M_2=\{( \bm I_3,\bm 0),( \bm R_\beta,\bm 0)\}$, $M_3=\{( \bm I_3,\bm 0),( \bm R_\omega,\bm 0)\}$, $M_4=\{( \bm I_3,\bm 0),( \bm I_3,(0,\frac{1}{2},\frac{1}{2}))\}$, $M_5=\{( \bm I_3,\bm 0),( \bm I_3,(\frac{1}{2},0,\frac{1}{2}))\}$.
	\item Set $H_0=\{e\}$ and $H_j=\langle M_1,\ldots,M_j\rangle$ for $1\leq j\leq5$. Then $|H_j|=2^j$ and $H_5=G$. Each inclusion has index two, and therefore
	\[
		\{e\}=H_0\triangleleft H_1\triangleleft H_2\triangleleft H_3\triangleleft H_4\triangleleft H_5=G
	\]
	is a solvable series whose five factors are isomorphic to $\mathrm{Cyc}_2$.
\end{itemize}

\noindent \textbf{Tetragonal system.} Consider space group No. 123 ($P4/mmm$) as an example.
\begin{itemize}
	\item Let $G=\mathcal G/T$ for $\mathcal G=P4/mmm$, so that $|G|=16$.
	\item Define $\bm R_\alpha=\begin{bmatrix} 1&0&0\\0&1&0\\0&0&-1\end{bmatrix}$, $\bm R_\beta=\begin{bmatrix} 0&1&0\\1&0&0\\0&0&1\end{bmatrix}$, $\bm R_\omega=\begin{bmatrix} -1 &0&0\\0&-1&0\\0&0&1\end{bmatrix}$, $\bm R_\theta=\begin{bmatrix} -1&0&0\\0&1&0\\0&0&1\end{bmatrix}$, and $\bm R_e=\bm I_3$.
	\item Let $M_1=\{( \bm I_3,\bm 0),( \bm R_\alpha,\bm 0)\}$, $M_2=\{( \bm I_3,\bm 0),( \bm R_\beta,\bm 0)\}$, $M_3=\{( \bm I_3,\bm 0),( \bm R_\omega,\bm 0)\}$, $M_4=\{( \bm I_3,\bm 0),( \bm R_\theta,\bm 0)\}$.
	\item Set $H_0=\{e\}$ and $H_j=\langle M_1,\ldots,M_j\rangle$ for $1\leq j\leq4$. Direct calculation gives $|H_j|=2^j$ and $H_4=G$. Thus every inclusion has index two, so
	\[
		\{e\}=H_0\triangleleft H_1\triangleleft H_2\triangleleft H_3\triangleleft H_4=G
	\]
	is a composition series with factors $H_j/H_{j-1}\cong\mathrm{Cyc}_2$.
\end{itemize}

\noindent\textbf{Trigonal system.}
Consider space group No.~160 ($R3m$) in the hexagonal-axis, $R$-obverse setting \cite{Aroyo2016ITC_A}.
\begin{itemize}
    \item Let $T$ be the integer translations of the hexagonal conventional cell. Then $G=\mathcal G/T$ has order $18$.
    \item Define
    \[
        \bm R_r=\begin{bmatrix}0&-1&0\\1&-1&0\\0&0&1\end{bmatrix},
        \qquad
        \bm R_s=\begin{bmatrix}0&-1&0\\-1&0&0\\0&0&1\end{bmatrix},
        \qquad
        \bm c=\begin{bmatrix}2/3\\1/3\\1/3\end{bmatrix}.
    \]
    Let $r=(\bm R_r,\bm0)$, $s=(\bm R_s,\bm0)$, and
    $\tau=(\bm I_3,\bm c)$, with each pair understood modulo $T$.
    These elements satisfy $r^3=s^2=\tau^3=e$ and $srs^{-1}=r^{-1}$.
    \item Since
    \[
        \bm R_r\bm c-\bm c=(-1,0,0)^{\mathsf T},\qquad
        \bm R_s\bm c-\bm c=(-1,-1,0)^{\mathsf T},
    \]
    $\tau$ commutes with $r$ and $s$ modulo $T$.
    Thus $G\cong\mathrm{Cyc}_3\times\mathrm{Sym}_3$.
    \item Set $C=\langle\tau\rangle$ and $H=\langle\tau,r\rangle$.
    Both are normal in $G$, and
    \[
        \{e\}\triangleleft C\triangleleft H\triangleleft G
    \]
    is a composition series. The subgroup orders are $1,3,9,18$,
    and the factors are $\mathrm{Cyc}_3$, $\mathrm{Cyc}_3$, and
    $\mathrm{Cyc}_2$, respectively.
\end{itemize}

\noindent \textbf{Hexagonal system.} Consider space group No. 191 ($P6/mmm$) as an example.
\begin{itemize}
	\item Let $G=\mathcal G/T$ for $\mathcal G=P6/mmm$, so that $|G|=24$.
	\item Define $\bm R_\alpha=\begin{bmatrix} -1&0&0\\0&-1&0\\0&0&-1\end{bmatrix}$, $\bm R_\beta=\begin{bmatrix} 0&-1&0\\-1&0&0\\0&0&1\end{bmatrix}$, $\bm R_\omega=\begin{bmatrix} 0&1&0\\1&0&0\\0&0&1\end{bmatrix}$, $\bm R_\theta=\begin{bmatrix} 0&-1&0\\1&-1&0\\0&0&1\end{bmatrix}$, and $\bm R_e=\bm I_3$.
	\item Define
	\[
		H_1=\langle(\bm R_\theta,\bm0)\rangle,
		\quad H_2=\langle(\bm R_\theta,\bm0),(\bm R_\beta,\bm0)\rangle,
		\quad H_3=\langle(\bm R_\theta,\bm0),(\bm R_\beta,\bm0),(\bm R_\omega,\bm0)\rangle.
	\]
	These groups have orders $3$, $6$, and $12$, respectively, and $G=\langle H_3,(\bm R_\alpha,\bm0)\rangle$ has order $24$. Here $H_1\trianglelefteq H_2$, while the remaining inclusions have index two. Therefore,
	\[
		\{e\}\triangleleft H_1\triangleleft H_2\triangleleft H_3\triangleleft G
	\]
	is a composition series. In ascending order, its factors have orders $3,2,2,$ and $2$, and are therefore cyclic.
\end{itemize}

\noindent \textbf{Cubic system.} Consider space group No. 229 ($Im\bar{3}m$) as an example.
\begin{itemize}
	\item Let $G=\mathcal G/T$ for $\mathcal G=Im\bar{3}m$, so that $|G|=96$.
	\item Define $\bm R_\alpha=\begin{bmatrix} -1&0&0\\0&-1&0\\0&0&-1\end{bmatrix}$, $\bm R_\beta=\begin{bmatrix} 1&0&0\\0&-1&0\\0&0&-1\end{bmatrix}$, $\bm R_\omega=\begin{bmatrix} -1&0&0\\0&-1&0\\0&0&1\end{bmatrix}$, $\bm R_\theta=\begin{bmatrix} 0&0&1\\1&0&0\\0&1&0\end{bmatrix}$, $\bm R_\gamma=\begin{bmatrix} -1&0&0\\0&0&1\\0&1&0\end{bmatrix}$, and $\bm R_e=\bm I_3$.
	\item Let $O_h=\langle\bm R_\alpha,\bm R_\beta,\bm R_\omega,\bm R_\theta,\bm R_\gamma\rangle$ denote the full octahedral point group, let $O$ be its orientation-preserving subgroup, and let $\mathrm{Alt}_4\triangleright V_4\triangleright\mathrm{Cyc}_2\triangleright\{e\}$ be the standard subgroup chain in $O\cong\mathrm{Sym}_4$. For a point subgroup $K$, write $\widetilde K=\{(\bm R,\bm0):\bm R\in K\}\leq G$. The body-centering translation $\tau=(\bm I_3,(\frac12,\frac12,\frac12))$ is central modulo $T$, and
	\[
		G=\widetilde{O_h}\times\langle\tau\rangle.
	\]
	Consequently,
	\[
		\{e\}\triangleleft\widetilde{\mathrm{Cyc}_2}
		\triangleleft\widetilde{V_4}
		\triangleleft\widetilde{\mathrm{Alt}_4}
		\triangleleft\widetilde O
		\triangleleft\widetilde{O_h}
		\triangleleft G
	\]
	is a composition series. Its factors, in ascending order, have orders $2,2,3,2,2,$ and $2$.
\end{itemize}

For this series, the successive subgroup orders are $1,2,4,12,24,48,96$.
Using the matrices above, choose the lifts
\[
    a_1=(\bm R_\beta,\bm0),\quad a_2=(\bm R_\omega,\bm0),\quad
    a_3=(\bm R_\theta,\bm0),\quad a_4=(\bm R_\gamma,\bm0),
\]
\[
    a_5=(\bm R_\alpha,\bm0),\qquad a_6=\tau.
\]
The subgroup at step $i$ is generated by $a_1,\ldots,a_i$.
By \cref{thm:hierarchical_group_action}, every element of $G$ has the unique form
\[
    a_6^{k_6}a_5^{k_5}a_4^{k_4}a_3^{k_3}a_2^{k_2}a_1^{k_1},
    \qquad 0\leq k_i<p_i,\qquad
    (p_1,\ldots,p_6)=(2,2,3,2,2,2).
\]
Thus one lifted generator at each step constructs all $96$ affine symmetry
operations. Applying them to a point with trivial stabilizer recovers $96$
distinct orbit points; for a special position, the number is
$96/|G_x|$. For $G$-invariant data, one sample value determines all values on
this orbit. The linear point group $O_h$ has order $48$, while $G$ also
contains the body-centring translation modulo $T$.

For a concrete example, take the torus point
$[\bm x]=[(1,3,5)^{\mathsf T}/32]$.
Its coordinates have distinct positive magnitudes below $1/4$, so only
the identity signed coordinate permutation fixes it. Composing any signed
coordinate permutation with $\tau$ places all coordinates in $(1/4,3/4)$
modulo one and therefore cannot fix this point. Thus its stabilizer is
trivial and its orbit has $96$ points. This point belongs to the
$G$-admissible grid with $\bm A=16\bm I_3$ and
$\bm s=(1/2,1/2,1/2)^{\mathsf T}$, where its index is $(0,1,2)^{\mathsf T}$.
The same six-step construction therefore recovers $96$ distinct indices
from this single discrete representative.

\section{Mathematical setup for SFFT}
\label{sec:math}

\subsection{Torus representation and continuous-domain symmetry}
\label{sec:space_group_induced_action}

Let $\mathbb{T}^n\coloneqq\mathbb{R}^n/\mathbb{Z}^n$ be the $n$-dimensional torus, where elements are denoted by $[\bm{x}]=\bm{x}+\mathbb{Z}^n$. Since the functions considered here are periodic under the integer translations $\mathbb Z^n$, continuous objects exhibiting this periodicity can be naturally viewed as being defined on the torus $\mathbb{T}^n$. The affine transformation action of the group $\mathcal{G}$ on $\mathbb{R}^n$ naturally induces an action of the quotient group $G$ on the torus $\mathbb{T}^n$.

\begin{definition}\label{def:group-action-on-torus}
    Define the binary map $\odot: G\times\mathbb{T}^n\to\mathbb{T}^n$ as
    \[
      \bar g\odot[\bm{x}] \coloneqq [g\cdot\bm{x}].
    \]
    It is straightforward to verify that this map is well-defined and constitutes a group action of $G$ on $\mathbb{T}^n$. 
\end{definition}

\begin{definition}\label{def:G-symmetry}
	A function $f:\mathbb{T}^n\to\mathbb{C}$ is said to possess $G$-symmetry if for any $\bar g\in G$ and $[\bm{x}]\in\mathbb{T}^n$, it holds that
	\[
		f(\bar g \odot [\bm{x}]) = f([\bm{x}]).
	\]
\end{definition}

\subsection{Spatial discretization and discrete index groups}
\label{subsec:grid_definition}

Consider a function $f:\mathbb{T}^n\to\mathbb{C}$. To discretize $f$, we introduce a positive integer diagonal matrix
\[
	\bm{A}=\operatorname{diag}(N_1,\ldots,N_n),\qquad N_j\in\mathbb{Z}_{\ge 1}.
\]
Define the finite discrete subgroup
\begin{equation}
	\Lambda_{\bm{A}}
	\coloneqq
	\bm{A}^{-1}\mathbb{Z}^n/\mathbb{Z}^n
	=\{[\bm{A}^{-1}\bm{k}]:\bm{k}\in\mathbb{Z}^n\}
	\subset \mathbb{T}^n.
\end{equation}
For a given shift vector $\bm{s}\in[0,1)^n$, the corresponding shifted grid is defined as a coset of $\Lambda_{\bm{A}}$
\begin{equation}
	\mathbb{T}^n_{\bm{A},\bm{s}}
	\coloneqq
	[\bm{A}^{-1}\bm{s}]+\Lambda_{\bm{A}}
	\subset \mathbb{T}^n.
\end{equation}
Denote the discrete index group as $X_{\bm{A}} \coloneqq \mathbb{Z}^n/\bm{A}\mathbb{Z}^n$, and define the following parameterization map
\begin{equation}
	\iota_{\bm{A},\bm{s}}:X_{\bm{A}}\to\mathbb{T}^n_{\bm{A},\bm{s}},
	\qquad
	[\bm{x}]\longmapsto [\bm{A}^{-1}(\bm{x}+\bm{s})].
\end{equation}
It is readily seen that this map is well-defined and establishes a bijection between the discrete index group $X_{\bm{A}}$ and the shifted grid $\mathbb{T}^n_{\bm{A},\bm{s}}$.
\begin{remark}
    For the positive integer diagonal matrix $\bm{A}=\operatorname{diag}(N_1,\dots,N_n)$, we consistently adopt the standard fundamental domain of $X_{\bm{A}}$ as
    \[
        \mathcal{D}_{\bm{A}} \coloneqq \{\bm{x}\in\mathbb{Z}^n : 0\le x_j < N_j,\ j=1,\dots,n\}.
    \]
    Consequently, each congruence class $[\bm{x}]\in X_{\bm{A}}$ has a unique representative in $\mathcal{D}_{\bm{A}}$.

    Unless otherwise specified, in what follows, for any positive integer diagonal matrix $\bm{M}$, the representative vector $\bm{v}$ of a congruence class $[\bm{v}] \in X_{\bm{M}}$ is by default assumed to be its unique representative in the corresponding standard fundamental domain $\mathcal{D}_{\bm{M}}$.
\end{remark}

\subsection{\texorpdfstring{$G$}{G}-admissibility of grids and algebraic criteria}
\label{subsec:admissibility}

\begin{definition}\label{def:admissible_grid}
	The grid parameters $(\bm{A},\bm{s})$ are said to be $G$-admissible if, for any $\bar g\in G$, it holds that
	\begin{equation}
		\bar g \odot \mathbb{T}^n_{\bm{A},\bm{s}} = \mathbb{T}^n_{\bm{A},\bm{s}}.
	\end{equation}
\end{definition}

To transform the aforementioned geometric admissibility condition into easily verifiable algebraic criteria, we define the following associated affine matrix and vector for any $g\in\mathcal{G}$
\begin{equation}
	\bm M_g \coloneqq \bm A\bm R_g\bm A^{-1},
	\qquad
	\bm z_g \coloneqq (\bm M_g-\bm I_n)\bm s+\bm A\bm t_g.
\end{equation}

\begin{lemma}\label{lem:rep-independence}
Let $\bar g\in G$. For any two chosen representatives $g_1,g_2\in\bar g$, we have
\begin{enumerate}
	\item $\bm R_{g_1}=\bm R_{g_2}$, which implies $\bm M_{g_1}=\bm M_{g_2}$.
	\item $\bm z_{g_1}-\bm z_{g_2}\in \bm A\mathbb Z^n$.
\end{enumerate}
\end{lemma}

\begin{proof}
    Since $g_1, g_2 \in \bar{g}$, there exists $h = (\bm{I}_n, \bm{k}) \in T$ with $\bm{k} \in \mathbb{Z}^n$ such that $g_1 = g_2 h$. By the multiplication rule for affine transformations,
    \[
        g_1 = g_2 h = (\bm{R}_{g_2}, \bm{t}_{g_2})(\bm{I}_n, \bm{k})
            = (\bm{R}_{g_2}\bm{I}_n,\, \bm{R}_{g_2}\bm{k} + \bm{t}_{g_2})
            = (\bm{R}_{g_2},\, \bm{R}_{g_2}\bm{k} + \bm{t}_{g_2}),
    \]
    so $\bm{R}_{g_1} = \bm{R}_{g_2}$, which implies $\bm{M}_{g_1} = \bm{M}_{g_2}$, and $\bm{t}_{g_1} = \bm{R}_{g_2}\bm{k} + \bm{t}_{g_2}$. Since $\bm{R}_{g_2} \in \operatorname{GL}^{\pm 1}(n, \mathbb{Z})$ and $\bm{k} \in \mathbb{Z}^n$, it follows that $\bm{t}_{g_1} - \bm{t}_{g_2} = \bm{R}_{g_2}\bm{k} \in \mathbb{Z}^n$. Therefore, using $\bm{M}_{g_1} = \bm{M}_{g_2}$,
    \[
        \bm{z}_{g_1} - \bm{z}_{g_2} = \bm{A}(\bm{t}_{g_1} - \bm{t}_{g_2}) \in \bm{A}\mathbb{Z}^n.
    \]
\end{proof}

\begin{remark}
    By \cref{lem:rep-independence}, for any $\bar g\in G$ and any representative $g\in\bar g$, the values of $\bm R_g$ and $\bm M_g$ do not depend on the choice of $g$. Thus, they are unambiguously determined by $\bar g$ and can be defined as
    \[
        \bm{R}_{\bar g} \coloneqq \bm{R}_g, \qquad 
        \bm{M}_{\bar g} \coloneqq \bm{M}_g.
    \]
    Furthermore, by \cref{lem:rep-independence} the congruence class of the associated affine vector $\bm z_g$ modulo $\bm A\mathbb Z^n$ is independent of the choice of representative. When $\bm z_g\in\mathbb Z^n$ (as holds for $G$-admissible grids), this class defines an element of $X_{\bm A}$, which we denote by
    \[
        \bm{z}_{\bar g} \coloneqq [\bm{z}_g] \in X_{\bm A}.
    \]
\end{remark}

\begin{proposition}\label{prop:grid_invariance}
Let $(\bm A,\bm s)$ be grid parameters. The following two conditions are equivalent.
    \begin{enumerate}
      \item $(\bm A,\bm s)$ is $G$-admissible.
      \item For any $\bar{g}\in G$ and a chosen representative $g\in \bar{g}$, it holds that
        \[
            \bm{M}_g \in \operatorname{GL}^{\pm 1}(n,\mathbb{Z}) \quad \text{and} \quad \bm{z}_g \in \mathbb{Z}^n.
        \]
    \end{enumerate}
\end{proposition}

\begin{proof}
    First, note that by \cref{lem:rep-independence}, the conditions $\bm{M}_g \in \operatorname{GL}^{\pm 1}(n,\mathbb{Z})$ and $\bm{z}_g \in \mathbb{Z}^n$ in condition 2 of the proposition are independent of the choice of the representative $g \in \bar{g}$; hence, it is well-defined.

  $\text{1}\Rightarrow\text{2}$. Fix $\bar{g}\in G$ and any representative $g=(\bm R_g,\bm t_g)\in \bar g$. By the $G$-admissibility, for any $\bm k\in\mathbb Z^n$, we have
	\[
		\bar g\odot[\bm A^{-1}(\bm k+\bm s)]\in\mathbb T^n_{\bm A,\bm s}.
	\]
	Therefore, there exists $\bm k'\in\mathbb Z^n$ such that
	\[
		\bar g\odot[\bm A^{-1}(\bm k+\bm s)]=[\bm A^{-1}(\bm k'+\bm s)].
	\]
	Equivalently, there exists $\bm p\in\mathbb Z^n$ such that
	\[
		\bm A\bm R_g\bm A^{-1}(\bm k+\bm s)+\bm A\bm t_g=\bm k'+\bm s+\bm A\bm p.
	\]
	Utilizing $\bm M_g=\bm A\bm R_g\bm A^{-1}$ and $\bm z_g=(\bm M_g-\bm I_n)\bm s+\bm A\bm t_g$, the above equation is equivalent to
	\[
		\bm M_g\bm k+\bm z_g=\bm k'+\bm A\bm p\in\mathbb Z^n,\qquad \forall \bm k\in\mathbb Z^n.
	\]
	Setting $\bm k=\bm 0$ yields $\bm z_g\in\mathbb Z^n$. Setting $\bm k=\bm e_j$ ($j=1,\ldots,n$ as standard basis vectors) gives $\bm M_g\bm e_j\in\mathbb Z^n$, implying $\bm M_g\in\mathbb Z^{n\times n}$. On the other hand,
	\[
		\det(\bm M_g)=\det(\bm A\bm R_g\bm A^{-1}) =\det(\bm R_g) = \pm 1.
	\]
	Thus, $\bm M_g$ is an integer matrix with a determinant of $\pm1$, i.e., $\bm M_g\in\operatorname{GL}^{\pm 1}(n,\mathbb Z)$.

	$\text{2}\Rightarrow\text{1}$. Fix $\bar{g}\in G$ and $[\bm x]\in\mathbb T^n_{\bm A,\bm s}$. By definition, we can take $\bm k\in\mathbb Z^n$ such that $[\bm x]=[\bm A^{-1}(\bm k+\bm s)]$. For any representative $g\in \bar{g}$, direct calculation yields
	\[
		\bar{g}\odot[\bm x]=[g\cdot\bm x]=\Big[\bm A^{-1}\big(\bm M_g\bm k+\bm z_g+\bm s\big)\Big].
	\]
	From condition 2 of the proposition, we know $\bm M_g\bm k+\bm z_g\in\mathbb Z^n$, thus $\bar{g}\odot[\bm x]\in\mathbb T^n_{\bm A,\bm s}$, which implies $\bar{g}\odot\mathbb T^n_{\bm A,\bm s}\subseteq\mathbb T^n_{\bm A,\bm s}$. Since condition (2) holds for every $\bar g\in G$, it also holds for $\bar g^{-1}$; thus, similarly, $\bar{g}^{-1}\odot\mathbb T^n_{\bm A,\bm s}\subseteq\mathbb T^n_{\bm A,\bm s}$. Applying $\bar{g}$ to both sides yields $\mathbb T^n_{\bm A,\bm s}\subseteq \bar{g}\odot\mathbb T^n_{\bm A,\bm s}$. Consequently, $\bar{g}\odot\mathbb T^n_{\bm A,\bm s}=\mathbb T^n_{\bm A,\bm s}$, indicating that $(\bm A, \bm s)$ is $G$-admissible.
\end{proof}

\begin{remark}
	If the point group contains a coordinate permutation $\bm R_g$, then $G$-admissibility requires $\bm M_g=\bm A\bm R_g\bm A^{-1}\in\operatorname{GL}^{\pm 1}(n,\mathbb Z)$; consequently, the corresponding diagonal entries $N_i$ of $\bm A$ within each cycle of that permutation must be equal. 
\end{remark}

\begin{proposition}\label{prop:existence_admissible}
For any arithmetic representation $\mathcal{G}$ of an $n$-dimensional space group, there exist $G$-admissible grid parameters $(\bm A, \bm s)$.
\end{proposition}

\begin{proof}
	Let $\bm A = N \bm I_n$ ($N \in \mathbb{Z}_{\geq 1}$) and $\bm s = \bm 0$. 
	Then, for any $g\in\mathcal G$, we have
	\[
		\bm M_g=\bm A\bm R_g\bm A^{-1}=\bm R_g\in\operatorname{GL}^{\pm 1}(n,\mathbb Z),
	\]
	and on the other hand, 
	\[
		\bm z_g=(\bm M_g-\bm I_n)\bm s+\bm A\bm t_g=N\bm t_g.
	\]
	Since $G$ is a finite group, we can choose a set of coset representatives $\{g_1,\dots,g_{|G|}\}$. Let $L_0$ be the smallest positive integer such that $L_0\bm t_{g_j}\in\mathbb Z^n$ for every $j$. By choosing $N$ to be an integer multiple of $L_0$, we obtain $N\bm t_{g_j}\in\mathbb Z^n$ for all $j$. 
	Fix an arbitrary $\bar g\in G$ and a representative $g\in\bar g$. Since $\{g_j\}$ is a complete set of coset representatives, $\bar g=g_jT$ for some $j$; hence there exists an integer translation $h = (\bm I_n, \bm k)\in T$ such that $g=g_j h$, 
	which yields $\bm t_g=\bm R_{g_j}\bm k+\bm t_{g_j}$, and consequently $N\bm t_g\in\mathbb Z^n$.
	
	In summary, for any $\bar g\in G$ (with any chosen representative $g$), it holds that $\bm M_g\in\operatorname{GL}^{\pm 1}(n,\mathbb Z)$ and $\bm z_g\in\mathbb Z^n$, satisfying condition 2 of \cref{prop:grid_invariance}. By \cref{prop:grid_invariance}, condition 1 holds, which shows that the grid parameters $(\bm A,\bm 0)$ are $G$-admissible. 
\end{proof}

\subsection{Discrete induced action and symmetry inheritance}

Define the discretized function
\begin{equation}
	f_{\bm A,\bm s}:X_{\bm A}\to\mathbb C,
	\qquad
	f_{\bm A,\bm s}([\bm x]) \coloneqq f(\iota_{\bm A,\bm s}([\bm x])),\qquad \forall [\bm x]\in X_{\bm A}.
\end{equation}

\begin{definition}\label{def:discrete_action}
	Let $(\bm A,\bm s)$ be $G$-admissible grid parameters. Since $(\bm A,\bm s)$ is $G$-admissible, $\bar g\odot\mathbb T^n_{\bm A,\bm s}=\mathbb T^n_{\bm A,\bm s}$ for every $\bar g\in G$; hence the composition below is a well-defined map $X_{\bm A}\to X_{\bm A}$. For any $\bar g\in G$, define $\Psi_{\bar g}:X_{\bm A}\to X_{\bm A}$ as
	\begin{equation}\label{eq:discrete_action_def}
		\Psi_{\bar g} \coloneqq \iota_{\bm A,\bm s}^{-1}\circ(\bar g\odot)\circ\iota_{\bm A,\bm s}.
	\end{equation}
\end{definition}

\begin{proposition}\label{prop:Psi_is_permutation}
    Let $(\bm{A},\bm{s})$ be $G$-admissible grid parameters. For any $\bar{g} \in G$, the map $\Psi_{\bar{g}}$ is a permutation on $X_{\bm{A}}$, i.e.,
    \[
        \Psi_{\bar{g}} \in \operatorname{Sym}(X_{\bm{A}}).
    \]
\end{proposition}

\begin{proof}
    Fix $\bar g\in G$. By \cref{def:group-action-on-torus}, the map $\bar g\odot(\cdot):\mathbb T^n\to\mathbb T^n$ is a well-defined bijection (being a group action). 
    By \cref{def:admissible_grid}, we have
    \[
        \bar g\odot\mathbb{T}^n_{\bm A,\bm s}=\mathbb{T}^n_{\bm A,\bm s},
    \]
    thus its restriction map $(\bar g\odot)|_{\mathbb{T}^n_{\bm A,\bm s}}: \mathbb{T}^n_{\bm A,\bm s}\to\mathbb{T}^n_{\bm A,\bm s}$ is a well-defined self-map. Furthermore, this self-map is a bijection, since its inverse is precisely the restricted action $(\bar g^{-1}\odot)|_{\mathbb{T}^n_{\bm A,\bm s}}$. 
    
    On the other hand, the parameterization map $\iota_{\bm A,\bm s}:X_{\bm A}\to\mathbb{T}^n_{\bm A,\bm s}$ itself is a bijection; hence, its inverse map $\iota_{\bm A,\bm s}^{-1}:\mathbb{T}^n_{\bm A,\bm s}\to X_{\bm A}$ is also well-defined and bijective. 
    
    Therefore, the map
    \[
        \Psi_{\bar g}=\iota_{\bm A,\bm s}^{-1}\circ(\bar g\odot)\circ\iota_{\bm A,\bm s}
    \]
    as a composition of three well-defined bijections, must constitute a bijection on $X_{\bm A}$. That is, $\Psi_{\bar g}\in\operatorname{Sym}(X_{\bm A})$. 
\end{proof}

\begin{proposition}\label{prop:G_A_group}
    Let $(\bm{A},\bm{s})$ be $G$-admissible grid parameters. Define the map
    \[
        \Psi \colon G \to \operatorname{Sym}(X_{\bm{A}}),\qquad \bar{g} \mapsto \Psi_{\bar{g}}.
    \]
    Then $\Psi$ is a group homomorphism, and its image
    \[
        G_{\bm{A}} \coloneqq \operatorname{Im}(\Psi) = \{\Psi_{\bar{g}} : \bar{g} \in G \} \leq \operatorname{Sym}(X_{\bm{A}})
    \]
    is called the discrete symmetry group compatible with the index group $X_{\bm{A}}$. 
\end{proposition}

\begin{proof}
    For any $\bar g_1,\bar g_2\in G$, we have
    \begin{align*}
        \Psi_{\bar g_1}\circ\Psi_{\bar g_2}
        &=\iota^{-1}\circ(\bar g_1\odot)\circ\iota\circ\iota^{-1}\circ(\bar g_2\odot)\circ\iota\\
        &=\iota^{-1}\circ(\bar g_1\odot)\circ(\bar g_2\odot)\circ\iota\\
        &        =\iota^{-1}\circ((\bar g_1\bar g_2)\odot)\circ\iota
        =\Psi_{\bar g_1\bar g_2},
    \end{align*}
    where $\iota=\iota_{\bm A,\bm s}$. The merging of the two actions uses the group-action compatibility law $(\bar g_1\bar g_2)\odot[\bm x]=\bar g_1\odot(\bar g_2\odot[\bm x])$ from \cref{def:group-action-on-torus}.
	Thus, $\Psi:G\to\operatorname{Sym}(X_{\bm A})$ is a group homomorphism. Hence $\operatorname{Im}(\Psi)\leq\operatorname{Sym}(X_{\bm A})$, i.e., $G_{\bm A}$ is a subgroup. This completes the proof. 
\end{proof}

\begin{proposition}\label{prop:discrete_action_formula} 
        Let $(\bm A,\bm s)$ be $G$-admissible grid parameters. For any $\bar g\in G$ and $[\bm x]\in X_{\bm A}$, we have
        \begin{equation}
                \Psi_{\bar g}([\bm x]) = [\bm M_{\bar g}\bm x] + \bm z_{\bar g},  
        \end{equation}
\end{proposition}

\begin{proof}
	Arbitrarily choose $\bar g\in G$ and $[\bm x]\in X_{\bm A}$, and take any representative $g\in\bar g$. 
	Due to the well-definedness of \cref{def:group-action-on-torus}, the following expression is independent of the choice of representative
	\[
		\bar g\odot\iota_{\bm A,\bm s}([\bm x])
		=\bar g\odot[\bm A^{-1}(\bm x+\bm s)]
		=[g\cdot(\bm A^{-1}(\bm x+\bm s))]
		=[\bm R_g\bm A^{-1}(\bm x+\bm s)+\bm t_g].
	\]
	Since $\bm M_g=\bm A\bm R_g\bm A^{-1}$, we obtain
	\[
		\bm A\big(\bm R_g\bm A^{-1}(\bm x+\bm s)+\bm t_g\big)
		=\bm M_g\bm x+\bm M_g\bm s+\bm A\bm t_g.
	\]
	Therefore,
	\[
		\bar g\odot\iota_{\bm A,\bm s}([\bm x])
		=\Big[\bm A^{-1}\big(\bm M_g\bm x+\bm M_g\bm s+\bm A\bm t_g\big)\Big].
	\]
	Furthermore, from
	\[
		\bm z_g=(\bm M_g-\bm I_n)\bm s+\bm A\bm t_g
		=\bm M_g\bm s+\bm A\bm t_g-\bm s
	\]
	it follows that
	\[
		\bm M_g\bm x+\bm M_g\bm s+\bm A\bm t_g
		=\bm M_g\bm x+\bm z_g+\bm s.
	\]
	Consequently,
	\[
		\Psi_{\bar g}([\bm x])
			=\iota_{\bm A,\bm s}^{-1}\!\left(\Big[\bm A^{-1}\big(\bm M_g\bm x+\bm z_g+\bm s\big)\Big]\right)
		=[\bm M_g\bm x]+[\bm z_g],
	\]
	where the result lies in $X_{\bm A}$ thanks to $G$-admissibility, by condition 2 of \cref{prop:grid_invariance} we have $\bm z_g\in\mathbb Z^n$, and the admissibility also ensures $\bar g\odot\iota_{\bm A,\bm s}([\bm x])\in\mathbb T^n_{\bm A,\bm s}$ so that $\iota_{\bm A,\bm s}^{-1}$ is applicable.
	Finally, by \cref{lem:rep-independence}, we have $\bm M_g=\bm M_{\bar g}$ and $[\bm z_g]=\bm z_{\bar g}$, which implies
	\[
		\Psi_{\bar g}([\bm x])=[\bm M_{\bar g}\bm x]+\bm z_{\bar g}.
	\]
	This value is moreover independent of the chosen representative $\bm x$ of $[\bm x]$, if $\bm x'=\bm x+\bm A\bm k$ with $\bm k\in\mathbb Z^n$, then $\bm M_{\bar g}\bm x'=\bm M_{\bar g}\bm x+\bm M_{\bar g}\bm A\bm k=\bm M_{\bar g}\bm x+\bm A\bm R_{\bar g}\bm k$, and $\bm A\bm R_{\bar g}\bm k\in\bm A\mathbb Z^n$, so $[\bm M_{\bar g}\bm x']=[\bm M_{\bar g}\bm x]$ in $X_{\bm A}$. This completes the proof. 
\end{proof}

From \cref{def:discrete_action} and \cref{prop:G_A_group}, it is evident that the action of $G$ on $\mathbb{T}^n_{\bm A,\bm s}$ is transported via conjugation by the bijection $\iota_{\bm A,\bm s}$ into a group action on $X_{\bm A}$. Therefore, if $f$ possesses $G$-symmetry, $f_{\bm A,\bm s}$ remain invariant under this discrete action. 

\begin{proposition}\label{prop:symmetry_inheritance}
	Suppose $f$ possesses $G$-symmetry, and let $(\bm A,\bm s)$ be $G$-admissible grid parameters. Then for any $\bar g\in G$ and $[\bm x]\in X_{\bm A}$, it holds that
	\begin{align}\label{eq:discrete_symmetry}
		f_{\bm A,\bm s}(\Psi_{\bar g}([\bm x])) = f_{\bm A,\bm s}([\bm x]).
	\end{align}
\end{proposition}

\begin{proof}
	Arbitrarily choose $\bar g\in G$ and $[\bm x]\in X_{\bm A}$. From \cref{def:discrete_action}, we have $\iota_{\bm A,\bm s}(\Psi_{\bar g}([\bm x])) = \bar g\odot\iota_{\bm A,\bm s}([\bm x])$. 
	Combining this with the definition of $f_{\bm A,\bm s}$ and the $G$-symmetry of $f$, we can directly deduce the following
	\begin{align*}
		f_{\bm A,\bm s}(\Psi_{\bar g}([\bm x]))
		&= f\big(\iota_{\bm A,\bm s}(\Psi_{\bar g}([\bm x]))\big) \\
		&= f\big(\bar g\odot\iota_{\bm A,\bm s}([\bm x])\big) \\
		&= f\big(\iota_{\bm A,\bm s}([\bm x])\big) \qquad\text{(by the $G$-symmetry of $f$, \cref{def:G-symmetry})} \\
		&= f_{\bm A,\bm s}([\bm x]).
	\end{align*}
	This completes the proof. 
\end{proof}

\subsection{Multidimensional Cooley--Tukey decomposition}\label{subsec:dft_convention}
The core idea behind the fast Fourier transform (FFT) stems from the seminal work of Cooley and Tukey \cite{cooley1965algorithm}. For the multidimensional Cooley--Tukey decomposition, we refer the reader to \cite{Shmueli2008ITC_B}. In this subsection, we first establish the notation for the discrete Fourier transform (DFT) and its inverse (IDFT) adopted in this paper, and then present the fundamental derivation of the multidimensional Cooley--Tukey decomposition. 

Define
\begin{equation}\label{eq:pairing}
		\langle [\bm h], [\bm x] \rangle \coloneqq e^{2\pi i \, \bm h^\mathsf{T} \bm A^{-1} \bm x}, \qquad [\bm h]\in X_{\bm A^\mathsf{T}},\ [\bm x]\in X_{\bm A}.
\end{equation}
It is straightforward to verify that the above expression is well-defined. 

\begin{remark}
	Since the discretization matrices in this paper are all positive integer diagonal matrices, we have $X_{\bm A^{\mathsf{T}}}=X_{\bm A}$, and consequently the superscript $\mathsf{T}$ on them may be omitted. 
\end{remark}

For a discrete function $f_{\bm A,\bm s}$ defined on $X_{\bm A}$, its DFT is defined as
\begin{equation}\label{eq:dft_convention}
	\hat f_{\bm A,\bm s}([\bm h]) \coloneqq \sum_{[\bm x]\in X_{\bm A}} f_{\bm A,\bm s}([\bm x])\, e^{-2\pi i\,\bm h^\mathsf{T}\bm A^{-1}\bm x},\qquad [\bm h]\in X_{\bm A},
\end{equation}
and the corresponding IDFT is defined as
\begin{equation}\label{eq:idft_convention}
	f_{\bm A,\bm s}([\bm x]) = \frac{1}{|X_{\bm A}|}\sum_{[\bm h]\in X_{\bm A}} \hat f_{\bm A,\bm s}([\bm h])\, e^{2\pi i\,\bm h^\mathsf{T}\bm A^{-1}\bm x},\qquad [\bm x]\in X_{\bm A},
\end{equation}
where $|X_{\bm A}|=\det\bm A=\prod_{j=1}^n N_j$. 

\begin{remark}
	Defining the DFT on the shifted grid $\mathbb{T}^n_{\bm A,\bm s}$ as
	\[
		\hat f^{\mathrm{phys}}_{\bm A,\bm s}([\bm h])
		\coloneqq\sum_{[\bm x]\in X_{\bm A}} f_{\bm A,\bm s}([\bm x])\,
		e^{-2\pi i\,\bm h^\mathsf{T}\bm A^{-1}(\bm x+\bm s)},
	\]
	it yields
	\[
		\hat f^{\mathrm{phys}}_{\bm A,\bm s}([\bm h])
		= e^{-2\pi i\,\bm h^\mathsf{T}\bm A^{-1}\bm s}\,\hat f_{\bm A,\bm s}([\bm h]).
	\]
	The phase factor $e^{-2\pi i\,\bm h^\mathsf{T}\bm A^{-1}\bm s}$ is independent of the spatial index $[\bm x]$ and depends solely on the frequency index $[\bm h]$. This paper investigates the standard DFT without the phase factor; to obtain the Fourier coefficients in a physical sense, one simply multiplies $\hat{f}_{\bm A,\bm s}([\bm h])$ pointwise by this phase factor. 
\end{remark}

Suppose the discretization matrix $\bm A$ can be factored as
$\bm A=\bm A_0\bm A_1$, where $\bm A_0=\operatorname{diag}(N_{0,1},\dots,N_{0,n})$ and 
$\bm A_1=\operatorname{diag}(N_{1,1},\dots,N_{1,n})$ are both positive integer diagonal matrices, satisfying 
$N_j=N_{0,j}N_{1,j}$ for each $j$. In what follows, $[\cdot]$ denotes the congruence class in the corresponding quotient group. Any spatial index $[\bm x]\in X_{\bm A}$ and frequency index $[\bm h]\in X_{\bm A}$ can be uniquely decomposed as
\begin{align}
	[\bm x] & = [\bm A_0 \bm x_1 + \bm x_0], \quad [\bm x_0]\in X_{\bm A_0},\ [\bm x_1]\in X_{\bm A_1}, \label{eq:x_decomp} \\
	[\bm h] & = [\bm A_1 \bm h_0 + \bm h_1], \quad [\bm h_0]\in X_{\bm A_0},\ [\bm h_1]\in X_{\bm A_1}. \label{eq:h_decomp}
\end{align}        
Substituting \eqref{eq:x_decomp}--\eqref{eq:h_decomp} into the definition of the DFT yields the multidimensional Cooley--Tukey decomposition formula
\begin{equation}\label{eq:cooley-tukey-decomp}
	\hat f_{\bm A,\bm s}([\bm h])=
	\sum_{[\bm x_0]\in X_{\bm A_0}}e^{-2\pi i\bm h_0^\mathsf{T}\bm A_0^{-1}\bm x_0}
	\left[
	e^{-2\pi i\bm h_1^\mathsf{T}\bm A^{-1}\bm x_0}\sum_{[\bm x_1]\in X_{\bm A_1}}f_{\bm A,\bm s}([\bm A_0\bm x_1+\bm x_0])\,
	e^{-2\pi i\bm h_1^\mathsf{T}\bm A_1^{-1}\bm x_1}
	\right].
\end{equation}
This decomposition constitutes the general algebraic framework for the subsequent symmetry reduction algorithms.

\section{Algorithmic framework and symmetry-reduction derivations for SFFT}
\label{sec:algor}
\subsection{Grid decomposition and partition strategy under group action}

For grid parameters $(\bm A,\bm s)$ satisfying the $G$-admissibility condition, this section investigates how to incorporate the discrete symmetry group $G_{\bm A}$ into the Cooley--Tukey decomposition framework and construct a compatible subgrid partition. Specifically, the Cooley--Tukey decomposition provides a factorization method for the matrix $\bm A$, while the discrete symmetry group $G_{\bm A}$ imposes algebraic constraints on this factorization, thereby laying the foundation for subsequent symmetry reduction. 

If $G_{\bm A}$ is a trivial group, this framework reduces to the standard Cooley--Tukey decomposition. Therefore, hereafter we only consider the case where $G_{\bm A}$ is nontrivial. In this scenario, depending on whether the linear part $\bm M_{\bar g}$ of the element $\Psi_{\bar g}$ in the discrete symmetry group is consistently equal to the identity matrix, the problem can be categorized into the following two mutually exclusive and exhaustive cases.
\[
\text{Case I. }\exists\,\Psi_{\bar g}\in G_{\bm A}\ \text{such that}\ \bm M_{\bar g}\neq\bm I_n,
\qquad
\text{Case II. }\forall\,\Psi_{\bar g}\in G_{\bm A},\ \bm M_{\bar g}=\bm I_n.
\]
In what follows, we will analyze Case I and Case II separately, and subsequently unify them within the framework of orbit decomposition. 

\subsubsection{Case I: nontrivial linear-action case}
\label{case:nontrivial_linear}
If there exists $\Psi_{\bar g}\in G_{\bm A}$ such that the linear part $\bm M_{\bar g} \neq \bm I_n$, this is referred to as the nontrivial linear-action case. In this scenario, the discretization matrix is factored as
\begin{equation}\label{eq:AR_factorization}
	\bm A = \bm A_R^0 \bm A_R^1,
\end{equation}
where both $\bm A_R^0$ and $\bm A_R^1$ are positive integer diagonal matrices. Unlike the standard Cooley--Tukey decomposition, the $\bm A_R^0$ chosen here must satisfy the following $G_{\bm A}$-compatibility condition.

\begin{definition}\label{def:linear_compatibility}
	A positive integer diagonal matrix $\bm A_R^0$ is said to be $G_{\bm A}$-compatible if, for any $\Psi_{\bar g} \in G_{\bm A}$, it holds that
	\begin{equation}
		\bm M_{\bar g}(\bm A_R^0\mathbb Z^n)=\bm A_R^0\mathbb Z^n.
	\end{equation}
\end{definition}

\begin{proposition}\label{prop:linear_compatibility}
	Let $\bm A_R^0$ be a positive integer diagonal matrix. Then the following two conditions are equivalent.
	\begin{enumerate}
		\item $\bm A_R^0$ is $G_{\bm A}$-compatible.
		\item For any $\Psi_{\bar g} \in G_{\bm A}$, we have
	\begin{align}
		(\bm A_R^0)^{-1}\bm M_{\bar g}\bm A_R^0 \in \operatorname{GL}^{\pm 1}(n,\mathbb Z).
	\end{align}
	\end{enumerate}
\end{proposition}

\begin{proof}
	$\text{1}\Rightarrow\text{2}$.
	Suppose $\bm A_R^0$ is $G_{\bm A}$-compatible. Then for any $\Psi_{\bar g}\in G_{\bm A}$,
	\[
		\bm M_{\bar g}(\bm A_R^0\mathbb Z^n)=\bm A_R^0\mathbb Z^n.
	\]
	Since $\bm A_R^0$ is invertible, the above equation is equivalent to
	\[
		(\bm A_R^0)^{-1}\bm M_{\bar g}\bm A_R^0\,\mathbb Z^n=\mathbb Z^n.
	\]
	Let $\bm K_{\bar g}\coloneqq(\bm A_R^0)^{-1}\bm M_{\bar g}\bm A_R^0$. Then $\bm K_{\bar g}\mathbb Z^n=\mathbb Z^n$. Consequently, $\bm K_{\bar g}$ maps the standard orthonormal basis to integer vectors, implying that $\bm K_{\bar g}$ is an integer matrix. On the other hand,
	\[
		\det(\bm K_{\bar g})=\det((\bm A_R^0)^{-1}\bm M_{\bar g}\bm A_R^0) =\det(\bm M_{\bar g}) = \pm 1.
	\]
	Therefore, $\bm K_{\bar g}\in\operatorname{GL}^{\pm 1}(n,\mathbb Z)$.

	$\text{2}\Rightarrow\text{1}$.
	Suppose for any $\Psi_{\bar g}\in G_{\bm A}$, it holds that
	\[
		\bm K_{\bar g}=(\bm A_R^0)^{-1}\bm M_{\bar g}\bm A_R^0\in\operatorname{GL}^{\pm 1}(n,\mathbb Z).
	\]
	Then $\bm K_{\bar g}\mathbb Z^n=\mathbb Z^n$, and further
	\[
		\bm M_{\bar g}(\bm A_R^0\mathbb Z^n)=\bm A_R^0\bm K_{\bar g}\mathbb Z^n=\bm A_R^0\mathbb Z^n.
	\]
	Thus, $\bm A_R^0$ is $G_{\bm A}$-compatible.
\end{proof}

\begin{remark}
	When $\bm A_R^0=k\bm I_n$ ($k\in\mathbb{Z}_{\ge1}$),
	\[
		(\bm A_R^0)^{-1}\bm M_{\bar g}\bm A_R^0
		=\bm M_{\bar g}\in\operatorname{GL}^{\pm 1}(n,\mathbb{Z}).
	\]
	By \cref{prop:linear_compatibility}, $\bm A_R^0$ is always $G_{\bm A}$-compatible in this case. Therefore, there exists at least one class of simple, explicitly constructible choices for $\bm A_R^0$, namely $\bm A_R^0=k\bm I_n$ for any $k\in\mathbb Z_{\ge1}$ (and, when deployed as a Cooley--Tukey factor in Case I, provided $k$ divides every diagonal entry $N_j$ of $\bm A$). However, in general, $\bm A_R^0$ is not uniquely determined by the group structure, but rather needs to be selected subject to the $G_{\bm A}$-compatibility constraint. 
\end{remark}

At this point, any $[\bm x]\in X_{\bm A}$ can be uniquely decomposed as
\begin{equation}\label{eq:x_AR_decomp}
	[\bm x]=[\bm A_R^0\bm x_1+\bm x_0],
	\quad
	[\bm x_0]\in X_{\bm A_R^0},\;[\bm x_1]\in X_{\bm A_R^1},
\end{equation}
which defines the family of subsets
\[
	\Theta_R \coloneqq \{X_R^{[\bm y]}:[\bm y]\in X_{\bm A_R^0}\}, 
\]
where $X_R^{[\bm y]} \coloneqq \{[\bm A_R^0\bm x_1+\bm y]:[\bm x_1]\in X_{\bm A_R^1}\}$.

\begin{proposition}\label{prop:layer_permutation}
	Let $\bm A=\bm A_R^0\bm A_R^1$ be the aforementioned fixed factorization, and suppose $\bm A_R^0$ is $G_{\bm A}$-compatible. Then for any $\Psi_{\bar g} \in G_{\bm A}$ and any $[\bm y]\in X_{\bm A_R^0}$, there exists a unique $[\bm y']\in X_{\bm A_R^0}$ such that
	\[
		\Psi_{\bar g}\bigl(X_R^{[\bm y]}\bigr)=X_R^{[\bm y']}.
	\]
	Specifically, $[\bm y']=[\bm M_{\bar g}\bm y+\bm z_g]\in X_{\bm A_R^0}$, where $g$ is any representative of $\bar g$ and $\bm y$ is any integer representative of $[\bm y]$. The class is independent of both choices: replacing $g$ by another representative $g'$ of $\bar g$ changes the class by $\bm z_{g'}-\bm z_g\in\bm A\mathbb Z^n\subseteq\bm A_R^0\mathbb Z^n$ (\cref{lem:rep-independence}), whereas replacing $\bm y$ by $\bm y+\bm A_R^0\bm\ell$ ($\bm\ell\in\mathbb Z^n$) changes it by $\bm M_{\bar g}\bm A_R^0\bm\ell\in\bm M_{\bar g}(\bm A_R^0\mathbb Z^n)=\bm A_R^0\mathbb Z^n$, where the last equality holds by the $G_{\bm A}$-compatibility of $\bm A_R^0$ (\cref{def:linear_compatibility}).
\end{proposition}

\begin{proof}
	Take any $[\bm y]\in X_{\bm A_R^0}$ and any $[\bm x_1]\in X_{\bm A_R^1}$, and let
	\[
		[\bm x]=[\bm A_R^0\bm x_1+\bm y]\in X_{\bm A}.
	\]
	By \cref{prop:discrete_action_formula}, for any $\bar g\in G$, we have
	\[
		\Psi_{\bar g}([\bm x])=[\bm M_{\bar g}(\bm A_R^0\bm x_1+\bm y)+\bm z_g].
	\]
	Since $\bm A_R^0$ is $G_{\bm A}$-compatible, there exists a unimodular matrix $\bm K_{\bar g}$ such that
	\[
		\bm M_{\bar g}\bm A_R^0=\bm A_R^0\bm K_{\bar g}.
	\]
	Thus,
	\[
		\Psi_{\bar g}([\bm x])
		=[\bm A_R^0\bm K_{\bar g}\bm x_1+\bm M_{\bar g}\bm y+\bm z_g].
	\]
	Denote $[\bm y']=[\bm M_{\bar g}\bm y+\bm z_g]\in X_{\bm A_R^0}$. If we take $\bm y'$ as the standard representative of this congruence class, then there exists $\bm k\in\mathbb Z^n$ such that $\bm M_{\bar g}\bm y+\bm z_g = \bm y'+\bm A_R^0\bm k$. Substituting this into the above equation yields
	\[
		\Psi_{\bar g}([\bm x])
		=[\bm A_R^0(\bm K_{\bar g}\bm x_1+\bm k)+\bm y'].
	\]
	Because $\bm K_{\bar g}\bm x_1+\bm k$ is an integer vector, its corresponding congruence class belongs to $X_{\bm A_R^1}$, which indicates that for any $[\bm x_1]\in X_{\bm A_R^1}$, we have
	\[
		\Psi_{\bar g}([\bm A_R^0\bm x_1+\bm y])\in X_R^{[\bm y']}.
	\]
	Consequently,
	\[
		\Psi_{\bar g}\bigl(X_R^{[\bm y]}\bigr)\subseteq X_R^{[\bm y']}.
	\]
	On the other hand, since $\Psi_{\bar g}$ is a bijection, and both $X_R^{[\bm y]}$ and $X_R^{[\bm y']}$ are equinumerous to $X_{\bm A_R^1}$, their cardinalities satisfy
	\[
		|\Psi_{\bar g}\bigl(X_R^{[\bm y]}\bigr)|=|X_R^{[\bm y]}|=|X_R^{[\bm y']}|.
	\]
	Combining this with the inclusion relation, we obtain
	\[
		\Psi_{\bar g}\bigl(X_R^{[\bm y]}\bigr)=X_R^{[\bm y']}.
	\]
	Since the sets within the family of subsets $\Theta_R=\{X_R^{[\bm y]}:[\bm y]\in X_{\bm A_R^0}\}$ are pairwise disjoint, the $[\bm y']$ satisfying the above equation is unique. 
\end{proof}

\subsubsection{Case II: pure translation case}
\label{case:pure_translation}
If $\bm M_{\bar g}=\bm I_n$ for all $\Psi_{\bar g}\in G_{\bm A}$, this scenario is referred to as the pure translation case. In this case, 
\begin{align}
	\Psi_{\bar g}([\bm x]) = [\bm x] + \bm z_{\bar g},\quad [\bm x]\in X_{\bm A},
\end{align}
where $\bm z_{\bar g} = [\bm z_g] = [\bm A\bm t_g]$. The discretization matrix is factored as
\begin{equation}\label{eq:AT_factorization}
	\bm A = \bm A_T^0 \bm A_T^1,
\end{equation}
where both $\bm A_T^0$ and $\bm A_T^1$ are positive integer diagonal matrices. Similar to Case I, the $\bm A_T^1$ here must ensure that all translation vectors $\bm t_g$, after being acted upon by $\bm A_T^1$, fall onto the integer lattice. 

\begin{definition}
	\label{def:translation_compatibility}
	A positive integer diagonal matrix $\bm A_T^1$ is said to be translation-compatible if, for any $\bar g \in G$ and a representative $g \in \bar g$, it holds that
	\begin{align}
		\bm A_T^1\bm t_g\in\mathbb Z^n.
	\end{align}
\end{definition}

\begin{remark}
	Since the difference between the translation vectors of any two representatives within the same coset is an integer vector, and $\bm A_T^1$ is an integer matrix, the property $\bm A_T^1\bm t_g \in \mathbb Z^n$ is independent of the specific choice of the representative. 
\end{remark}

\begin{remark}
	Because the quotient group $G$ is finite, there are only finitely many pure translation operations, hence only finitely many rational translation vectors $\bm t_g$; moreover, since $\bm A\bm t_g=\bm z_g\in\mathbb Z^n$, each coordinate $(\bm t_g)_j$ has denominator dividing $N_j$, so if $\kappa_j$ denotes the least common multiple of the denominators of the $j$-th coordinates, then $\kappa_j\mid N_j$. Consequently $\bm A_T^1=\operatorname{diag}(\kappa_1,\dots,\kappa_n)$ satisfies $\bm A_T^1\bm t_g\in\mathbb Z^n$ for every pure translation and is a valid factor in the decomposition $\bm A=\bm A_T^0\bm A_T^1$, with $\det\bm A_T^1=\prod_{j=1}^n\kappa_j$. Thus there exists at least one constructible class of translation-compatible matrices $\bm A_T^1$. For example, when the group includes the body-centered translation $\bm t = (1/2, \dots, 1/2)^\mathsf{T}$ (which forces every $N_j$ to be even), one can choose $\bm A_T^1 = 2\bm I_n$.
\end{remark}

Consequently, any $[\bm x]\in X_{\bm A}$ can be uniquely decomposed as
\begin{equation}\label{eq:x_AT_decomp}
	[\bm x]=[\bm A_T^0\bm x_1+\bm x_0],
	\quad
	[\bm x_0]\in X_{\bm A_T^0},\;[\bm x_1]\in X_{\bm A_T^1},
\end{equation}
which defines the family of subsets
\[
	\Theta_T \coloneqq \{X_T^{[\bm y]}:[\bm y]\in X_{\bm A_T^1}\},
\]
where $X_T^{[\bm y]} \coloneqq \{[\bm A_T^0\bm y+\bm x_0]:[\bm x_0]\in X_{\bm A_T^0}\}$, with $\bm y$ and $\bm x_0$ denoting the standard representatives of their classes.

\begin{proposition}
	\label{prop:layer_permutation_translation}
	Let $\bm A=\bm A_T^0\bm A_T^1$ be the fixed factorization as above, and suppose $\bm A_T^1$ is translation-compatible. Then for any $\Psi_{\bar g}\in G_{\bm A}$ and any $[\bm y]\in X_{\bm A_T^1}$, there exists a unique $[\bm y']\in X_{\bm A_T^1}$ such that
	\[
		\Psi_{\bar g}\bigl(X_T^{[\bm y]}\bigr)=X_T^{[\bm y']}.
	\]
	Specifically, $[\bm y']=[\bm y + \bm A_T^1\bm t_g]\in X_{\bm A_T^1}$, where $g$ is any representative of $\bar g$; this class is independent of the choice of $g$, since for any two representatives $g,g'$ of the same coset we have $\bm t_{g'}-\bm t_g\in\mathbb Z^n$, hence $\bm A_T^1(\bm t_{g'}-\bm t_g)\in\bm A_T^1\mathbb Z^n$ and the two classes coincide in $X_{\bm A_T^1}$.
\end{proposition}

\begin{proof}
	Take any $[\bm y]\in X_{\bm A_T^1}$ and any $[\bm x_0]\in X_{\bm A_T^0}$, and let
	\[
		[\bm x]=[\bm A_T^0\bm y+\bm x_0]\in X_{\bm A}.
	\]
	In the pure translation case, $\bm M_{\bar g}=\bm I_n$. By \cref{prop:discrete_action_formula}, we have
	\[
		\Psi_{\bar g}([\bm x])=[\bm A_T^0\bm y+\bm x_0+\bm z_g].
	\]
	The translation compatibility (\cref{def:translation_compatibility}) implies that $\bm A_T^1\bm t_g\in\mathbb Z^n$, yielding
	\[
		\bm z_g=\bm A\bm t_g=\bm A_T^0\bm A_T^1\bm t_g\in\bm A_T^0\mathbb Z^n.
	\]
	Thus, we can write
	\[
		\Psi_{\bar g}([\bm x])=[\bm A_T^0(\bm y+\bm A_T^1\bm t_g)+\bm x_0].
	\]
	Denote $[\bm y']=[\bm y+\bm A_T^1\bm t_g]\in X_{\bm A_T^1}$. If we take $\bm y'$ as the standard representative of this congruence class, then there exists $\bm k\in\mathbb Z^n$ such that $\bm y+\bm A_T^1\bm t_g = \bm y'+\bm A_T^1\bm k$. Substituting this into the above equation gives
	\[
		\Psi_{\bar g}([\bm x])=[\bm A_T^0(\bm y'+\bm A_T^1\bm k)+\bm x_0]=[\bm A_T^0\bm y'+\bm x_0+\bm A\bm k]=[\bm A_T^0\bm y'+\bm x_0].
	\]
	This indicates that for any $[\bm x_0]\in X_{\bm A_T^0}$, we have
	\[
		\Psi_{\bar g}([\bm A_T^0\bm y+\bm x_0])\in X_T^{[\bm y']}.
	\]
	Consequently,
	\[
		\Psi_{\bar g}\bigl(X_T^{[\bm y]}\bigr)\subseteq X_T^{[\bm y']}.
	\]
	On the other hand, since $\Psi_{\bar g}$ is a bijection, and both $X_T^{[\bm y]}$ and $X_T^{[\bm y']}$ are equinumerous to $X_{\bm A_T^0}$, their cardinalities satisfy
	\[
		\bigl|\Psi_{\bar g}\bigl(X_T^{[\bm y]}\bigr)\bigr|=\bigl|X_T^{[\bm y]}\bigr|=\bigl|X_T^{[\bm y']}\bigr|.
	\]
	Combining this with the inclusion relation, we obtain
	\[
		\Psi_{\bar g}\bigl(X_T^{[\bm y]}\bigr)=X_T^{[\bm y']}.
	\]
	Since the sets within the family of subsets $\Theta_T=\{X_T^{[\bm y]}:[\bm y]\in X_{\bm A_T^1}\}$ are pairwise disjoint, the $[\bm y']$ satisfying the above equation is unique. 
\end{proof}

\subsection{Discrete group action and orbit decomposition}
\label{subsec:orbit_decomposition}

To unify the notation, let $\Theta$ denote the family of subsets $\Theta_R$ or $\Theta_T$, 
and let $\mathcal I_{\Theta}$ denote its index set, corresponding to $X_{\bm A_R^0}$ or $X_{\bm A_T^1}$. 
For any $[\bm y]\in \mathcal I_{\Theta}$, let $X^{[\bm y]}$ denote the subset in $\Theta$ indexed by $[\bm y]$. 
By construction, $\Theta$ forms a partition of $X_{\bm A}$, and
\[
	X_{\bm A}=\bigsqcup_{[\bm y]\in\mathcal I_{\Theta}} X^{[\bm y]}.
\]

For any $[\bm y] \in \mathcal I_{\Theta}$ and its corresponding subset $X^{[\bm y]}$, its \textbf{stabilizer} $G_{[\bm y]} {\leqslant} G_{\bm A}$ is defined as
\begin{align}
	G_{[\bm y]} = \{ \Psi_{\bar g} \in G_{\bm A} : \Psi_{\bar g}(X^{[\bm y]}) = X^{[\bm y]} \}.
\end{align}
Accordingly, the \textbf{orbit} $O_{[\bm{y}]} {\subseteq} \Theta$ of the subset $X^{[\bm y]}$ is defined as
\begin{align}
	O_{[\bm{y}]} = \{ \Psi_{\bar g}(X^{[\bm y]}) : \Psi_{\bar g} \in G_{\bm A} \}.
\end{align}
According to the orbit-stabilizer theorem in group theory (recall that $G_{\bm A}$ is finite), the cardinality of the orbit and the order of the stabilizer satisfy
\[
	|O_{[\bm y]}|=\frac{|G_{\bm A}|}{|G_{[\bm y]}|}.
\]

\begin{proposition}\label{prop:translation_free}
	In the pure translation case (with the factorization $\bm A=\bm A_T^0\bm A_T^1$ and $\bm A_T^1$ translation-compatible, as in \cref{prop:layer_permutation_translation}), for any $[\bm y] \in X_{\bm A_T^1}$, its stabilizer is the trivial group, i.e.,
	\[
		G_{[\bm y]}=\{\mathrm{id}\}.
	\]
\end{proposition}

\begin{proof}
	Fix an arbitrary $[\bm y]\in X_{\bm A_T^1}$. By \cref{prop:layer_permutation_translation}, for each $\Psi_{\bar g}\in G_{\bm A}$, there exists a unique $[\bm y']\in X_{\bm A_T^1}$ such that
	\[
		\Psi_{\bar g}\bigl(X_T^{[\bm y]}\bigr)=X_T^{[\bm y']},
		\qquad
	[\bm y^{\prime}]=[\bm y+\bm A_T^1\bm t_g].
	\]
	Therefore, $\Psi_{\bar g}\in G_{[\bm y]}$ if and only if $[\bm y']=[\bm y]$, which means
	\[
		\bm A_T^1\bm t_g\in \bm A_T^1\mathbb Z^n.
	\]
	Since $\bm A_T^1$ is invertible, this is equivalent to $\bm t_g \in \mathbb Z^n$, whence $\bm z_g=\bm A\bm t_g\in\bm A\mathbb Z^n$ and $\Psi_{\bar g}$ acts as the identity on $X_{\bm A}$, i.e., $\Psi_{\bar g}=\mathrm{id}$. Therefore,
	\[
		G_{[\bm y]}=\{\mathrm{id}\}.
	\]
\end{proof}

Define the \textbf{orbit representative index set} $Y_{\mathrm{ref}} \subset \mathcal I_{\Theta}$, satisfying the following conditions.
\begin{enumerate}
	\item For any distinct $[\bm y_1], [\bm y_2] \in Y_{\mathrm{ref}}$, we have
	\[
	O_{[\bm y_1]} \cap O_{[\bm y_2]} = \emptyset.
	\]
	\item $\Theta$ can be decomposed into the disjoint union of these orbits, i.e.,
	\[
	\Theta = \bigsqcup_{[\bm y] \in Y_{\mathrm{ref}}} O_{[\bm y]}.
	\]
\end{enumerate}

For each $[\bm y] \in Y_{\mathrm{ref}}$, select a system of left coset representatives $\mathcal{R}([\bm y]) \subset G_{\bm A}$ of $G_{\bm A}$ modulo the stabilizer $G_{[\bm y]}$. Although the choice of the left coset representatives is not unique, once $\mathcal{R}([\bm y])$ is fixed, each element in the orbit $O_{[\bm y]}$ can be generated by the action of a unique coset representative $\Psi_{\bar g} \in \mathcal{R}([\bm y])$ on $X^{[\bm y]}$. Therefore,
\begin{align}
	O_{[\bm y]} = \{ \Psi_{\bar g}(X^{[\bm y]}) : \Psi_{\bar g} \in \mathcal{R}([\bm y]) \}.
\end{align}
By the orbit-stabilizer theorem, the number of coset representatives equals the orbit size, i.e., $|\mathcal{R}([\bm y])|=[G_{\bm A}:G_{[\bm y]}] = |O_{[\bm y]}|$. Based on the above construction, $X_{\bm A}$ can be decomposed into the following disjoint union induced by symmetry
\begin{equation}\label{eq:symmetric-partition}
	X_{\bm A} = \bigsqcup_{[\bm y] \in Y_{\mathrm{ref}}} \bigsqcup_{\Psi_{\bar g} \in \mathcal{R}([\bm y])} \Psi_{\bar g} (X^{[\bm y]}).
\end{equation}

The representatives in this partition can be constructed using
\cref{thm:hierarchical_group_action} for the finite permutation group
$G_{\bm A}$. For each representative layer $X^{[\bm y]}$, enumerate its
distinct images under the operations in
\eqref{eq:hierarchical_group_normal_form} and retain one operation for each
image. Two operations give the same image exactly when they belong to the
same left coset of $G_{[\bm y]}$. Hence the retained operations form the
required set $\mathcal R([\bm y])$. This procedure also constructs the orbit
partition and $Y_{\mathrm{ref}}$ by selecting a layer outside the previously
constructed orbits until all layers are covered. In three dimensions,
\cref{cor:three_dimensional_reconstruction} reduces the construction to one
lifted generator per composition step.

\subsection{Algorithm derivation for the nontrivial linear-action case}
\label{sec:nontrivial_linear_derivation}

This section derives the DFT symmetry reduction formula for the SFFT under the setting of \cref{case:nontrivial_linear}. The derivation proceeds in three steps: (i) employing the Cooley--Tukey decomposition to rewrite the global DFT as a layered sum over layer indices; (ii) utilizing orbit decomposition to recast the outer sum into a double sum over orbit representatives and coset representatives, followed by simplifying the outer phase factor; and (iii) applying a change of variables to the inner sum and exploiting symmetry inheritance to reduce it to a local DFT with corresponding phase corrections.

Let the decomposition of the discretization matrix be $\bm A = \bm A_R^0 \bm A_R^1$. Based on this decomposition, any frequency index $[\bm h] \in X_{\bm A}$ can be uniquely decomposed as
\[
	[\bm h] = [\bm A_R^1 \bm h_0 + \bm h_1],\qquad [\bm h_0]\in X_{\bm A_R^0},\quad [\bm h_1]\in X_{\bm A_R^1}.
\]

\paragraph{Step 1: Cooley--Tukey layered expansion}
Substituting the factorization $\bm A=\bm A_R^0\bm A_R^1$ from \cref{case:nontrivial_linear} into the multidimensional Cooley--Tukey decomposition formula \cref{eq:cooley-tukey-decomp} (with the identifications $\bm A_R^0\leftrightarrow\bm A_0$, $\bm A_R^1\leftrightarrow\bm A_1$, and the outer index $[\bm y]\leftrightarrow[\bm x_0]$) yields
\begin{equation}
	\begin{aligned}
		\label{eq:linear_action_reduce}
		\hat{f}_{\bm A,\bm s}([\bm{h}])
		 & =\sum_{[\bm{y}]\in X_{\bm A_R^0}}e^{-2\pi i\bm{h}_0^\mathsf{T}(\bm{A}_R^0)^{-1}\bm{y}}\left[\sum_{[\bm{x}_1]\in X_{\bm A_R^1}}f_{\bm A,\bm s}([\bm A_R^0 \bm x_1 + \bm{y}])e^{-2\pi i\bm{h}_1^\mathsf{T}\bm{A}^{-1}(\bm A_R^0 \bm x_1 + \bm{y})}\right] \\
		 & =\sum_{[\bm{y}]\in X_{\bm A_R^0}}e^{-2\pi i\bm{h}_0^\mathsf{T}(\bm{A}_R^0)^{-1}\bm{y}}\left[\sum_{[\bm{x}]\in X_{R}^{[\bm y]}}f_{\bm A,\bm s}([\bm{x}])e^{-2\pi i\bm{h}_1^\mathsf{T}\bm{A}^{-1}\bm{x}}\right],
	\end{aligned}
\end{equation}
where the second equality utilizes the bijection between $X_{\bm A_R^1}$ and $X_R^{[\bm y]}$ given by the mapping $[\bm x_1] \mapsto [\bm x] = [\bm A_R^0 \bm x_1 + \bm y]$.

\paragraph{Step 2: orbit decomposition and outer phase simplification}
Combining \cref{subsec:orbit_decomposition} and \cref{prop:layer_permutation}, for any $[\bm y]\in X_{\bm A_R^0}$, there exist a unique $[\bm y_{\mathrm{ref}}]\in Y_{\mathrm{ref}}$ and  a unique $\Psi_{\bar g}\in\mathcal{R}([\bm y_{\mathrm{ref}}])$ such that
\begin{align}
	X_R^{[\bm y]} =\Psi_{\bar g}\left(X_R^{[\bm y_{\mathrm{ref}}]}\right),
	\quad 
	[\bm y] = [\bm M_{\bar g} \bm y_{\mathrm{ref}} + \bm z_{g}] \in X_{\bm A_R^0}.
\end{align}
Thus, the inner summation domain can be rewritten as $[\bm x]\in \Psi_{\bar g}(X_R^{[\bm y_{\mathrm{ref}}]})$.
Furthermore, there exists $\bm k \in \mathbb{Z}^n$ such that
\begin{align}
	\bm y = \bm M_{\bar g}\bm y_{\mathrm{ref}} + \bm z_g + \bm A_R^0 \bm k,
\end{align}
where $g\in\bar g$ is a fixed representative, so that $\bm y$ and $\bm y_{\mathrm{ref}}$ are fixed integer representatives of their respective classes. Substituting this into the outer phase factor gives
\begin{align}
	e^{-2\pi i\bm{h}_0^\mathsf{T}(\bm{A}_R^0)^{-1}\bm{y}}
	&= e^{-2\pi i\bm{h}_0^\mathsf{T}(\bm{A}_R^0)^{-1}(\bm{M}_{\bar g}\bm{y}_{\mathrm{ref}}+\bm{z}_{g}+\bm A_R^0\bm k)} \notag \\
	&= e^{-2\pi i\bm{h}_0^\mathsf{T}(\bm{A}_R^0)^{-1}(\bm{M}_{\bar g}\bm{y}_{\mathrm{ref}}+\bm{z}_{g})} \, e^{-2\pi i\bm{h}_0^\mathsf{T}\bm k} \notag \\
	&= e^{-2\pi i\bm{h}_0^\mathsf{T}(\bm{A}_R^0)^{-1}(\bm{M}_{\bar g}\bm{y}_{\mathrm{ref}}+\bm{z}_{g})},
\end{align}
where the middle factor equals $1$ since $\bm h_0\in\mathbb Z^n$ and $\bm k\in\mathbb Z^n$, so $\bm h_0^\mathsf T\bm k\in\mathbb Z$.
Therefore, \cref{eq:linear_action_reduce} can be rewritten as
\begin{equation} \label{eq:linear_action_outer_reduced}
	\hat{f}_{\bm A,\bm s}([\bm h]) = \sum_{[\bm{y}_{\mathrm{ref}}]\in Y_{\mathrm{ref}}} \sum_{\Psi_{\bar g} \in \mathcal{R}([\bm y_{\mathrm{ref}}])} e^{-2\pi i\bm{h}_0^\mathsf{T}(\bm{A}_R^0)^{-1}(\bm{M}_{\bar g}\bm{y}_{\mathrm{ref}}+\bm{z}_{g})} \cdot W_{(\Psi_{\bar g}, \bm{y}_{\mathrm{ref}})}([\bm{h}_1]),
\end{equation}
where
\begin{align}
	\label{eq:W_definition}
	W_{(\Psi_{\bar g},\bm y_{\mathrm{ref}})}([\bm h_1]) \coloneqq \sum_{[\bm{x}]\in \Psi_{\bar g}(X_{R}^{[\bm{y}_{\mathrm{ref}}]})}f_{\bm A,\bm s}([\bm{x}])e^{-2\pi i\bm{h}_1^\mathsf{T}\bm{A}^{-1}\bm{x}}.
\end{align}

\paragraph{Step 3: simplification of the inner sum}
Apply the change of variables $[\bm x] = \Psi_{\bar g}([\bm x_{\mathrm{ref}}])$ to $W_{(\Psi_{\bar g},\bm y_{\mathrm{ref}})}([\bm h_1])$, where $[\bm x_{\mathrm{ref}}]\in X_R^{[\bm y_{\mathrm{ref}}]}$. From \cref{prop:symmetry_inheritance}, we have $f_{\bm A,\bm s}(\Psi_{\bar g}([\bm x_{\mathrm{ref}}])) = f_{\bm A,\bm s}([\bm x_{\mathrm{ref}}])$. By the definition of the discrete group action, there exists a unique $\bm k\in\mathbb{Z}^n$ such that
\[
	\bm x = \bm M_{\bar g} \bm x_{\mathrm{ref}} + \bm z_g + \bm A \bm k.
\]
Substituting this into the phase factor yields
\begin{align}
	e^{-2\pi i\bm{h}_1^\mathsf{T}\bm{A}^{-1}\bm x}
	 & = e^{-2\pi i\bm{h}_1^\mathsf{T}\bm{A}^{-1}(\bm M_{\bar g}\bm x_{\mathrm{ref}} + \bm z_{g})}.
\end{align}
Utilizing the representative $\bm x_{\mathrm{ref}} = \bm A_R^0 \bm x_1 + \bm y_{\mathrm{ref}} + \bm A \bm k'$ of the class $[\bm x_{\mathrm{ref}}] = [\bm A_R^0 \bm x_1 + \bm y_{\mathrm{ref}}]$ (with $\bm k'\in\mathbb Z^n$), the exponent expands to
\begin{align}
	&\bm{h}_1^\mathsf{T}\bm{A}^{-1}(\bm M_{\bar g}(\bm A_R^0 \bm x_1 + \bm y_{\mathrm{ref}} + \bm A\bm k') + \bm z_{g})\\
	 =&  \bm{h}_1^\mathsf{T}\bm{A}^{-1}(\bm M_{\bar g}\bm y_{\mathrm{ref}} + \bm z_{g}) + \bm{h}_1^\mathsf{T}\bm{A}^{-1}\bm M_{\bar g}\bm A_R^0\bm x_1 + \bm{h}_1^\mathsf{T}\bm{A}^{-1}\bm M_{\bar g}\bm A\bm k' \notag \\
	 =&  \bm{h}_1^\mathsf{T}\bm{A}^{-1}(\bm M_{\bar g}\bm y_{\mathrm{ref}} + \bm z_{g}) + \bm{h}_1^\mathsf{T}\bm{A}^{-1}\bm M_{\bar g}\bm A_R^0\bm x_1 + \bm{h}_1^\mathsf{T}\bm R_{\bar g}\bm k',
\end{align}
where the last term is an integer, $\bm h_1^\mathsf T\bm R_{\bar g}\bm k'\in\mathbb Z$, so its contribution to the overall phase factor is $e^{-2\pi i\bm h_1^\mathsf T\bm R_{\bar g}\bm k'}=1$ and may be dropped.
The second term can be further simplified as
\begin{align}
	\bm{h}_1^\mathsf{T}\bm{A}^{-1}\bm M_{\bar g}\bm A_R^0\bm x_1
	 & = \bm{h}_1^\mathsf{T} \bm R_{\bar g} (\bm A^{-1} \bm A_R^0) \bm x_1 \notag                  \\
	 & = \bm{h}_1^\mathsf{T} \bm R_{\bar g} (\bm A_R^1)^{-1} \bm x_1 \notag                        \\
	 & = (\bm R_{\bar g}^\mathsf{T} \bm h_1)^\mathsf{T} (\bm A_R^1)^{-1} \bm x_1.
\end{align}
Consequently,
\begin{align}
	\label{eq:W_decomposition}
	W_{(\Psi_{\bar g},\bm y_{\mathrm{ref}})}([\bm h_1])
	 & = e^{-2\pi i\bm{h}_1^\mathsf{T}\bm{A}^{-1}(\bm M_{\bar g}\bm y_{\mathrm{ref}} + \bm z_{g})} \hat{V}_{\bm y_{\mathrm{ref}}}([\bm R_{\bar g}^\mathsf{T}\bm h_1]),
\end{align}
where 
\begin{align}
	\label{eq:local_dft_definition}
	\hat{V}_{\bm{y}_{\mathrm{ref}}}([\bm h_1]) \coloneqq \sum_{[\bm{x}_1] \in X_{\bm A_R^1}} f_{\bm A,\bm s}([\bm{A}_R^0 \bm{x}_1 + \bm y_{\mathrm{ref}}]) e^{-2\pi i\bm h_1^\mathsf{T}(\bm{A}_R^1)^{-1}\bm{x}_1}.
\end{align}

Combining the three derivation steps above, we obtain the DFT symmetry reduction formula for the nontrivial linear-action case
\begin{equation} \label{eq:linear_action_final}
	\begin{aligned}
		\hat{f}_{\bm A,\bm s}([\bm h]) = \sum_{[\bm{y}_{\mathrm{ref}}]\in Y_{\mathrm{ref}}} \sum_{\Psi_{\bar g} \in \mathcal{R}([\bm y_{\mathrm{ref}}])}
		 & e^{-2\pi i\bm{h}_0^\mathsf{T}\left(\bm{A}_R^0\right)^{-1}\left(\bm{M}_{\bar g}\bm{y}_{\mathrm{ref}}+\bm{z}_{g}\right)}                                                             \\
		 & \cdot e^{-2\pi i\bm{h}_1^\mathsf{T}\bm{A}^{-1}\left(\bm M_{\bar g}\bm y_{\mathrm{ref}} + \bm z_{g}\right)} \cdot \hat{V}_{\bm y_{\mathrm{ref}}}\left([\bm R_{\bar g}^\mathsf{T}\bm h_1]\right).
	\end{aligned}
\end{equation}

\subsubsection{Matrix representation for the nontrivial linear-action case}
\label{sec:algorithm_matrix}

To facilitate the algorithmic implementation, this section first unifies the indexing conventions.

\begin{definition}[Lexicographical order]
	Let $\bm x=(x_1,\dots,x_n),\bm y=(y_1,\dots,y_n)\in\mathbb Z^n$. The lexicographical order is defined as
	\[
		\bm x\prec_{\mathrm{lex}}\bm y
		\iff
		\exists j\in\{1,\dots,n\},\ x_1=y_1,\dots,x_{j-1}=y_{j-1},\ x_j<y_j.
	\]
	For a finite set $S\subset\mathbb Z^n$, let $S^{\mathrm{lex}}=(\bm s_0,\dots,\bm s_{|S|-1})$ denote the unique sequence obtained by arranging its elements in increasing order according to $\prec_{\mathrm{lex}}$ (which is a total order on $\mathbb Z^n$).
\end{definition}

Based on this, we establish the following conventions.
\begin{enumerate}
	\item \textbf{Ordering of orbit representatives}. Let $Y_{\mathrm{ref}}^{\mathrm{lex}} = (\bm y_{0}, \dots, \bm y_{\ell - 1})$ be the lexicographical arrangement of the set $\{\bm y\in \mathcal D_{\bm A_R^0}: [\bm y]\in Y_{\mathrm{ref}}\}$, where $\ell = |Y_{\mathrm{ref}}|$.
	\item \textbf{Indexing of coset representatives}. For each $j =0,\dots,\ell-1$, we fix a system of coset representatives
    \[
      \mathcal{R}([\bm y_j]) = (\Psi_{\bar g_{j0}}, \dots, \Psi_{\bar g_{j,r_j-1}}),\qquad \bar g_{jc}\in G,
    \]
    where $r_j = |\mathcal{R}([\bm y_j])|$, with the convention that $\Psi_{\bar g_{j0}} = \mathrm{id}$.
	\item \textbf{Ordering of double indices}. Define
    \[
        \mathcal I \coloneqq \{(j,c): 0 \le j < \ell,\ 0 \le c < r_j\}\subset \mathbb{Z}^2,
    \] 
    and arrange it according to the lexicographical order on $\mathbb{Z}^2$ as
    \[
        \mathcal I^{\mathrm{lex}} = ((j_0,c_0), \ldots, (j_{D_0-1},c_{D_0-1})),
    \]
    where $D_0 = |X_{\bm A_R^0}|$. Hereafter, $v$ denotes the positional index of the sequence $\mathcal I^{\mathrm{lex}}$, i.e., the $v$-th double index is denoted by $(j_v,c_v)$.
\end{enumerate}

We first present two key lemmas.

\begin{lemma}\label{lem:phase_invariance}
	For any $g_1,g_2\in \bar g$, and any integer representatives $\bm y,\bm h_0\in\mathbb Z^n$ of classes in $X_{\bm A_R^{0}}$, and $\bm h_1\in\mathbb Z^n$ of a class in $X_{\bm A_R^{1}}$, it holds that
	\begin{align}
		e^{-2\pi i \bm h_0^\mathsf{T} (\bm A_R^0)^{-1} (\bm M_{\bar g_1} \bm y + \bm z_{g_1})} & = e^{-2\pi i \bm h_0^\mathsf{T} (\bm A_R^0)^{-1} (\bm M_{\bar g_2} \bm y + \bm z_{g_2})}, \\
		e^{-2\pi i\bm h_1^\mathsf{T}\bm A^{-1} (\bm M_{\bar g_1} \bm y + \bm z_{g_1})}         & = e^{-2\pi i\,\bm h_1^\mathsf{T}\bm A^{-1} (\bm M_{\bar g_2} \bm y + \bm z_{g_2})}.
	\end{align}
\end{lemma}

\begin{proof}
	Let $\bm y,\bm h_0,\bm h_1\in\mathbb Z^n$ be arbitrary integer representatives of their respective congruence classes. 
	Since $g_1,g_2\in\bar g$, we have
	\[
		\bm M_{\bar g_1} = \bm M_{\bar g_2} = \bm M_{\bar g}, \qquad [\bm z_{g_1}]=[\bm z_{g_2}]\in X_{\bm A}.
	\]
	Thus, there exists $\bm k\in\mathbb Z^n$ such that
	\[
		\bm z_{g_1}-\bm z_{g_2}=\bm A\bm k=\bm A_R^0\bm A_R^1\bm k.
	\]
	Then
	\begin{align*}
		\bm h_0^\mathsf{T}(\bm A_R^0)^{-1}(\bm M_{\bar g_1}\bm y+\bm z_{g_1})
		 & =\bm h_0^\mathsf{T}(\bm A_R^0)^{-1}(\bm M_{\bar g_2}\bm y+\bm z_{g_2})
		+\bm h_0^\mathsf{T}\bm A_R^1\bm k.
	\end{align*}
	Since $\bm h_0\in\mathbb Z^n$ and $\bm A_R^1\bm k\in\mathbb Z^n$, we have $\bm h_0^\mathsf{T}\bm A_R^1\bm k\in\mathbb Z$, which implies
	\[
		e^{-2\pi i \bm h_0^\mathsf{T} (\bm A_R^0)^{-1}(\bm M_{\bar g_1}\bm y+\bm z_{g_1})}
		= e^{-2\pi i \bm h_0^\mathsf{T} (\bm A_R^0)^{-1}(\bm M_{\bar g_2}\bm y+\bm z_{g_2})}.
	\]
	Similarly,
	\begin{align*}
		\bm h_1^\mathsf{T}\bm A^{-1}(\bm M_{\bar g_1}\bm y+\bm z_{g_1})
		 & =\bm h_1^\mathsf{T}\bm A^{-1}(\bm M_{\bar g_2}\bm y+\bm z_{g_2})+\bm h_1^\mathsf{T}\bm k,
	\end{align*}
	and since $\bm h_1^\mathsf{T}\bm k\in\mathbb Z$, we obtain
	\[
	e^{-2\pi i\,\bm h_1^\mathsf{T}\bm A^{-1} (\bm M_{\bar g_1} \bm y + \bm z_{g_1})}= e^{-2\pi i\,\bm h_1^\mathsf{T}\bm A^{-1} (\bm M_{\bar g_2} \bm y + \bm z_{g_2})}.
	\]
\end{proof}

\begin{lemma}\label{lem:bijection}
	In the nontrivial linear-action case, define the mapping $\Phi: \mathcal{I} \to X_{\bm A_R^0}$ as
	\begin{align}
		\Phi(j,c) \coloneqq [\bm M_{\bar g_{jc}} \bm y_j + \bm z_{g_{jc}}].
	\end{align}
	Then $\Phi$ is a bijection.
\end{lemma}

\begin{proof}
In the nontrivial linear-action case, we have $\Theta=\Theta_R$ and $X^{[\bm y]}=X_R^{[\bm y]}$. The unique decomposition in \cref{eq:x_AR_decomp} implies that the mapping
\[
\iota:\ X_{\bm A_R^0}\to\Theta,\qquad [\bm y]\mapsto X^{[\bm y]},
\]
is a bijection.

From the orbit decomposition in \cref{subsec:orbit_decomposition} and the fixed system of left coset representatives $\mathcal{R}([\bm y_j])$ for each orbit representative $[\bm y_j]\in Y_{\mathrm{ref}}$, we have the disjoint union decomposition
\[
\Theta = \bigsqcup_{j=0}^{\ell-1} \bigsqcup_{c=0}^{r_j-1} \{ \Psi_{\bar g_{jc}}(X^{[\bm y_j]}) \}.
\]
Furthermore, by \cref{prop:layer_permutation},
\[
\Psi_{\bar g_{jc}}(X^{[\bm y_j]}) = X^{[\bm M_{\bar g_{jc}}\bm y_j+\bm z_{g_{jc}}]} = X^{\Phi(j,c)}.
\]
Thus,
\[
\Theta = \bigsqcup_{(j,c)\in\mathcal I} \{X^{\Phi(j,c)}\}.
\]

\textbf{Injectivity}. If $\Phi(j,c)=\Phi(j',c')$, then $X^{\Phi(j,c)}=X^{\Phi(j',c')}$. Combined with the disjoint union decomposition above, it must hold that $(j,c)=(j',c')$. Hence, $\Phi$ is injective.

\textbf{Surjectivity}. For any $[\bm y]\in X_{\bm A_R^0}$, we have $X^{[\bm y]}\in\Theta$. By the above disjoint union decomposition, there exists some $(j,c)\in\mathcal I$ such that $X^{[\bm y]}=X^{\Phi(j,c)}$. Since $\iota$ is injective, we obtain $[\bm y]=\Phi(j,c)$, making $\Phi$ surjective.

In conclusion, $\Phi$ is a bijection.
\end{proof}

Let $X_{\bm A_R^0}^{\mathrm{lex}} = (\bm h_0^{(0)},\dots,\bm h_0^{(D_0-1)})$ denote its lexicographical arrangement. 

\begin{definition}\label{def:linear_action_matrix}
	For a fixed $\bm h_1 \in \mathcal D_{\bm A_R^{1}}$, define the following vectors and matrices.
	\begin{enumerate}
		\item The global Fourier coefficient vector $\hat{\bm F}(\bm h_1)\in\mathbb{C}^{D_0}$, whose $p$-th component is given by
		    \[
			(\hat{\bm F}(\bm h_1))_p \coloneqq \hat{f}_{\bm A,\bm s}([\bm A_R^1 \bm h_0^{(p)} + \bm h_1]).
			\]
		\item The local transform vector $\hat{\bm V}(\bm h_1)\in\mathbb{C}^{D_0}$, whose $v$-th component is given by
		    \[
		    (\hat{\bm V}(\bm h_1))_v \coloneqq \hat{V}_{\bm y_{j_v}}([\bm R_{\bar g_{j_v c_v}}^\mathsf{T} \bm h_1]).
			\] 
		\item The constant matrix $\bm T_{c,r}\in\mathbb{C}^{D_0\times D_0}$, with entries
		      \[
		      (\bm T_{c,r})_{p,v} = e^{-2\pi i (\bm h_0^{(p)})^\mathsf{T} (\bm A_R^0)^{-1} (\bm M_{\bar g_{j_v c_v}} \bm y_{j_v} + \bm z_{g_{j_v c_v}})}.
		      \]
		\item The phase diagonal matrix $\bm T_{s,r}(\bm h_1)\in\mathbb{C}^{D_0\times D_0}$, with diagonal entries
		      \[
		      (\bm T_{s,r}(\bm h_1))_{v,v} = e^{-2\pi i \bm h_1^\mathsf{T} \bm A^{-1} (\bm M_{\bar g_{j_v c_v}} \bm y_{j_v} + \bm z_{g_{j_v c_v}})}.
		      \]
	\end{enumerate}
\end{definition}

The entries of these vectors and matrices are well defined, since by \cref{lem:phase_invariance} the phase factors do not depend on the choice of the coset representatives. This leads to the matrix relation
\begin{equation}\label{eq:core_matrix_relation_linear_action}
	\hat{\bm F}(\bm h_1) = \bm T_{c,r}\,\bm T_{s,r}(\bm h_1)\,\hat{\bm V}(\bm h_1).
\end{equation}

\begin{theorem}
	\label{thm:Tc_inv}
	The row vectors of the matrix $\bm T_{c,r}$ are mutually orthogonal, and it satisfies 
	\[
		\bm T_{c,r} \bm T_{c,r}^* = D_0 \bm I_{D_0}.
	\]
	Consequently, 
	\[
	\bm T_{c,r}^{-1} = \frac{1}{D_0} \bm T_{c,r}^*.
	\]
\end{theorem}

\begin{proof}
	For any row indices $p_1,p_2\in\{0,\dots,D_0-1\}$,
	\begin{align*}
		  & (\bm T_{c,r}\bm T_{c,r}^*)_{p_1,p_2}\\
		= & \sum_{v=0}^{D_0-1} e^{-2\pi i (\bm h_0^{(p_1)})^\mathsf{T}(\bm A_R^0)^{-1}(\bm M_{\bar g_{j_v c_v}}\bm y_{j_v}+\bm z_{g_{j_v c_v}})}\, e^{2\pi i (\bm h_0^{(p_2)})^\mathsf{T}(\bm A_R^0)^{-1}(\bm M_{\bar g_{j_v c_v}}\bm y_{j_v}+\bm z_{g_{j_v c_v}})} \\
		= & \sum_{v=0}^{D_0-1} e^{-2\pi i (\bm h_0^{(p_1)}-\bm h_0^{(p_2)})^\mathsf{T}(\bm A_R^0)^{-1}(\bm M_{\bar g_{j_v c_v}}\bm y_{j_v}+\bm z_{g_{j_v c_v}})}.
	\end{align*}
	Let $\bm \eta_v\in\mathcal D_{\bm A_R^0}$ be the standard integer representative of $\Phi(j_v,c_v)=[\bm M_{\bar g_{j_v c_v}}\bm y_{j_v}+\bm z_{g_{j_v c_v}}]\in X_{\bm A_R^0}$. The above equation can be rewritten as
	\[
		(\bm T_{c,r}\bm T_{c,r}^*)_{p_1,p_2}=\sum_{v=0}^{D_0-1} e^{-2\pi i (\bm h_0^{(p_1)}-\bm h_0^{(p_2)})^\mathsf{T}(\bm A_R^0)^{-1}\bm \eta_v}.
	\]
	By \cref{lem:bijection}, we have
	\[
		(\bm T_{c,r}\bm T_{c,r}^*)_{p_1,p_2}=\sum_{\bm y\in\mathcal D_{\bm A_R^0}} e^{-2\pi i (\bm h_0^{(p_1)}-\bm h_0^{(p_2)})^\mathsf{T}(\bm A_R^0)^{-1}\bm y}.
	\]
	Utilizing the one-to-one correspondence between $\mathcal D_{\bm A_R^0}$ and $X_{\bm A_R^0}$, we denote
	\[
		\chi_{p_1,p_2}([\bm y])
		\coloneqq
		e^{-2\pi i (\bm h_0^{(p_1)}-\bm h_0^{(p_2)})^\mathsf{T}(\bm A_R^0)^{-1}\bm y},
		\quad [\bm y]\in X_{\bm A_R^0}.
	\]
	Since $\bm h_0^{(p_1)}-\bm h_0^{(p_2)}\in\mathbb Z^n$, replacing $\bm y$ by $\bm y+\bm A_R^0\bm k$ ($\bm k\in\mathbb Z^n$) changes the right-hand side by the factor $e^{-2\pi i(\bm h_0^{(p_1)}-\bm h_0^{(p_2)})^\mathsf{T}\bm k}=1$. Hence $\chi_{p_1,p_2}$ is a well-defined character (group homomorphism) on $X_{\bm A_R^0}$, and in particular $\chi_{p_1,p_2}([\bm y]+[\bm y'])=\chi_{p_1,p_2}([\bm y])\chi_{p_1,p_2}([\bm y'])$.
	Then
	\[
		(\bm T_{c,r}\bm T_{c,r}^*)_{p_1,p_2}
		=
		\sum_{[\bm y]\in X_{\bm A_R^0}}\chi_{p_1,p_2}([\bm y]).
	\]
	When $p_1=p_2$,
	\[
		(\bm T_{c,r}\bm T_{c,r}^*)_{p_1,p_2}=D_0.
	\]
    When $p_1\neq p_2$, we have $[\bm h_0^{(p_1)}]\neq[\bm h_0^{(p_2)}]$ in $X_{\bm A_R^0}$, which implies $\bm w\coloneqq\bm h_0^{(p_1)}-\bm h_0^{(p_2)}\notin \bm A_R^0\mathbb Z^n$. If $\chi_{p_1,p_2}([\bm y])\equiv 1$ for all $[\bm y]\in X_{\bm A_R^0}$, then in particular for each standard basis class $[\bm y]=[\bm e_i]$ ($i=1,\dots,n$) we would have
	\[
		\chi_{p_1,p_2}([\bm e_i])
		= e^{-2\pi i\,\bm w^{\mathsf{T}}(\bm A_R^0)^{-1}\bm e_i}
		= 1,
	\]
	whence $\bm w^{\mathsf{T}}(\bm A_R^0)^{-1}\bm e_i\in\mathbb Z$ for every $i$, i.e., $(\bm A_R^0)^{-1}\bm w\in\mathbb Z^n$, which is equivalent to $\bm w\in\bm A_R^0\mathbb Z^n$, a contradiction. Therefore $\chi_{p_1,p_2}$ is a nontrivial character, and there exists $[\bm y_*]\in X_{\bm A_R^0}$ such that
	$\chi_{p_1,p_2}([\bm y_*])\neq 1$.
	Let
	\[
		S\coloneqq\sum_{[\bm y]\in X_{\bm A_R^0}}\chi_{p_1,p_2}([\bm y]).
	\]
	Since the mapping $[\bm y]\mapsto [\bm y]+[\bm y_*]$ is a bijection on $X_{\bm A_R^0}$, it follows that
	\[
		S
		=\sum_{[\bm y]\in X_{\bm A_R^0}}\chi_{p_1,p_2}([\bm y]+[\bm y_*])
		=\chi_{p_1,p_2}([\bm y_*])\sum_{[\bm y]\in X_{\bm A_R^0}}\chi_{p_1,p_2}([\bm y])
		=\chi_{p_1,p_2}([\bm y_*])S.
	\]
	Thus, $(1-\chi_{p_1,p_2}([\bm y_*]))S=0$. Since $\chi_{p_1,p_2}([\bm y_*])\neq 1$, we obtain $S=0$, i.e., $(\bm T_{c,r}\bm T_{c,r}^*)_{p_1,p_2}=0$.
	Therefore,
	\[
		\bm T_{c,r}\bm T_{c,r}^*=D_0\bm I_{D_0}.
	\]
	Consequently, $\bm T_{c,r}$ is invertible, and
	\[
		\bm T_{c,r}^{-1}=\frac{1}{D_0}\bm T_{c,r}^*.
	\]
\end{proof}

Before applying \cref{alg:SFFT-LIN,alg:SIFFT-LIN}, construct
$G_{\bm A}$ from a composition series and lifted factor generators as in
\cref{thm:hierarchical_group_action}, and form $Y_{\mathrm{ref}}$ and
$\mathcal R([\bm y])$ by the procedure following
\eqref{eq:symmetric-partition}. The discrete affine operations are composed
using
\[
    \bm M_{\overline{gh}}=\bm M_{\bar g}\bm M_{\bar h},\qquad
    \bm z_{gh}\equiv\bm M_{\bar g}\bm z_h+\bm z_g
    \pmod{\bm A\mathbb Z^n}.
\]
These identities follow by composing the induced affine maps on
$X_{\bm A}$. If different words induce the same permutation, retain one
representative of that operation. The resulting operations associated with
$\mathcal R([\bm y])$ supply the pairs
$(\bm M_{\bar g_{jc}},\bm z_{g_{jc}})$ in the algorithm inputs.
This preprocessing depends on the symmetry group, grid and shift, and can
be reused for transforms with the same setting. Based on the matrix
representation above, the SFFT and SIFFT algorithms for the nontrivial
linear-action case are presented below.

\begin{algorithm}[H]
	\renewcommand{\algorithmicrequire}{\textbf{Input}}
	\renewcommand{\algorithmicensure}{\textbf{Output}}
	\caption{SFFT for the nontrivial linear-action case}
	\label{alg:SFFT-LIN}
	\begin{algorithmic}[1]
		\REQUIRE $\bm{A}_R^0$, $\bm{A}_R^1$, $\bm{A}$, $Y_{\mathrm{ref}}^{\mathrm{lex}}$, discrete function values $\{f_{\bm A,\bm s}([\bm A_R^0 \bm x_1 + \bm y_j])\}_{j=0}^{\ell-1}$ on the orbit representatives $\bm y_j\in Y_{\mathrm{ref}}^{\mathrm{lex}}$, set of symmetry operations $\{(\bm M_{\bar g_{jc}}, \bm z_{g_{jc}})\}_{(j,c)\in\mathcal I^{\mathrm{lex}}}$.
		\ENSURE $\hat{\bm F}(\bm h_1)$ for all $\bm h_1\in\mathcal{D}_{\bm A_R^1}$.
% 		\STATE Precompute all $\bm M_{\bar g_{jc}}$ and $\bm z_{g_{jc}}$.
		\STATE Precompute the constant matrix $\bm T_{c,r}$.
		\FOR{each $\bm y_j \in Y_{\mathrm{ref}}^{\mathrm{lex}}$}
			\STATE $\hat{V}_{\bm y_j}([\bm h_1]) \leftarrow \operatorname{FFT}_{\bm x_1}\bigl[f_{\bm A,\bm s}([\bm A_R^0 \bm x_1 + \bm y_j])\bigr],\ \forall \bm h_1$
		\ENDFOR
		\FOR{each $\bm h_1 \in \mathcal{D}_{\bm A_R^1}$}
			\STATE Compute the diagonal matrix $\bm T_{s,r}(\bm h_1)$
			\STATE Assemble $\hat{\bm V}(\bm h_1)$ according to the double indices $(j_v,c_v)$
			\STATE $\hat{\bm F}(\bm h_1) \leftarrow \bm T_{c,r} \, \bm T_{s,r}(\bm h_1) \, \hat{\bm V}(\bm h_1)$
		\ENDFOR
	\end{algorithmic}
\end{algorithm}

\begin{algorithm}[H]
	\renewcommand{\algorithmicrequire}{\textbf{Input}}
	\renewcommand{\algorithmicensure}{\textbf{Output}}
	\caption{SIFFT for the nontrivial linear-action case}
	\label{alg:SIFFT-LIN}
	\begin{algorithmic}[1]
		\REQUIRE $\bm A_R^0$, $\bm A_R^1$, $\bm A$, $Y_{\mathrm{ref}}^{\mathrm{lex}}$, Fourier coefficients $\hat{\bm F}(\bm h_1)$, set of symmetry operations $\{(\bm M_{\bar g_{jc}}, \bm z_{g_{jc}})\}_{(j,c)\in\mathcal I^{\mathrm{lex}}}$.
		\ENSURE $f_{\bm A,\bm s}([\bm A_R^0 \bm x_1 + \bm y_j])$ for all $\bm y_j\in Y_{\mathrm{ref}}^{\mathrm{lex}}$.
		\STATE Precompute $\bm T_{c,r}^*$ and the index sequence $\nu_j = \sum_{k=0}^{j-1} r_k$.
		\FOR{each $\bm h_1 \in \mathcal{D}_{\bm A_R^1}$}
			\STATE Compute the diagonal matrix $\bm T_{s,r}^*(\bm h_1)$
			\FOR{each $j=0,\dots,\ell-1$}
				\STATE 
				\[
					\hat{V}_{\bm y_j}([\bm h_1]) \leftarrow \frac{1}{D_0} \bigl(\bm T_{s,r}^*(\bm h_1)\bigr)_{\nu_j,\nu_j} 
					\cdot \bigl[(\bm T_{c,r}^*)_{\nu_j,:} \cdot \hat{\bm F}(\bm h_1)\bigr]
				\]
			\ENDFOR
		\ENDFOR
		\FOR{each $\bm y_j \in Y_{\mathrm{ref}}^{\mathrm{lex}}$}
			\STATE $f_{\bm A,\bm s}([\bm A_R^0 \bm x_1 + \bm y_j]) \leftarrow \operatorname{IFFT}_{\bm h_1}\bigl[\hat{V}_{\bm y_j}([\bm h_1])\bigr]$
		\ENDFOR
	\end{algorithmic}
\end{algorithm}

In Step 1 of \cref{alg:SIFFT-LIN}, the sequence $\nu_j=\sum_{k=0}^{j-1}r_k$ gives the first position of orbit $j$ in $\mathcal I^{\mathrm{lex}}$ under the indexing conventions of \cref{sec:algorithm_matrix}: each orbit occupies a contiguous block of $r_j$ positions, and its coset representatives are ordered so that the first one is $\mathrm{id}$. Consequently the $\nu_j$-th component of $\hat{\bm V}(\bm h_1)$ is exactly $\hat V_{\bm y_j}([\bm h_1])$.

\subsection{Algorithm derivation for the pure translation case}
\label{sec:translation_derivation}

This section derives the DFT translation symmetry reduction formula for the SFFT under the setting of \cref{case:pure_translation}. The derivation proceeds in three steps: (i) obtaining a hierarchical summation form via the Cooley--Tukey decomposition; (ii) decoupling the inner summation into the product of a local term and a sum of phase factors using orbit decomposition and translation symmetry; and (iii) filtering effective frequencies by analyzing the behavior of the phase factor sum, which ultimately yields the final DFT reduction formula.

Let the decomposition of the discretization matrix be $\bm A = \bm A_T^0 \bm A_T^1$. Based on this decomposition, any frequency index $[\bm h]\in X_{\bm A}$ can be uniquely decomposed as
\[
[\bm h]=[\bm A_T^1\bm h_0+\bm h_1], \qquad [\bm h_0]\in X_{\bm A_T^0}, \quad [\bm h_1]\in X_{\bm A_T^1}.
\]
This is the same frequency-index decomposition as in \cref{eq:cooley-tukey-decomp}, with the identifications $\bm A_0\leftrightarrow\bm A_T^0$ and $\bm A_1\leftrightarrow\bm A_T^1$.

\paragraph{Step 1: hierarchical expansion via the Cooley--Tukey decomposition}
Substituting the notation from \cref{case:pure_translation} into the multidimensional Cooley--Tukey decomposition formula \cref{eq:cooley-tukey-decomp} and rearranging the order of summation yields
\begin{align}\label{eq:translation_reduce}
	\hat{f}_{\bm A,\bm s}([\bm{h}])
	 & = \sum_{[\bm{x}_1] \in X_{\bm A_T^1}}\sum_{[\bm{x}_0] \in X_{\bm A_T^0}} f_{\bm{A},\bm s}([\bm{A}_T^0\bm{x}_1 + \bm{x}_0]) e^{-2\pi i\bm{h}_1^{\mathsf{T}}\bm{A}^{-1}(\bm{A}_T^0\bm{x}_1+\bm{x}_0)} e^{-2\pi i\bm{h}_0^\mathsf{T}(\bm{A}_T^0)^{-1}\bm{x}_0} \notag                                            \\
	 & = \sum_{[\bm{x}_0] \in X_{\bm A_T^0}}e^{-2\pi i\bm{h}_0^\mathsf{T}(\bm{A}_T^0)^{-1}\bm{x}_0} \left[\sum_{[\bm{y}]\in X_{\bm A_T^1}} f_{\bm{A},\bm s}([\bm{A}_T^0\bm{y} + \bm{x}_0]) e^{-2\pi i\bm{h}_1^\mathsf{T}\bm{A}^{-1}(\bm{A}_T^0\bm{y} + \bm{x}_0)}\right].
\end{align}
Let the inner summation be denoted as
\begin{equation}\label{eq:zeta}
	\zeta_{\bm h_1}(\bm x_0) \coloneqq \sum_{[\bm{y}]\in X_{\bm A_T^1}} f_{\bm{A},\bm s}([\bm{A}_T^0\bm{y} + \bm{x}_0]) e^{-2\pi i\bm{h}_1^\mathsf{T}\bm{A}^{-1}(\bm{A}_T^0\bm{y} + \bm{x}_0)},
\end{equation}
where $\bm x_0\in\mathcal D_{\bm A_T^0}$ is the standard integer representative of the congruence class $[\bm x_0]\in X_{\bm A_T^0}$.

\paragraph{Step 2: orbit decomposition and decoupling of the inner summation}
Combining \cref{subsec:orbit_decomposition} and \cref{prop:layer_permutation_translation}, for any $[\bm y]\in X_{\bm A_T^1}$, there exists a unique orbit representative $[\bm y_{\mathrm{ref}}]\in Y_{\mathrm{ref}}$ and $\Psi_{\bar g}\in\mathcal{R}([\bm y_{\mathrm{ref}}])$ such that
\[
	X_T^{[\bm y]} = \Psi_{\bar g}\bigl(X_T^{[\bm y_{\mathrm{ref}}]}\bigr),
	\qquad
	[\bm y]=[\bm y_{\mathrm{ref}}+\bm A_T^1\bm t_g]\in X_{\bm A_T^1},
\]
For a fixed $[\bm x_0]\in X_{\bm A_T^0}$ and $[\bm y_{\mathrm{ref}}]\in Y_{\mathrm{ref}}$, let $\bm x_0\in\mathcal D_{\bm A_T^0}$ be the standard integer representative of $[\bm x_0]$, let $\bm y_{\mathrm{ref}}\in\mathcal D_{\bm A_T^1}$ be the standard integer representative of $[\bm y_{\mathrm{ref}}]$, and select the integer representative $\bm x_{\mathrm{ref}}\coloneqq \bm A_T^0\bm y_{\mathrm{ref}}+\bm x_0$ of $[\bm x_{\mathrm{ref}}]\in X_T^{[\bm y_{\mathrm{ref}}]}$. For each element $\Psi_{\bar g}\in \mathcal{R}([\bm y_{\mathrm{ref}}])$, we take the integer representative of its translation part as $\bm z_g = \bm A\bm t_g$.
Since $\Psi_{\bar g}([\bm x_{\mathrm{ref}}]) = [\bm x_{\mathrm{ref}} + \bm z_g] = [\bm A_T^0\bm y+\bm x_0] \in X_T^{[\bm y]}$, we can reparameterize the inner summation based on the orbit decomposition \cref{eq:symmetric-partition} as
\[
	\zeta_{\bm h_1}(\bm x_0)
	= \sum_{[\bm y_{\mathrm{ref}}]\in Y_{\mathrm{ref}}} \sum_{\Psi_{\bar g}\in \mathcal{R}([\bm y_{\mathrm{ref}}])} 
	f_{\bm A,\bm s}([\bm x_{\mathrm{ref}} + \bm z_g]) 
	e^{-2\pi i \bm h_1^\mathsf{T} \bm A^{-1} (\bm x_{\mathrm{ref}} + \bm z_g)}.
\]
From \cref{prop:symmetry_inheritance}, we have $f_{\bm A,\bm s}([\bm x_{\mathrm{ref}} + \bm z_g]) = f_{\bm A,\bm s}([\bm x_{\mathrm{ref}}])$. Expanding the exponential term gives
\[
	\zeta_{\bm h_1}(\bm x_0)
	= \sum_{[\bm y_{\mathrm{ref}}]\in Y_{\mathrm{ref}}} \sum_{\Psi_{\bar g}\in \mathcal{R}([\bm y_{\mathrm{ref}}])} 
	f_{\bm A,\bm s}([\bm x_{\mathrm{ref}}]) 
	e^{-2\pi i \bm h_1^\mathsf{T} \bm A^{-1} \bm x_{\mathrm{ref}}} 
	e^{-2\pi i \bm h_1^\mathsf{T} \bm A^{-1} \bm z_g}.
\]
As established in \cref{prop:translation_free}, the stabilizer subgroups are all trivial in the pure translation case. Thus, for any $[\bm y_{\mathrm{ref}}]\in Y_{\mathrm{ref}}$, the system of coset representatives $\mathcal{R}([\bm y_{\mathrm{ref}}])$ can be chosen as the entire group $G_{\bm A}$. In this case, the only factor in the summand that depends on $\Psi_{\bar g}$ is $e^{-2\pi i\bm{h}_1^\mathsf{T}\bm{A}^{-1}\bm z_g}$. Consequently, the double summation decouples into the product of two independent sums.
\begin{equation}
	\label{eq:trans_W}
	\zeta_{\bm h_1}(\bm x_0)
	= V_{\bm h_1}(\bm x_0) \cdot \chi([\bm h_1]),
\end{equation}
where
\begin{align}
	V_{\bm h_1}(\bm x_0) 
	&\coloneqq \sum_{[\bm y_{\mathrm{ref}}]\in Y_{\mathrm{ref}}} f_{\bm A,\bm s}([\bm A_T^0 \bm y_{\mathrm{ref}} + \bm x_0]) 
	e^{-2\pi i \bm h_1^\mathsf{T} \bm A^{-1} (\bm A_T^0 \bm y_{\mathrm{ref}} + \bm x_0)}, \label{eq:V_t} \\
	\chi([\bm h_1]) 
	&\coloneqq \sum_{\Psi_{\bar g}\in G_{\bm A}} e^{-2\pi i \bm h_1^\mathsf{T} \bm A^{-1} \bm z_g}. \label{eq:Chi_t}
\end{align}

\begin{remark}
	It should be noted that although a specific integer representative $\bm z_g = \bm A\bm t_g$ of the congruence class $\bm z_{\bar g}$ is selected in the above derivation, the difference between any two representatives is a vector in $\bm A\mathbb{Z}^n$. Therefore, its inner product with $\bm h_1^\mathsf{T}\bm A^{-1}$ is guaranteed to be an integer. 	As a result, the phase factor $e^{-2\pi i \bm h_1^\mathsf{T} \bm A^{-1} \bm z_g}$ is independent of the choice of representative. Moreover, in the pure translation case $\bm z_g = \bm A\bm t_g \in \bm A_T^0\mathbb{Z}^n$ since $\bm A_T^1\bm t_g\in\mathbb{Z}^n$; hence replacing $\bm h_1$ by $\bm h_1 + \bm A_T^1\bm k$ (with $\bm k\in\mathbb{Z}^n$) leaves every term---and therefore $\chi$---unchanged, because $\bm k^\mathsf{T}(\bm A_T^0)^{-1}\bm z_g\in\mathbb{Z}$. Consequently, $\chi$ is in fact a class function of $[\bm h_1]\in X_{\bm A_T^1}$, so the bracketed notation $\chi([\bm h_1])$ in \cref{eq:Chi_t} is well defined.
\end{remark}

\paragraph{Step 3: effective frequency filtering}
If $\chi([\bm{h}_1]) = 0$, then the corresponding Fourier coefficient is $\hat{f}_{\bm A,\bm s}([\bm A_T^1\bm h_0 + \bm h_1]) = 0$. Hence, it is only necessary to compute for frequencies that satisfy $\chi([\bm h_1]) \neq 0$. We define the set of effective frequency indices as
\[
	H_{\chi} \coloneqq \{[\bm{h}_1] \in X_{\bm A_T^{1}} : \chi([\bm{h}_1]) \neq 0\}.
\]

\begin{theorem}\label{thm:zeta-count}
	In the pure translation case, assume $\bm A_T^1$ satisfies \cref{def:translation_compatibility}. Let $M \coloneqq |G_{\bm A}|$ and $D_1 \coloneqq |X_{\bm A_T^1}|$. Then,
	\begin{enumerate}
		\item For any $[\bm h_1]\in X_{\bm A_T^{1}}$, $\chi([\bm h_1])\in\{0,M\}$.
		\item $|H_\chi| = D_1 / M$.
	\end{enumerate}
\end{theorem}

\begin{proof}
	\begin{enumerate}
		\item Fix any $[\bm h_1]\in X_{\bm A_T^{1}}$. For each $\Psi_{\bar g}\in G_{\bm A}$, choose its corresponding translation integer representative $\bm z_g = \bm A \bm t_g$. Define the mapping
			\[
				\rho_{[\bm h_1]}:G_{\bm A}\to\mathbb C^*,\qquad \rho_{[\bm h_1]}(\Psi_{\bar g})\coloneqq e^{-2\pi i\,\bm h_1^{\mathsf T}\bm A^{-1}\bm z_g}.
			\]
			It is straightforward to verify that $\rho_{[\bm h_1]}$ is well-defined.

			For any $\Psi_{\bar g_1},\Psi_{\bar g_2}\in G_{\bm A}$, take their corresponding translation integer representatives $\bm z_{g_1}, \bm z_{g_2}$. For any integer representative $\bm x\in\mathbb Z^n$, we have
			\[
				(\Psi_{\bar g_1}\Psi_{\bar g_2})([\bm x])
				=\Psi_{\bar g_1}([\bm x+\bm z_{g_2}])
				=[\bm x+\bm z_{g_2}+\bm z_{g_1}].
			\]
			Therefore,
			\[
				\rho_{[\bm h_1]}(\Psi_{\bar g_1}\Psi_{\bar g_2})
				=e^{-2\pi i\,\bm h_1^{\mathsf T}\bm A^{-1}(\bm z_{g_1}+\bm z_{g_2})}
				=\rho_{[\bm h_1]}(\Psi_{\bar g_1})\,\rho_{[\bm h_1]}(\Psi_{\bar g_2}),
			\]
			which means $\rho_{[\bm h_1]}$ is a group homomorphism. Thus,
			\[
				\chi([\bm h_1])=\sum_{\Psi_{\bar g}\in G_{\bm A}}\rho_{[\bm h_1]}(\Psi_{\bar g}).
			\]
			If $\rho_{[\bm h_1]}$ is trivial, then $\rho_{[\bm h_1]}(\Psi_{\bar g})=1$ for all $\Psi_{\bar g}$, yielding $\chi([\bm h_1])=|G_{\bm A}|=M$.
			If $\rho_{[\bm h_1]}$ is nontrivial, there exists some $\Psi_{\bar g_0}\in G_{\bm A}$ such that $\rho_{[\bm h_1]}(\Psi_{\bar g_0})\neq 1$.
			Let $S\coloneqq\sum_{\Psi_{\bar g}\in G_{\bm A}}\rho_{[\bm h_1]}(\Psi_{\bar g})$. By the properties of a group homomorphism,
			\[
				\rho_{[\bm h_1]}(\Psi_{\bar g_0})S
				=\sum_{\Psi_{\bar g}\in G_{\bm A}}\rho_{[\bm h_1]}(\Psi_{\bar g_0}\Psi_{\bar g}).
			\]
			Since the mapping $\Psi_{\bar g}\mapsto\Psi_{\bar g_0}\Psi_{\bar g}$ is a bijection on $G_{\bm A}$, the summation on the right-hand side is simply a rearrangement of $S$. Hence $\rho_{[\bm h_1]}(\Psi_{\bar g_0})S=S$, or $(\rho_{[\bm h_1]}(\Psi_{\bar g_0})-1)S=0$.
			Since $\rho_{[\bm h_1]}(\Psi_{\bar g_0})\neq 1$, we must have $S=0$, meaning $\chi([\bm h_1])=0$.
			In summary, $\chi([\bm h_1])\in\{0,M\}$.

	\item Let
			\[
				S\coloneqq\sum_{[\bm h_1]\in X_{\bm A_T^1}}\chi([\bm h_1]).
			\]
			Exchanging the order of summation and substituting \cref{eq:Chi_t} yields
			\[
				S=\sum_{\Psi_{\bar g}\in G_{\bm A}}\sum_{[\bm h_1]\in X_{\bm A_T^1}}e^{-2\pi i\,\bm h_1^{\mathsf T}\bm A^{-1}\bm z_g}.
			\]
			where $\bm z_g \in \mathbb Z^n$ is the integer representative corresponding to $\Psi_{\bar g}$. Let $\bm k_{g}\coloneqq (\bm A_T^0)^{-1}\bm z_g$. By the translation compatibility \cref{def:translation_compatibility} and $\bm z_g=\bm A\bm t_g$, it follows that $\bm k_{g}=\bm A_T^1\bm t_g\in\mathbb Z^n$. Since $\bm A_T^1$ is a positive integer diagonal matrix, let $\bm A_T^1=\operatorname{diag}(N_{1,1},\dots,N_{1,n})$. By utilizing the set of standard representatives of $X_{\bm A_T^1}$ to expand the inner summation, we obtain
			\[
				\sum_{[\bm h_1]\in X_{\bm A_T^1}}e^{-2\pi i\,\bm h_1^{\mathsf T}(\bm A_T^1)^{-1}\bm k_{g}}
				=\prod_{j=1}^n\left(\sum_{m=0}^{N_{1,j}-1}e^{-2\pi i\,m\,(\bm k_{g})_j/N_{1,j}}\right).
			\]
			Consequently, the inner summation evaluates to $D_1$ if and only if for all $j$, $(\bm k_{g})_j$ is an integer multiple of $N_{1,j}$, i.e., $\bm k_{g}\in\bm A_T^1\mathbb Z^n$. Otherwise, there exists at least one $j$ for which the one-dimensional geometric series sums to $0$, causing the inner summation to vanish. Notice that
			$\bm k_{g}\in\bm A_T^1\mathbb Z^n\iff \bm z_g=\bm A_T^0\bm k_{ g}\in\bm A\mathbb Z^n\iff \bm z_{\bar g}=[\bm 0]\in X_{\bm A}\iff \Psi_{\bar g}=\mathrm{id}$.
			Therefore, the inner summation evaluates to $D_1$ exclusively for $\Psi_{\bar g}=\mathrm{id}$, and to $0$ otherwise, resulting in $S=D_1$.

			On the other hand, from conclusion (1), $\chi([\bm h_1])$ is constantly $M$ on $H_\chi$ and $0$ elsewhere, giving
			\[
				S=\sum_{[\bm h_1]\in H_\chi}M=|H_\chi|\,M.
			\]
			Comparing the two expressions yields $|H_\chi|\,M=D_1$, or $|H_\chi|=D_1/M$.
		\end{enumerate}
\end{proof}

By the above theorem, for $\bm{h}_1 \in H_{\chi}$, \cref{eq:translation_reduce} simplifies to
\begin{align}\label{eq:trans_dft_step}
	\hat{f}_{\bm A,\bm s}([\bm{A}_T^{1}\bm{h}_0 + \bm{h}_1])
	= M \cdot \sum_{[\bm{x}_0] \in X_{\bm A_T^0}} V_{\bm{h}_1}(\bm{x}_0) e^{-2\pi i \bm{h}_0^\mathsf{T} (\bm{A}_T^0)^{-1} \bm{x}_0}.
\end{align}
Synthesizing the above three steps of the derivation yields the DFT symmetry reduction formula in the pure translation case
\begin{equation}
	\hat{f}_{\bm A,\bm s}\left([\bm h]\right)
	= \chi([\bm h_1])\sum_{[\bm x_0]\in X_{\bm A_T^0}} V_{\bm h_1}(\bm x_0)
	e^{-2\pi i\bm h_0^\mathsf{T}(\bm A_T^0)^{-1}\bm x_0}.
\end{equation}
In particular, when $[\bm h_1]\notin H_{\chi}$, $\hat f_{\bm A,\bm s}([\bm h])=0$, and when $[\bm h_1]\in H_{\chi}$, the above equation reduces to \cref{eq:trans_dft_step}.

\begin{remark}[Pure translation case in three-dimensional space groups]
	For a three-dimensional space group $\mathcal G$, the finite group of pure translations modulo the conventional-cell integer translations is $T_{\mathcal G}/T$, where $T_{\mathcal G}=\mathcal G\cap\mathbb R^3$ is the full translation subgroup and $T=\{(\bm I_3,\bm z):\bm z\in\mathbb Z^3\}$ is the integer translation subgroup in the chosen conventional basis. Let $T_{\mathrm{cen}}$ be a chosen set of coset representatives. For each $\bm t\in T_{\mathrm{cen}}$, define the pure translation
	\[
		\Psi_{\bm t}:[\bm x]\mapsto [\bm x]+[\bm A\bm t].
	\]
	By the translation compatibility \cref{def:translation_compatibility} we have $\bm A_T^1\bm t\in\mathbb Z^n$; hence the integer representative
	\[
		\bm z_{\bm t}\coloneqq \bm A\bm t = \bm A_T^0(\bm A_T^1\bm t)\in\mathbb Z^n
	\]
	is well defined, and the discrete symmetry group in the pure translation case is
	\[
		G_{\bm A}=\{\Psi_{\bm t}:[\bm x]\mapsto [\bm x]+[\bm A\bm t],\ \bm t\in T_{\mathrm{cen}}\},
		\qquad M=|G_{\bm A}|=|T_{\mathcal G}/T|=|T_{\mathrm{cen}}|.
	\]
	Now fix $[\bm h_1]\in X_{\bm A_T^1}$ and let $\bm h_1=(h,k,l)^\mathsf T$ be its standard integer representative. Substituting $\bm z_{\bm t}=\bm A\bm t$ into the character \cref{eq:Chi_t} yields
	\[
		\chi([\bm h_1])
		=\sum_{\bm t\in T_{\mathrm{cen}}} e^{-2\pi i\,\bm h_1^{\mathsf T}\bm A^{-1}\bm z_{\bm t}}
		=\sum_{\bm t\in T_{\mathrm{cen}}} e^{-2\pi i\,\bm h_1^{\mathsf T}\bm t},
	\]
	from which the criteria in the table below follow directly. The possible centering types and their corresponding criteria for effective frequencies are as follows.
	\begin{table}[htbp]
		\centering
		\caption{Centering types and criteria for effective frequencies}
		\renewcommand{\arraystretch}{1.2}
		\begin{tabular}{c c c c}
			\toprule
			Centering type          & $T_{\mathrm{cen}}$                          & $M=|T_{\mathrm{cen}}|$ & Condition for $\chi([\bm h_1])\neq 0$           \\
			\midrule
			$P$                     & $\{(0,0,0)\}$                               & 1                      & None                                            \\
			$A$                     & $\{(0,0,0),(0,\frac12,\frac12)\}$           & 2                      & $k+l\equiv 0 \pmod 2$                           \\
			$B$                     & $\{(0,0,0),(\frac12,0,\frac12)\}$           & 2                      & $h+l\equiv 0 \pmod 2$                           \\
			$C$                     & $\{(0,0,0),(\frac12,\frac12,0)\}$           & 2                      & $h+k\equiv 0 \pmod 2$                           \\
			$I$                     & $\{(0,0,0),(\frac12,\frac12,\frac12)\}$     & 2                      & $h+k+l\equiv 0 \pmod 2$                         \\
			\multirow[t]{2}{*}{$F$} & $\{(0,0,0),(0,\frac12,\frac12),$            & \multirow[t]{2}{*}{4}  & \multirow[t]{2}{*}{$h\equiv k\equiv l \pmod 2$} \\
			                        & $(\frac12,0,\frac12),(\frac12,\frac12,0)\}$ &                        &                                                 \\
			\multirow[t]{2}{*}{$R$} & $\{(0,0,0),(\frac23,\frac13,\frac13),$      & \multirow[t]{2}{*}{3}  & \multirow[t]{2}{*}{$-h+k+l\equiv 0 \pmod 3$}    \\
			                        & $(\frac13,\frac23,\frac23)\}$               &                        &                                                 \\
			\bottomrule
		\end{tabular}
	\end{table}

	Here, the congruence condition corresponding to the $R$ centering is given with respect to hexagonal axes in the obverse setting. Therefore, for all 230 three-dimensional space groups, the only pure translation types that need to be considered in the pure translation reduction are those listed in the table above. For these centering types, $\chi([\bm h_1])$ is constantly equal to $M$ at frequencies satisfying the criteria in the table, and constantly $0$ at all other frequencies. In other words, the table above provides the specific characterization of the effective frequency set $H_{\chi}$ in three-dimensional crystallography.
\end{remark}

\subsubsection{Matrix representation in the pure translation case}
\label{sec:algorithm_matrix_translation}

To formulate the vector and matrix representations for the pure translation case, the effective frequency set $H_\chi$ and the orbit representative set $Y_{\mathrm{ref}}$ are ordered lexicographically as in the preceding discussion. 

\begin{lemma}
	\label{lem:space_freq_duality}
	In the pure translation case, $|Y_{\mathrm{ref}}| = |H_{\chi}| = D_1 / M \coloneqq d$. 
\end{lemma}

\begin{proof}
	In the pure translation case, we have $\Theta=\Theta_T=\{X_T^{[\bm y]}:[\bm y]\in X_{\bm A_T^1}\}$. By definition, the mapping
	\[
		X_{\bm A_T^1}\to\Theta,
		\qquad
		[\bm y]\mapsto X_T^{[\bm y]}
	\]
	is surjective. Furthermore, by the uniqueness of the decomposition \cref{eq:x_AT_decomp}, distinct elements $[\bm y]\in X_{\bm A_T^1}$ yield distinct subsets $X_T^{[\bm y]}$. Thus, the mapping is also injective, and consequently bijective. Therefore,
	\[
		|\Theta|=|X_{\bm A_T^1}|=D_1.
	\]

	On the other hand, from \cref{prop:translation_free}, the stabilizer subgroup $G_{[\bm y]}$ of $X_T^{[\bm y]}$ is trivial for any $[\bm y]\in X_{\bm A_T^1}$. Therefore, by the orbit-stabilizer theorem,
	\[
		|O_{[\bm y]}|=\frac{|G_{\bm A}|}{|G_{[\bm y]}|}=|G_{\bm A}|=M.
	\]
	Given the definition of $Y_{\mathrm{ref}}$ as a set of orbit representatives, $\Theta$ can be expressed as the disjoint union 
	$\Theta=\bigsqcup_{[\bm y]\in Y_{\mathrm{ref}}}O_{[\bm y]}$. Hence,
	\[
		|\Theta|=\sum_{[\bm y]\in Y_{\mathrm{ref}}}|O_{[\bm y]}|=|Y_{\mathrm{ref}}|\,M.
	\]
	Substituting $|\Theta|=D_1$ yields $|Y_{\mathrm{ref}}|=D_1/M$.

	Finally, it is known from \cref{thm:zeta-count} that $|H_{\chi}|=D_1/M$, which implies $|Y_{\mathrm{ref}}|=|H_{\chi}|=D_1/M\coloneqq d$.
\end{proof}

By \cref{lem:space_freq_duality}, the two sets have the same cardinality, $|Y_{\mathrm{ref}}|=|H_{\chi}|=d$. Sorting the effective frequency set $H_{\chi}$ and the orbit representative set $Y_{\mathrm{ref}}$ lexicographically, we obtain
\[
	H_{\chi}^{\mathrm{lex}} = (\bm h_1^{(0)}, \dots, \bm h_1^{(d-1)}),\qquad
	Y_{\mathrm{ref}}^{\mathrm{lex}} = (\bm y^{(0)}, \dots, \bm y^{(d-1)}).
\]

\begin{definition}
	\label{def:trans_matrix}
	For a fixed $\bm x_0 \in \mathcal{D}_{\bm A_T^0}$, we define the following vectors and matrices, where the indices range over $v,p=0,\dots,d-1$ and $\bm y^{(v)}$, $\bm h_1^{(p)}$ are the lexicographically ordered standard integer representatives of $Y_{\mathrm{ref}}$ and $H_\chi$, respectively,
	\begin{enumerate}
		\item The function value vector $\bm F(\bm x_0) \in \mathbb{C}^d$, whose $v$-th component is given by
		    \[
		    (\bm F(\bm x_0))_v \coloneqq f_{\bm A,\bm s}([\bm A_T^0 \bm y^{(v)} + \bm x_0]).
			\]
		\item The local transformation vector $\bm V(\bm x_0) \in \mathbb{C}^d$, whose $p$-th component is given by
		    \[
		    (\bm V(\bm x_0))_p \coloneqq V_{\bm h_1^{(p)}}(\bm x_0).
			\]
		\item The constant matrix $\bm T_{c,t} \in \mathbb{C}^{d\times d}$, with entries
		      \[
			      (\bm T_{c,t})_{p,v} = e^{-2\pi i (\bm h_1^{(p)})^\mathsf{T} (\bm A_T^1)^{-1} \bm y^{(v)}}.
		      \]
		\item The phase diagonal matrix $\bm T_{s,t}(\bm x_0) \in \mathbb{C}^{d\times d}$, with diagonal entries
		      \[
			      (\bm T_{s,t}(\bm x_0))_{p,p} = e^{-2\pi i (\bm h_1^{(p)})^\mathsf{T} \bm A^{-1} \bm x_0}.
		      \]
	\end{enumerate}
\end{definition}

This leads to the matrix relation
\begin{equation}\label{eq:core_matrix_relation_translation}
	\bm V(\bm x_0) = \bm T_{s,t}(\bm x_0) \, \bm T_{c,t} \, \bm F(\bm x_0).
\end{equation}

\begin{theorem}
	\label{thm:Tc_inv_trans}
	The matrix $\bm T_{c,t}$ satisfies $\bm T_{c,t}^* \bm T_{c,t} = d \bm I_d$, and its inverse is
	\[
		\bm T_{c,t}^{-1} = \frac{1}{d} \bm T_{c,t}^*.
	\]
\end{theorem}

\begin{proof}
	The matrix entries are defined as
	$(\bm T_{c,t})_{p,v} = e^{-2\pi i (\bm h_1^{(p)})^\mathsf{T} (\bm A_T^1)^{-1} \bm y^{(v)}}$,
	where $\bm h_1^{(p)}$ and $\bm y^{(v)}$ are the standard integer representatives of the congruence classes $[\bm h_1^{(p)}] \in H_\chi$ and $[\bm y^{(v)}] \in Y_{\mathrm{ref}}$, respectively.
	Let $\bm K = \bm T_{c,t}^* \bm T_{c,t}$. Then its $(u,v)$-th entry is
	\begin{align}
		(\bm K)_{u,v}
		 & = \sum_{p=0}^{d-1} \overline{(\bm T_{c,t})_{p,u}} (\bm T_{c,t})_{p,v} \notag                                                                                                                                           \\
		 & = \sum_{p=0}^{d-1} e^{-2\pi i (\bm h_1^{(p)})^\mathsf{T} (\bm A_T^1)^{-1} (\bm y^{(v)} - \bm y^{(u)})} \notag \\
		 & = \sum_{[\bm h_1] \in H_\chi} e^{-2\pi i \bm h_1^\mathsf{T} (\bm A_T^1)^{-1} (\bm y^{(v)} - \bm y^{(u)})}.
	\end{align}
	Here the $d$ standard integer representatives $\bm h_1^{(p)}$ enumerate $H_\chi$, so the sum over $p$ coincides with the sum over $[\bm h_1]\in H_\chi$.

	When $u=v$, $\bm y^{(v)} - \bm y^{(u)} = \bm 0$, and the above expression evaluates to $|H_\chi| = d$, i.e., $(\bm K)_{u,u} = d$.

	When $u \neq v$, by \cref{thm:zeta-count}, we have $\chi([\bm h_1]) = M$ if $[\bm h_1] \in H_\chi$, and $\chi([\bm h_1]) = 0$ otherwise. Thus, the expression can be rewritten as a weighted sum over the finite group $X_{\bm A_T^1}$. For each $\Psi_{\bar g}\in G_{\bm A}$, taking its corresponding translation integer representative $\bm z_g = \bm A \bm t_g$, we obtain
	\begin{align}
		(\bm K)_{u,v}
		 & = \frac{1}{M} \sum_{[\bm h_1] \in X_{\bm A_T^1}} \chi([\bm h_1]) e^{-2\pi i \bm h_1^\mathsf{T} (\bm A_T^1)^{-1} (\bm y^{(v)} - \bm y^{(u)})} \notag                                                        \\
		 & = \frac{1}{M} \sum_{\Psi_{\bar g}\in G_{\bm A}} \sum_{[\bm h_1] \in X_{\bm A_T^1}} e^{-2\pi i\,\bm h_1^\mathsf{T} (\bm A_T^1)^{-1}\bigl((\bm A_T^0)^{-1}\bm z_g + \bm y^{(v)} - \bm y^{(u)}\bigr)}.
	\end{align}

	For a fixed $\Psi_{\bar g}\in G_{\bm A}$, the translation compatibility condition \cref{def:translation_compatibility} implies
	$(\bm A_T^0)^{-1}\bm z_g=\bm A_T^1\bm t_g\in\mathbb Z^n$.
	Let
	$\bm w_{g,u,v}\coloneqq \bm A_T^1\bm t_g + \bm y^{(v)} - \bm y^{(u)}\in\mathbb Z^n$. Denote
	\[
		S(\bm w)\coloneqq \sum_{[\bm h_1]\in X_{\bm A_T^1}} e^{-2\pi i\,\bm h_1^\mathsf{T}(\bm A_T^1)^{-1}\bm w}.
	\]
	Let $\bm A_T^1 = \operatorname{diag}(N_{1,1},\dots,N_{1,n})$, and expanding the summation using the set of standard representatives of $X_{\bm A_T^1}$ yields
	\[
		S(\bm w)=\prod_{j=1}^n\left(\sum_{m=0}^{N_{1,j}-1}e^{-2\pi i\,m\,w_j/N_{1,j}}\right).
	\]
	From the finite geometric series sum formula, if there exists a $j$ such that $N_{1,j}\nmid w_j$, the corresponding one-dimensional sum is
	$\sum_{m=0}^{N_{1,j}-1}e^{-2\pi i\,m\,w_j/N_{1,j}}=0$, resulting in $S(\bm w)=0$. If $N_{1,j}\mid w_j$ for all $j$, then $\sum_{m=0}^{N_{1,j}-1}e^{-2\pi i\,m\,w_j/N_{1,j}}=N_{1,j}$. Consequently, $S(\bm w)=\prod_{j=1}^n N_{1,j}=|X_{\bm A_T^1}|=D_1$.
	In other words,
	\[
		S(\bm w)\neq 0\ \Longleftrightarrow\ \bm w\in \bm A_T^1\mathbb Z^n.
	\]
	Therefore, for a fixed $\Psi_{\bar g}\in G_{\bm A}$, $S(\bm w)$ is non-zero if and only if
	\begin{align}
		[\bm A_T^1\bm t_g + \bm y^{(v)} - \bm y^{(u)}] = [\bm 0] \in X_{\bm A_T^1}
		\iff [\bm y^{(u)}] = [\bm y^{(v)} + \bm A_T^1\bm t_g].
	\end{align}
	According to \cref{prop:layer_permutation_translation}, the above equation is equivalent to
	\begin{align}\label{eq:orbit_eq_uv}
		X_T^{[\bm y^{(u)}]} = \Psi_{\bar g}(X_T^{[\bm y^{(v)}]}). 
	\end{align}
	Note that when $u \neq v$, $[\bm y^{(u)}]$ and $[\bm y^{(v)}]$ are distinct representatives in the orbit representative set $Y_{\mathrm{ref}}$, so their corresponding orbits are strictly disjoint. Hence, for any group element $\Psi_{\bar g}\in G_{\bm A}$, \cref{eq:orbit_eq_uv} cannot hold. Thus $S(\bm w) = 0$, which implies $(\bm K)_{u,v}=0$.

	In conclusion, combining this with the result for $u=v$, we have
	\begin{align}
		(\bm K)_{u,v} = \delta_{uv} d.
	\end{align}
	Therefore, $\bm T_{c,t}^* \bm T_{c,t} = d \bm I_d$, which implies $\bm T_{c,t}^{-1} = \frac{1}{d} \bm T_{c,t}^*$.
\end{proof}

Based on the above formulation, the SFFT and SIFFT algorithms for the pure translation case are summarized in \cref{alg:trans_SFFT} and \cref{alg:trans_SIFFT}, respectively. 

\begin{algorithm}[H]
	\renewcommand{\algorithmicrequire}{\textbf{Input:}}
	\renewcommand{\algorithmicensure}{\textbf{Output:}}
	\caption{SFFT for pure translation (SFFT-TRN)}
	\label{alg:trans_SFFT}
	\begin{algorithmic}[1]
		\REQUIRE $\bm{A}_T^0$, $\bm{A}_T^1$, $\bm{A}$, $Y_{\mathrm{ref}}^{\mathrm{lex}}$, discrete function values $\{f_{\bm A, \bm s}([\bm A_T^0\bm y^{(v)} + \bm x_0])\}_{v=0}^{d-1}$ on the orbit representatives $[\bm y^{(v)}]\in Y_{\mathrm{ref}}^{\mathrm{lex}}$, discrete translation vector set $\{\bm z_g:\  \bar g\in G\}$.
		\ENSURE $\hat{f}_{\bm A,\bm s}\left([\bm{A}_T^{1} \bm{h}_0 + \bm{h}_1]\right)$, for all $[\bm h_1]\in H_{\chi}$.
		% 		\STATE Precompute all $\bm z_g = \bm A \bm t_g$.
		\STATE Precompute $\chi([\bm{h}_1])$ and construct $H_{\chi}$.
		\STATE Precompute the constant matrix $\bm{T}_{c,t}$.
		\FOR{each $\bm{x}_0 \in \mathcal D_{\bm A_T^0}$}
		\STATE Compute the diagonal matrix $\bm T_{s,t}(\bm x_0)$.
		\STATE Construct $\bm F(\bm x_0)$ from $\{f_{\bm A,\bm s}([\bm A_T^0 \bm y^{(v)} + \bm x_0])\}_{v=0}^{d-1}$
		\STATE $\bm V(\bm x_0) \leftarrow \bm T_{s,t}(\bm x_0) \, \bm T_{c,t} \, \bm F(\bm x_0)$
		\STATE Extract $\{V_{\bm h_1}(\bm x_0)\}_{[\bm h_1]\in H_{\chi}}$
		\ENDFOR
		\FOR{each $[\bm h_1]\in H_{\chi}$}
		\STATE $\hat{f}_{\bm A,\bm s}([\bm A_T^1 \bm h_0 + \bm h_1]) \leftarrow \operatorname{FFT}_{\bm x_0}\bigl[M \cdot V_{\bm h_1}(\bm x_0)\bigr]$.
		\ENDFOR
	\end{algorithmic}
\end{algorithm}

\begin{algorithm}[H]
	\renewcommand{\algorithmicrequire}{\textbf{Input:}}
	\renewcommand{\algorithmicensure}{\textbf{Output:}}
	\caption{SIFFT for pure translation (SIFFT-TRN)}
	\label{alg:trans_SIFFT}
	\begin{algorithmic}[1]
		\REQUIRE $\bm A_T^0$, $\bm A_T^1$, $\bm A$, $Y_{\mathrm{ref}}^{\mathrm{lex}}$, the Fourier coefficients $\hat{f}_{\bm A,\bm s}([\bm A_T^1\bm h_0 + \bm h_1])$ for $[\bm h_1]\in H_\chi$ (as produced by SFFT-TRN), discrete translation vector set $\{\bm z_g:\  \bar g\in G\}$.
		\ENSURE the function values $\{f_{\bm A,\bm s}([\bm A_T^0\bm y^{(v)} + \bm x_0])\}_{v=0}^{d-1}$ on the orbit representatives $[\bm y^{(v)}]\in Y_{\mathrm{ref}}^{\mathrm{lex}}$.
		% 		\STATE Precompute all $\bm z_g = \bm A \bm t_g$.
% 		\STATE Precompute $\chi([\bm h_1])$ and construct $H_{\chi}$.
		\STATE Construct the constant matrix $\bm T_{c,t}^*$.
		\FOR{each $[\bm h_1]\in H_{\chi}$}
		\STATE $\tilde{f}_{[\bm h_1]}(\bm h_0) \leftarrow \hat{f}_{\bm A,\bm s}([\bm A_T^1 \bm h_0 + \bm h_1])$
		\STATE $V_{\bm h_1}(\bm x_0) \leftarrow \frac{1}{M} \operatorname{IFFT}_{\bm h_0}\bigl[\tilde{f}_{[\bm h_1]}(\bm h_0)\bigr]$
		\ENDFOR
		\FOR{each $\bm x_0 \in \mathcal{D}_{\bm A_T^0}$}
		\STATE Compute the diagonal matrix ${\bm T}_{s,t}^*(\bm x_0)$
		\STATE Construct $\bm V(\bm x_0)$ from $\{V_{\bm h_1}(\bm x_0)\}_{[\bm h_1]\in H_{\chi}}$
		\STATE $\bm F(\bm x_0) \leftarrow \frac{1}{d} \bm T_{c,t}^* \bigl( {\bm T}_{s,t}^*(\bm x_0) \, \bm V(\bm x_0) \bigr)$
		\STATE Extract $f_{\bm A,\bm s}([\bm A_T^0 \bm y^{(v)} + \bm x_0]) \leftarrow (\bm F(\bm x_0))_v$
		\ENDFOR
	\end{algorithmic}
\end{algorithm}

\subsection{Recursive symmetry reduction framework}
\label{sec:recursive_framework}

Complex space groups typically contain both nontrivial linear actions and pure translational symmetries, making it difficult to completely eliminate computational redundancy through a single reduction. Therefore, it is necessary to construct a unified framework capable of recursively invoking the two aforementioned reduction strategies.

The core of this framework lies in the fact that the mathematical essence of all types of symmetry reduction is orbit decomposition—utilizing the induced action of finite groups on discrete grids to transform the global DFT into several local DFTs defined on subgrids. These local problems often inherit partial symmetries of the original problem, thereby constituting independent subproblems that are smaller in scale but similar in structure. Recursively processing these subproblems allows their inherited symmetries to be exploited through compatible grid decompositions.

At each recursive node, the current discrete symmetry group determines the available reduction. A nontrivial linear action is treated using a compatible grid decomposition; a nontrivial pure translation group is treated by the translational reduction with a translation-compatible factorization. By \cref{prop:translation_free}, each representative layer has a trivial stabilizer in the current pure translation group. The remaining local DFTs are computed by the standard FFT as in \cref{alg:trans_SFFT,alg:trans_SIFFT}, and this branch terminates.

\subsubsection{Formalization of the SFFT problem}

\begin{definition}
	\label{def:sfft_problem}
	An $n$-dimensional \emph{SFFT problem} $\mathcal{P}$ is specified by a continuous function $f:\mathbb{T}^n\to\mathbb{C}$ possessing $G$-symmetry (\cref{def:G-symmetry}) together with the discretization parameters $(\bm A,\bm s)$. These determine the discrete samples
	\[
		f_{\bm A,\bm s} = f\circ\iota_{\bm A,\bm s}: X_{\bm A}\to\mathbb{C},
		\qquad X_{\bm A} \coloneqq \mathbb{Z}^n/\bm A\mathbb{Z}^n
	\]
	(\cref{subsec:grid_definition}). We record the problem as the 4-tuple
	\[
		\mathcal{P} = \left( f_{\bm A,\bm s},\, \bm A,\, \bm s,\, \mathcal{G} \right),
	\]
	with components
	\begin{enumerate}
		\item $f_{\bm A,\bm s}:X_{\bm A}\to\mathbb{C}$ is the discrete function sampled on the grid, i.e.\ $f_{\bm A,\bm s}([\bm x])=f(\iota_{\bm A,\bm s}([\bm x]))$; it is the discretization of the $G$-symmetric $f$.
		\item $\bm{A} = \operatorname{diag}(N_1, \dots, N_n)$ is the discretization matrix, with $N_j \in \mathbb{Z}_{\geq 1}$.
		\item $\bm{s}\in [0,1)^n$ is the translation parameter.
		\item $\mathcal{G}$ is an $n$-dimensional space group (satisfying \cref{def:space_group}), with finite symmetry quotient $G=\mathcal{G}/T$.
	\end{enumerate}
\end{definition}

Denote the integer translation subgroup as $T = \{(\bm{I}_n, \bm{t}) : \bm{t}\in\mathbb{Z}^n\} \lhd \mathcal{G}$ and its quotient group as $G = \mathcal{G}/T.$

\begin{definition}
	\label{def:well_formed_problem}
	An SFFT problem $\mathcal{P} = (f_{\bm A,\bm s}, \bm A, \bm s, \mathcal{G})$ is said to be \emph{well-formed} when its discrete function is invariant under its discrete symmetry group $G_{\bm A}$, namely
	\begin{equation}\label{eq:wellformed_invariance}
		f_{\bm A,\bm s}\bigl(\Psi_{\bar g}([\bm{x}])\bigr) = f_{\bm A,\bm s}([\bm{x}]),
		\quad \forall\,\Psi_{\bar g}\in G_{\bm A},\; \forall\,[\bm{x}]\in X_{\bm A}.
	\end{equation}
	Here $G_{\bm A}$ is the discrete symmetry group from \cref{def:discrete_action}, which is well-defined precisely when the grid parameters $(\bm A,\bm s)$ are $G$-admissible.
\end{definition}

\begin{remark}
	\label{rem:base_case}
	The base problem $\mathcal{P}^{(0)} = \bigl( f_{\bm{A},\bm{s}},\, \bm{A},\, \bm{s},\, \mathcal{G} \bigr)$ is well-formed as soon as its grid $(\bm A,\bm s)$ is $G$-admissible and the underlying continuous function $f$ possesses $G$-symmetry. By \cref{prop:grid_invariance}, $G$-admissibility makes the discrete action $\Psi_{\bar g}$ (hence $G_{\bm A}$) well-defined; by \cref{prop:symmetry_inheritance}, the $G$-symmetry of $f$ then yields exactly the invariance \eqref{eq:wellformed_invariance}. Thus the base problem satisfies the defining condition of \cref{def:well_formed_problem}.
\end{remark}

\subsubsection{Construction of subproblems}

We proceed by induction on the recursion depth $k$. The base case $k=0$ is well-formed by \cref{rem:base_case}. As the induction hypothesis, let $\mathcal{P}^{(k)}$ denote a well-formed SFFT problem at the $k$-th level---the parent node from which the $(k+1)$-st level subproblems are constructed below---with its discrete symmetry group denoted by $G_{\bm A^{(k)}}$.

\paragraph{Subproblems of the nontrivial linear-action case}

When $G_{\bm A^{(k)}}$ contains a nontrivial linear action, select a $G_{\bm A^{(k)}}$-compatible (in the sense of \cref{def:linear_compatibility}) diagonal decomposition $\bm{A}^{(k)} = \bm{A}_0^{(k)} \bm{A}_1^{(k)}$ (such a decomposition always exists: $\bm A_0^{(k)}=\bm A^{(k)}$ with $\bm A_1^{(k)}=\bm I_n$ is trivially $G_{\bm A^{(k)}}$-compatible, and the uniform choice $\bm A_0^{(k)}=k\bm I_n$, $k\ge 2$ dividing every diagonal entry of $\bm A^{(k)}$, is available whenever $\gcd(N_1^{(k)},\dots,N_n^{(k)})\ge 2$). For each $[\bm y]\in X_{\bm A_0^{(k)}}$, define the subset
\[
	X_R^{(k),[\bm y]} = \{[\bm A_0^{(k)}\bm x_1 + \bm y] : [\bm x_1]\in X_{\bm A_1^{(k)}}\}.
\]
By \cref{prop:layer_permutation}, $G_{\bm A^{(k)}}$ induces a permutation action on the family of subsets
\[
	\Theta_R^{(k)} = \{ X_R^{(k),[\bm y]} : [\bm y]\in X_{\bm A_0^{(k)}} \}.
\]
Choose the set of orbit representatives $Y_{\mathrm{ref}}^{(k)}\subset X_{\bm A_0^{(k)}}$ of the $G_{\bm A^{(k)}}$-action on $\Theta_R^{(k)}$. For each orbit representative $[\bm y]\in Y_{\mathrm{ref}}^{(k)}$, construct the $(k+1)$-th level subproblem $\mathcal{P}_{[\bm y]}^{(k+1)}$ as follows.

\begin{definition}
	\label{def:subproblem_linear_action}
	Fixing $[\bm y]\in Y_{\mathrm{ref}}^{(k)}$, we define
	\[
		\begin{aligned}
			\mathcal{P}_{[\bm y]}^{(k+1)} = \bigl( f_{[\bm y]}^{(k+1)},\, \bm A^{(k+1)},\, \bm s^{(k+1)},\, \mathcal{G}^{(k+1)} \bigr),
		\end{aligned}
	\]
	where
	\begin{enumerate}
		\item $\bm{A}^{(k+1)} \coloneqq \bm{A}_1^{(k)}$, $X_{\bm A^{(k+1)}} \coloneqq X_{\bm A_1^{(k)}}$.
		\item $\bm{s}^{(k+1)} \coloneqq (\bm{A}_0^{(k)})^{-1} (\bm{y} + \bm{s}^{(k)}) \in [0,1)^n$.
		\item $f_{[\bm y]}^{(k+1)}([\bm x_1]) \coloneqq f^{(k)}([\bm A_0^{(k)}\bm x_1 + \bm y])$.
		\item $\mathcal{G}^{(k+1)} \coloneqq \bigl\{ g\in\mathcal{G}^{(k)} : \Psi_{\bar g}^{(k)}(X_R^{(k),[\bm y]}) = X_R^{(k),[\bm y]} \bigr\}$.
	\end{enumerate}
\end{definition}

\begin{lemma}
	\label{lem:subgroup_inheritance}
	$\mathcal{G}^{(k+1)}$ is a subgroup of $\mathcal{G}^{(k)}$, and $T \subseteq \mathcal{G}^{(k+1)}$. Furthermore, if $\mathcal{G}^{(k)}$ satisfies \cref{def:space_group}, then $\mathcal{G}^{(k+1)}$ also satisfies \cref{def:space_group}.
\end{lemma}

\begin{proof}
	First, we prove $\mathcal G^{(k+1)}\leqslant\mathcal G^{(k)}$. Let
	\[
		\begin{aligned}
			\operatorname{Stab}_{G_{\bm A^{(k)}}}\bigl(X_R^{(k),[\bm y]}\bigr)
			\coloneqq \bigl\{\Psi_{\bar g}^{(k)}\in G_{\bm A^{(k)}}:\ \Psi_{\bar g}^{(k)}\bigl(X_R^{(k),[\bm y]}\bigr)=X_R^{(k),[\bm y]}\bigr\}
		\end{aligned}
	\]
	be the stabilizer subgroup of the subset $X_R^{(k),[\bm y]}$ in the discrete symmetry group $G_{\bm A^{(k)}}$.

	Denote the natural quotient map by $q^{(k)}:\mathcal G^{(k)}\to G^{(k)}\coloneqq \mathcal G^{(k)}/T$, and denote the $k$-th level discrete induced action homomorphism by
	\[
		\begin{aligned}
			\Psi^{(k)}:G^{(k)}\to \operatorname{Sym}(X_{\bm A^{(k)}}),\qquad \bar g\mapsto \Psi_{\bar g}^{(k)}
		\end{aligned}
	\]
	(its proof as a group homomorphism is identical to that of \cref{prop:G_A_group}). Define $G_{\bm A^{(k)}}\coloneqq \operatorname{Im}(\Psi^{(k)})$, and consider $\Psi^{(k)}$ as a surjective homomorphism onto $G_{\bm A^{(k)}}$. Further define the composite homomorphism
	\[
		\begin{aligned}
			\varpi^{(k)}\coloneqq \Psi^{(k)}\circ q^{(k)}:\mathcal G^{(k)}\to G_{\bm A^{(k)}},\qquad g\mapsto \Psi_{q^{(k)}(g)}^{(k)}.
		\end{aligned}
	\]
	Thus, from the definition of $\mathcal G^{(k+1)}$ in \cref{def:subproblem_linear_action}, we have
	\[
		\begin{aligned}
			\mathcal G^{(k+1)}=(\varpi^{(k)})^{-1}\!\left(\operatorname{Stab}_{G_{\bm A^{(k)}}}\bigl(X_R^{(k),[\bm y]}\bigr)\right).
		\end{aligned}
	\]
	Since $\varpi^{(k)}$ is a group homomorphism, and $\operatorname{Stab}_{G_{\bm A^{(k)}}}(X_R^{(k),[\bm y]})\leqslant G_{\bm A^{(k)}}$ is a subgroup, its inverse image $(\varpi^{(k)})^{-1}(\operatorname{Stab}_{G_{\bm A^{(k)}}}(X_R^{(k),[\bm y]}))$ is a subgroup of $\mathcal G^{(k)}$, i.e., $\mathcal G^{(k+1)}\leqslant\mathcal G^{(k)}$.

	Next, we prove $T\subseteq\mathcal G^{(k+1)}$. For any arbitrary $g\in T$, we have $q^{(k)}(g)=\bar e\in G^{(k)}$, hence
	\[
		\begin{aligned}
			\varpi^{(k)}(g)=\Psi_{\bar e}^{(k)}=\mathrm{id}_{X_{\bm A^{(k)}}}.
		\end{aligned}
	\]
	Therefore, $\varpi^{(k)}(g)\in \operatorname{Stab}_{G_{\bm A^{(k)}}}(X_R^{(k),[\bm y]})$, which means $g\in\mathcal G^{(k+1)}$. Thus, $T\subseteq\mathcal G^{(k+1)}$.

	Finally, we verify that $\mathcal G^{(k+1)}$ satisfies the three conditions of \cref{def:space_group}.
	\begin{enumerate}
		\item Since $\mathcal G^{(k+1)}\leqslant\mathcal G^{(k)}$, and $\mathcal G^{(k)}$ satisfies Condition (1) of \cref{def:space_group}, for any $g=(\bm R_g,\bm t_g)\in\mathcal G^{(k+1)}$, it still holds that $\bm R_g\in\operatorname{GL}^{\pm 1}(n,\mathbb Z)$ and $\bm t_g\in\mathbb Q^n$.
		\item As proven above, $T\subseteq\mathcal G^{(k+1)}$, which means Condition (2) of \cref{def:space_group} holds.
		\item From the description of \cref{def:space_group}, we know $T\triangleleft\mathcal G^{(k)}$, thus $T\triangleleft\mathcal G^{(k+1)}$ (normality is inherited by subgroups). Therefore, the inclusion map $\mathcal G^{(k+1)}\hookrightarrow\mathcal G^{(k)}$ induces a natural injection $\mathcal G^{(k+1)}/T\hookrightarrow \mathcal G^{(k)}/T$. Since $\mathcal G^{(k)}/T$ is finite, $\mathcal G^{(k+1)}/T$ is also finite, meaning Condition (3) of \cref{def:space_group} holds.
	\end{enumerate}

	In summary, $\mathcal G^{(k+1)}$ satisfies all conditions of \cref{def:space_group} and is therefore an $n$-dimensional space group.
\end{proof}

The finite symmetry group of a recursive subproblem can therefore be treated by the
same construction. Its stabilizer is a subgroup of the current discrete
symmetry group, and its action on the selected layer is transferred to the
local index group by the bijection
$[\bm x_1]\mapsto[\bm A_0^{(k)}\bm x_1+\bm y]$.
The resulting local symmetry group is a homomorphic image of this
stabilizer. In general dimensions, a composition series of this finite
node group gives at most two lifted generators per factor by
\cref{thm:hierarchical_group_action}; in three dimensions,
\cref{cor:three_dimensional_reconstruction} gives one per factor at every
recursive node. The grid compatibility and symmetry inheritance required
for the Fourier reduction are established separately by
\cref{prop:layer_permutation,prop:well_formed_inheritance}.

\begin{proposition}
	\label{prop:well_formed_inheritance}
	The subproblem $\mathcal{P}_{[\bm y]}^{(k+1)}$ constructed by \cref{def:subproblem_linear_action} is well-formed in the sense of \cref{def:well_formed_problem}, i.e.\ its discrete function satisfies \eqref{eq:wellformed_invariance}.
\end{proposition}

\begin{proof}
	We verify that the subproblem $\mathcal{P}_{[\bm y]}^{(k+1)}$ satisfies \eqref{eq:wellformed_invariance}. The argument proceeds in two steps, first we establish that the subproblem's grid $(\bm A^{(k+1)},\bm s^{(k+1)})$ is $G^{(k+1)}$-admissible (so that its discrete symmetry group $G_{\bm A^{(k+1)}}$ is well-defined), and then we prove the invariance itself.

	First, we establish that the subproblem's grid is $G^{(k+1)}$-admissible. For any arbitrary $g\in\mathcal G^{(k+1)}\subseteq\mathcal G^{(k)}$ (the inclusion is \cref{lem:subgroup_inheritance}), from
	\[
		\begin{aligned}
			\bm A^{(k)}=\bm A_0^{(k)}\bm A_1^{(k)},\qquad \bm A^{(k+1)}=\bm A_1^{(k)}
		\end{aligned}
	\]
	in \cref{def:subproblem_linear_action}, we obtain that the associated matrix in the subproblem satisfies
	\[
		\begin{aligned}
			\bm M_{\bar g}^{(k+1)}=(\bm A_0^{(k)})^{-1}\bm M_{\bar g}^{(k)}\bm A_0^{(k)}.
		\end{aligned}
	\]
	Since $\bm A_0^{(k)}$ is $G_{\bm A^{(k)}}$-compatible, i.e., for any $\bar g\in G^{(k)}$ it holds that
	\[
		\begin{aligned}
			\bm M_{\bar g}^{(k)}\bigl(\bm A_0^{(k)}\mathbb Z^n\bigr)=\bm A_0^{(k)}\mathbb Z^n,
		\end{aligned}
	\]
	it follows that $(\bm A_0^{(k)})^{-1}\bm M_{\bar g}^{(k)}\bm A_0^{(k)}$ leaves $\mathbb Z^n$ invariant. Taking the standard basis vectors $\bm e_j\in\mathbb Z^n$ shows that $\bm M_{\bar g}^{(k+1)}\bm e_j\in\mathbb Z^n$, hence $\bm M_{\bar g}^{(k+1)}$ is an integer matrix.
	Furthermore, since the parent problem $\mathcal P^{(k)}$ is well-formed, its grid is $G^{(k)}$-admissible, so $\bm M_{\bar g}^{(k)}\in \operatorname{GL}^{\pm 1}(n,\mathbb Z)$. Similarity transformations do not alter the determinant; therefore,
	\[
		\begin{aligned}
			\det\bm M_{\bar g}^{(k+1)}=\det\bm M_{\bar g}^{(k)}=\pm 1,
		\end{aligned}
	\]
	which yields
	\[
		\begin{aligned}
			\bm M_{\bar g}^{(k+1)}\in \operatorname{GL}^{\pm 1}(n,\mathbb Z).
		\end{aligned}
	\]

	Next, we consider the associated vector. By definition,
	\[
		\begin{aligned}
			\bm z_g^{(k+1)}
			 & =(\bm M_{\bar g}^{(k+1)}-\bm I_n)\bm s^{(k+1)}+\bm A^{(k+1)}\bm t_g                                                                                       \\
			 & =\left[(\bm A_0^{(k)})^{-1}\bm M_{\bar g}^{(k)}\bm A_0^{(k)}-\bm I_n\right](\bm A_0^{(k)})^{-1}(\bm y+\bm s^{(k)})+(\bm A_0^{(k)})^{-1}\bm A^{(k)}\bm t_g \\
			 & =(\bm A_0^{(k)})^{-1}\left[(\bm M_{\bar g}^{(k)}-\bm I_n)(\bm y+\bm s^{(k)})+\bm A^{(k)}\bm t_g\right]                                                    \\
			 & =(\bm A_0^{(k)})^{-1}\left(\bm M_{\bar g}^{(k)}\bm y+\bm z_g^{(k)}-\bm y\right).
		\end{aligned}
	\]
	On the other hand, since $g\in\mathcal G^{(k+1)}$, by definition we have
	\[
		\begin{aligned}
			\Psi_{\bar g}^{(k)}\bigl(X_R^{(k),[\bm y]}\bigr)=X_R^{(k),[\bm y]}.
		\end{aligned}
	\]
	By \cref{prop:layer_permutation}, the image of the layer under the action of $\Psi_{\bar g}^{(k)}$ is determined by the congruence class
	\[
		\begin{aligned}
			[\bm M_{\bar g}^{(k)}\bm y+\bm z_g^{(k)}]\in X_{\bm A_0^{(k)}}.
		\end{aligned}
	\]
	Since this layer is stabilized, it must be true that
	\[
		\begin{aligned}
			\bm M_{\bar g}^{(k)}\bm y+\bm z_g^{(k)}\equiv \bm y \pmod{\bm A_0^{(k)}\mathbb Z^n}.
		\end{aligned}
	\]
	Therefore, there exists $\bm \ell\in\mathbb Z^n$ such that
	\[
		\begin{aligned}
			\bm M_{\bar g}^{(k)}\bm y+\bm z_g^{(k)}-\bm y=\bm A_0^{(k)}\bm \ell,
		\end{aligned}
	\]
	and substituting this back into the above equation yields
	\[
		\begin{aligned}
			\bm z_g^{(k+1)}=\bm \ell\in\mathbb Z^n.
		\end{aligned}
	\]
	Hence, the subproblem's grid is $G^{(k+1)}$-admissible.

	Next, we prove the invariance \eqref{eq:wellformed_invariance} for the subproblem, i.e., that its function remains invariant under the action of its discrete symmetry group. For any arbitrary $\Psi_{\bar g}^{(k+1)}\in G_{\bm A^{(k+1)}}$ and $[\bm x_1]\in X_{\bm A^{(k+1)}}$, let
	\[
		\begin{aligned}
			[\bm x]=[\bm A_0^{(k)}\bm x_1+\bm y]\in X_R^{(k),[\bm y]}.
		\end{aligned}
	\]
	We first prove that the discrete actions of the parent problem and the subproblem satisfy
	\[
		\begin{aligned}
		\Psi_{\bar g}^{(k)}([\bm x])=[\bm y+\bm A_0^{(k)}\bm x_1'],
	\end{aligned}
	\]
	where $\bm x_1'\in\mathcal D_{\bm A^{(k+1)}}$ is the standard representative of the congruence class $\Psi_{\bar g}^{(k+1)}([\bm x_1])$.
	In fact, from the definition of the discrete action in the subproblem,
	\[
		\begin{aligned}
			\Psi_{\bar g}^{(k+1)}([\bm x_1])=[\bm M_{\bar g}^{(k+1)}\bm x_1+\bm z_g^{(k+1)}].
		\end{aligned}
	\]
	Then, utilizing the previously obtained
	\[
		\begin{aligned}
			\bm M_{\bar g}^{(k+1)}=(\bm A_0^{(k)})^{-1}\bm M_{\bar g}^{(k)}\bm A_0^{(k)},
			\qquad
			\bm z_g^{(k+1)}=(\bm A_0^{(k)})^{-1}(\bm M_{\bar g}^{(k)}\bm y+\bm z_g^{(k)}-\bm y),
		\end{aligned}
	\]
	we can calculate
	\[
		\begin{aligned}
			\bm y+\bm A_0^{(k)}\bigl(\bm M_{\bar g}^{(k+1)}\bm x_1+\bm z_g^{(k+1)}\bigr)
			 & =\bm y+\bm A_0^{(k)}(\bm A_0^{(k)})^{-1}\bm M_{\bar g}^{(k)}\bm A_0^{(k)}\bm x_1        \\
			 & \quad +\bm A_0^{(k)}(\bm A_0^{(k)})^{-1}(\bm M_{\bar g}^{(k)}\bm y+\bm z_g^{(k)}-\bm y) \\
			 & =\bm M_{\bar g}^{(k)}(\bm A_0^{(k)}\bm x_1+\bm y)+\bm z_g^{(k)}.
		\end{aligned}
	\]
	Taking the congruence class yields
	\[
		\begin{aligned}
			[\bm y+\bm A_0^{(k)}\bm x_1']
			= [\bm M_{\bar g}^{(k)}(\bm A_0^{(k)}\bm x_1+\bm y)+\bm z_g^{(k)}]
			= \Psi_{\bar g}^{(k)}([\bm x]).
		\end{aligned}
	\]

	Finally, by the definition of the subproblem's function and the discrete symmetry of the parent problem (\cref{def:well_formed_problem}),
	\[
		\begin{aligned}
			f_{[\bm y]}^{(k+1)}\bigl(\Psi_{\bar g}^{(k+1)}([\bm x_1])\bigr)
			 & = f^{(k)}\bigl([\bm y+\bm A_0^{(k)}\bm x_1']\bigr) \\
			 & = f^{(k)}\bigl(\Psi_{\bar g}^{(k)}([\bm x])\bigr)  \\
			 & = f^{(k)}([\bm x])                                 \\
			 & = f_{[\bm y]}^{(k+1)}([\bm x_1]).
		\end{aligned}
	\]
	Thus, $f^{(k+1)}$ is invariant under the action of $G_{\bm A^{(k+1)}}$, which is precisely \eqref{eq:wellformed_invariance}.

	In conclusion, the subproblem $\mathcal{P}_{[\bm y]}^{(k+1)}$ satisfies \eqref{eq:wellformed_invariance}, and is therefore a well-formed SFFT problem by \cref{def:well_formed_problem}.
\end{proof}

\paragraph{Pure translational subproblems}
If $G_{\bm A^{(k)}}$ is nontrivial and all associated matrices are
$\bm I_n$, apply the pure translation reduction in
\cref{sec:translation_derivation} with a translation-compatible factorization.
By \cref{prop:translation_free}, each representative layer has a trivial
stabilizer in the current discrete symmetry group. The remaining local DFTs
are computed by the standard FFT, and this branch terminates.

\subsubsection{Summary of the recursive framework}

The complete recursive structure of the SFFT algorithm is as follows
\begin{enumerate}
	\item Take the well-formed original problem $\mathcal{P}^{(0)}$ as the initial input.
	\item For each level $k=0,1,2,\ldots$
	      \begin{enumerate}
		      \item If the discrete index group is too small for an effective Cooley--Tukey reduction, directly execute the standard FFT and terminate.
		      \item If $G_{\bm A^{(k)}}$ contains a nontrivial linear action, choose a compatible decomposition $\bm A^{(k)}=\bm A_0^{(k)}\bm A_1^{(k)}$ with $|\det\bm A_0^{(k)}|>1$. Construct the required symmetry operations and layer representatives using \cref{thm:hierarchical_group_action}, then apply the linear reduction to obtain $|Y_{\mathrm{ref}}^{(k)}|$ smaller well-formed subproblems. If no such decomposition is available, execute the standard FFT and terminate.
		      \item If $G_{\bm A^{(k)}}$ is a nontrivial pure translation group, execute the translational symmetry reduction and terminate.
		      \item If $G_{\bm A^{(k)}} = \{\mathrm{id}\}$, directly execute the standard FFT and terminate.
	      \end{enumerate}
\end{enumerate}

Every recursive linear reduction decreases the number of local indices,
since
\[
    |X_{\bm A^{(k+1)}}|
    =\frac{|X_{\bm A^{(k)}}|}{|\det\bm A_0^{(k)}|}
    <|X_{\bm A^{(k)}}|.
\]
Hence each branch contains only finitely many such reductions. Together
with the FFT and pure translation terminal cases, this proves termination
of the recursive procedure.

\section{Computational complexity analysis}
\label{sec:complexity}

This section analyzes the computational complexities of the SFFT-LIN/SIFFT-LIN and SFFT-TRN/SIFFT-TRN algorithms. The symmetry operations and orbit representatives are precomputed for a fixed group, grid and shift; the estimates below describe transform execution with these inputs. We evaluate the algorithm complexity from the following two aspects
\begin{enumerate}
	\item \textbf{Arithmetic complexity}, primarily measured by the number of complex multiplications. 
	\item \textbf{Data access complexity}, measured by the number of non-contiguous memory index accesses (Contiguous accesses are considered low-order overhead). 
\end{enumerate}
On actual hardware, non-contiguous memory access significantly affects performance, but its cost is highly dependent on the specific implementation; therefore, this paper only provides a qualitative discussion. 

The following arithmetic complexity estimations are based on the radix-2 Cooley--Tukey FFT and count only the number of complex multiplications. Accordingly, the arithmetic complexity of a 1D FFT/IFFT of length $D$ is $\frac{1}{2}D\log_2 D$. In the same way, the constant matrices $\bm T_{c,r}$ and $\bm T_{c,t}$ are DFT matrices of finite abelian groups of orders $D_0$ and $d$, respectively (see \cref{def:linear_action_matrix,thm:Tc_inv,thm:Tc_inv_trans}); a full matrix--vector product with either of them is therefore counted at the cost of the corresponding fast transform, rather than at the dense-multiplication cost $D_0^2$ or $d^2$. Let the total grid size be $D=\det\bm A$. For the nontrivial linear reduction, let $D=D_0D_1$, where $D_0=\det\bm A_R^0$ and $D_1=\det\bm A_R^1$. For the pure translational reduction, similarly let $D=D_0D_1$, where $D_0=\det\bm A_T^0$ and $D_1=\det\bm A_T^1$. 

\subsection{Computational complexity of nontrivial linear reduction}

\subsubsection{SFFT-LIN}

The computational complexity of SFFT-LIN can be estimated through the following three steps:
\begin{enumerate}
	\item Performing a local FFT with respect to $\bm x_1$ for each $f_{\bm A, \bm s}([\bm A_R^0\bm x_1+\bm y])$ yields an arithmetic complexity of $|Y_{\mathrm{ref}}|\cdot \frac12 D_1\log_2 D_1$ and a non-contiguous access complexity of $0$. 
	\item Assembling the local transform vector $\hat{\bm V}(\bm h_1)$ yields an arithmetic complexity of $0$ and a non-contiguous access complexity of $(D_0-|Y_{\mathrm{ref}}|)D_1$. 
	\item Computing $\hat{\bm F}(\bm h_1)=\bm T_{c,r}\,\bm T_{s,r}(\bm h_1)\,\hat{\bm V}(\bm h_1)$ yields an arithmetic complexity of $\bigl(\frac12 D_0\log_2 D_0 + D_0\bigr)D_1$ and a non-contiguous access complexity of $0$. 
\end{enumerate}
Therefore, the computational complexity of SFFT-LIN is
\begin{itemize}
	\item Arithmetic complexity: $|Y_{\mathrm{ref}}|\cdot \frac12 D_1\log_2 D_1 + \bigl(\frac12 D_0\log_2 D_0 + D_0\bigr)D_1$.
	\item Data access complexity: $(D_0 - |Y_{\mathrm{ref}}|)D_1$.
\end{itemize}

\subsubsection{SIFFT-LIN}

The computational complexity of SIFFT-LIN can be estimated through the following two steps:

\begin{enumerate}
	\item Computing $\hat{V}_{\bm y}([\bm h_1])$ on the orbit representatives from the global Fourier coefficient vector $\hat{\bm F}(\bm h_1)$ yields an arithmetic complexity of $\min\bigl(|Y_{\mathrm{ref}}|(D_0+1),\, \frac12 D_0\log_2 D_0 + D_0\bigr)D_1$, since the required components can be obtained either by evaluating the $|Y_{\mathrm{ref}}|$ inner products directly or by applying the fast inverse transform associated with $\bm T_{c,r}^*$ and retaining the required components, and a data access complexity of $0$. 
	\item Performing a local IFFT with respect to $\bm h_1$ for each $\hat{V}_{\bm y}([\bm h_1])$ yields an arithmetic complexity of $|Y_{\mathrm{ref}}|\cdot \frac12 D_1\log_2 D_1$ and a data access complexity of $0$. 
\end{enumerate}
Therefore, the computational complexity of SIFFT-LIN is
\begin{itemize}
	\item Arithmetic complexity: $|Y_{\mathrm{ref}}|\cdot \frac12 D_1\log_2 D_1 + \min\bigl(|Y_{\mathrm{ref}}|(D_0+1),\, \frac12 D_0\log_2 D_0 + D_0\bigr)D_1$.
	\item Data access complexity: $0$.
\end{itemize}
The arithmetic complexity of SIFFT-LIN is no higher than that of SFFT-LIN, and it completely avoids non-contiguous memory accesses.

Next, we analyze its asymptotic speedup ratio. As $D_1\to\infty$ (with $D_0$ fixed, determined by the symmetry), the standard FFT complexity is $\frac12 D\log_2 D\sim \frac12 D_0 D_1\log_2 D_1$, whereas the dominant term for both SFFT-LIN and SIFFT-LIN is $|Y_{\mathrm{ref}}|\cdot \frac12 D_1\log_2 D_1$. Therefore, the theoretical speedup ratio for both is
\[
	\frac{D_0}{|Y_{\mathrm{ref}}|}.
\]

\subsection{Computational complexity of pure translational reduction}

From \cref{thm:zeta-count}, we have $|H_{\chi}| = |Y_{\mathrm{ref}}| = d = D_1/M$. In this scenario, all data accesses are contiguous; hence, the data access complexity is negligible. 

\subsubsection{SFFT-TRN and SIFFT-TRN}

Their arithmetic complexities are identical and can be estimated through the following two steps.
\begin{enumerate}
	\item The matrix-vector multiplication (either $\bm V(\bm x_0) = \bm T_{s,t}(\bm x_0)\bm T_{c,t}\bm F(\bm x_0)$ or its inverse formulation) yields an arithmetic complexity of $\bigl(\frac12 d\log_2 d + d\bigr)D_0$. 
	\item Performing the FFT (or IFFT) with respect to $\bm x_0$ (or $\bm h_0$) for each $V_{\bm h_1}(\bm x_0)$ (or $\hat f_{\bm A,\bm s}\left([\bm A_T^{1}\bm h_0 + \bm h_1]\right)$) yields an arithmetic complexity of $d\cdot \frac12 D_0\log_2 D_0$. 
\end{enumerate}
The total arithmetic complexity is
\[
	d\cdot \frac12 D_0\log_2 D_0 + \bigl(\frac12 d\log_2 d + d\bigr)D_0.
\]

Next, we analyze their asymptotic speedup ratio. As $D_0\to\infty$ (with $D_1$ fixed, determined by the symmetry), the standard FFT complexity is $\frac12 D\log_2 D\sim \frac12 D_0 D_1\log_2 D_0$, whereas the dominant term for SFFT-TRN/SIFFT-TRN is $d\cdot \frac12 D_0\log_2 D_0$. Therefore, the theoretical speedup ratio is
\[
	\frac{D_1}{d} = M = |G_{\bm A}|.
\]
That is, the speedup ratio equals the order of the discrete symmetry group.

\section{Recursive symmetry reduction of the SFFT for \texorpdfstring{$Im\bar{3}m$}{Im3m}}
\label{sec:algorithm_example}

This section takes the 3D space group $Im\bar{3}m$ (No. 229 in the International Tables for Crystallography) as an example to illustrate how the mathematical framework established previously can be specifically applied to the recursive reduction of the SFFT algorithm.

\begin{remark}
	Since the entire derivation focuses on the recursive reduction of the SFFT, the subsequent demonstrations in this section will be organized around the SFFT as the main thread. A completely analogous recursive derivation can be performed for the SIFFT, where each step corresponds to its SFFT counterpart; thus it will not be repeated here.
\end{remark}

Consider the case of $n=3$, where the translation subgroup is $T = \{ (\bm I_3, \bm t) : \bm t \in \mathbb{Z}^3 \}$. Let $\mathcal{S}_c \subset \operatorname{Aff}(\mathbb{R}^3)$ be a finite set consisting of the following six generators.
\[
	\mathcal{S}_c \coloneqq \{ g_{\alpha}, g_{\beta}, g_{\gamma}, g_{\sigma}, g_{\psi}, g_{\tau} \},
\]
where
\begin{itemize}
	\item $g_{\alpha} = (\bm R_{g_\alpha}, \bm 0)$, $\bm R_{g_\alpha} = \operatorname{diag}(-1, 1, 1)$.
	\item $g_{\beta} = (\bm R_{g_\beta}, \bm 0)$, $\bm R_{g_\beta} = \operatorname{diag}(1, -1, 1)$.
	\item $g_{\gamma} = (\bm R_{g_\gamma}, \bm 0)$, $\bm R_{g_\gamma} = \operatorname{diag}(1, 1, -1)$.
	\item $g_{\sigma} = (\bm R_{g_\sigma}, \bm 0)$, $\bm R_{g_\sigma} = \begin{pmatrix} 0 & 1 & 0 \\ 0 & 0 & 1 \\ 1 & 0 & 0 \end{pmatrix}$.
	\item $g_{\psi} = (\bm R_{g_\psi}, \bm 0)$, $\bm R_{g_\psi} = \begin{pmatrix} 1 & 0 & 0 \\ 0 & 0 & 1 \\ 0 & 1 & 0 \end{pmatrix}$.
	\item $g_{\tau} = (\bm I_3, \bm t_{g_\tau})$, $\bm t_{g_\tau} = (\tfrac{1}{2}, \tfrac{1}{2}, \tfrac{1}{2})^\mathsf{T}$.
\end{itemize}
The space group $\mathcal{G} = \langle T \cup \mathcal{S}_c \rangle_{\operatorname{Aff}(\mathbb{R}^3)}$ satisfies all the conditions outlined in \cref{def:space_group} and corresponds to the space group $Im\bar{3}m$ in the International Tables for Crystallography \cite{Aroyo2016ITC_A}.

To apply the discretization framework introduced in \cref{subsec:grid_definition}, let $m\geq 4$ be an integer (to ensure the grid is sufficiently large to fully demonstrate the subsequent four-layer recursive reduction process), the discretization matrix $\bm A = N \bm I_3$ where $N = 2^m$, and the translation vector $\bm s = (\tfrac{1}{2}, \tfrac{1}{2}, \tfrac{1}{2})^\mathsf{T}$. The six generators satisfy the admissibility condition, and the admissible elements are closed under multiplication ($\bm z_{g_1g_2}=\bm M_{g_1}\bm z_{g_2}+\bm z_{g_1}$); hence $(\bm A,\bm s)$ is $G$-admissible.

If $f$ possesses $G$-symmetry, then its discretization $f_{\bm A,\bm s}$ inherits this symmetry; that is, $f_{\bm A,\bm s}(\Psi_{\bar g}([\bm x])) = f_{\bm A,\bm s}([\bm x])$ holds for any $\bar g\in G$. To elucidate the action of the generators of $G_{\bm A}$, let $\bm x = (x_1,x_2,x_3)^\mathsf{T} \in \mathcal{D}_{\bm A}$ be the standard representative, which yields the discrete actions corresponding to the six generators.
\begin{align*}
	\Psi_{\bar g_\alpha}([\bm x]) & = [(-x_1-1,x_2,x_3)^\mathsf{T}],          \\
	\Psi_{\bar g_\beta}([\bm x])  & = [(x_1,-x_2-1,x_3)^\mathsf{T}],          \\
	\Psi_{\bar g_\gamma}([\bm x]) & = [(x_1,x_2,-x_3-1)^\mathsf{T}],          \\
	\Psi_{\bar g_\sigma}([\bm x]) & = [(x_2,x_3,x_1)^\mathsf{T}],             \\
	\Psi_{\bar g_\psi}([\bm x])   & = [(x_1,x_3,x_2)^\mathsf{T}],             \\
	\Psi_{\bar g_\tau}([\bm x])   & = [(x_1+N/2,x_2+N/2,x_3+N/2)^\mathsf{T}].
\end{align*}
Define four key subgroups: $G_1=\langle\Psi_{\bar g_\alpha},\Psi_{\bar g_\beta},\Psi_{\bar g_\gamma}\rangle$ (mirror reflection subgroup), $G_2=\langle\Psi_{\bar g_\sigma}\rangle$ (three-fold rotation subgroup), $G_3=\langle\Psi_{\bar g_\psi}\rangle$ (coordinate permutation subgroup), and $G_4=\langle\Psi_{\bar g_\tau}\rangle$ (body-centered translation subgroup). Then $G_{\bm A}=\langle G_1\cup G_2\cup G_3\cup G_4\rangle$. To simplify the subsequent representation and indexing of standard representatives, we introduce the notation $\bm y_{abc}\coloneqq (a, b, c)^\mathsf{T}\in\mathbb{Z}^3$. To distinguish between spatial and frequency indices, the standard frequency representative is denoted by $\bm k_{abc}\coloneqq (a,b,c)^\mathsf{T}\in\mathbb Z^3$.

\cref{fig:im3m_recursive_flow} gives an overview of the four-layer recursion for the $Im\bar{3}m$ case. When the recursion termination conditions are met, the standard FFT is called directly. Otherwise, the symmetry reduction for the nontrivial linear-action case is performed sequentially (the first three layers, progressively shrinking the residual symmetry group), and the reduction for the pure translation symmetry is conducted in the final layer, thereby yielding several subproblems that can be evaluated directly via FFT.
\begin{figure}[tbp]
	\centering
	\includegraphics[width=0.96\linewidth]{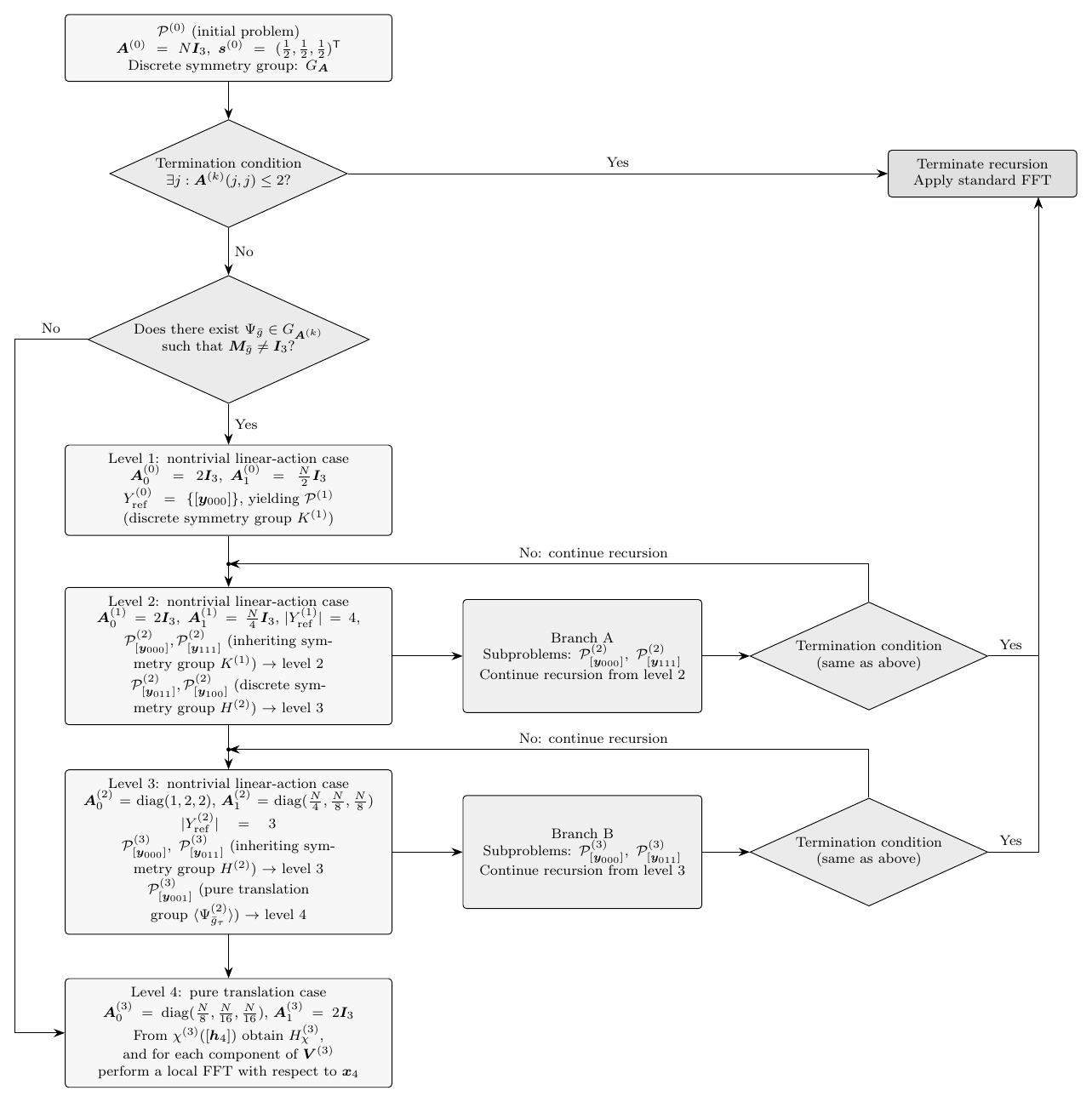}
	\caption{Flowchart of the recursive symmetry reduction of SFFT for the $Im\bar{3}m$ case (four-layer example).}
	\label{fig:im3m_recursive_flow}
\end{figure}

\begin{remark}
	\label{rem:unified_computation}
	In the subsequent layers of symmetry reduction, the cardinality of the quotient group $X_{\bm A_0^{(l)}}$ (or $X_{\bm A_1^{(l)}}$ for the pure translation case) is minimal (e.g., $8$ or $4$). The permutation action of the discrete symmetry group on these finite index sets can be directly computed and determined using the action formulas derived earlier for the corresponding cases.
	Therefore, the subsequent propositions regarding the orbit partitioning and stabilizer subgroups for each layer can be rigorously verified by directly computing the action of the group generators on the finite sets and combining them with the orbit-stabilizer theorem. To avoid repetition, apart from the first layer of reduction where a detailed derivation is provided as an illustrative example, the lengthy modular arithmetic is omitted in the subsequent layers, and only the verification results will be presented directly.
\end{remark}

\subsection{Layer 1 reduction (Case I): symmetry reduction under \texorpdfstring{$G_{\bm A^{(0)}}$}{G\_A(0)}}\label{subsec:G1_draft}

The parameters of the initial problem $\mathcal{P}^{(0)}$ are $\bm A^{(0)}=N\bm I_3$, $\bm s^{(0)}=(\tfrac12,\tfrac12,\tfrac12)^\mathsf{T}$, and the discrete function is $f^{(0)}=f_{\bm A,\bm s}$. Since the initial grid parameters satisfy the admissibility condition and the original function possesses $G$-symmetry, \cref{rem:base_case} indicates that this initial problem is well-formed. Furthermore, the explicit actions of the aforementioned generators show that the linear part of, for instance, $\Psi_{\bar g_\alpha}\in G_{\bm A^{(0)}}=G_{\bm A}$ satisfies $\bm M_{\bar g_\alpha}\neq \bm I_3$; thus, this initial problem falls under Case I. We then choose a $G_{\bm A}$-compatible decomposition $\bm A^{(0)}=\bm A_0^{(0)}\bm A_1^{(0)}$, where $\bm A_0^{(0)}=2\bm I_3$ and $\bm A_1^{(0)}=(N/2)\bm I_3$. Under the above decomposition and grid parameters, we define
\[
	\Theta_R^{(0)}\coloneqq \left\{X_R^{(0),[\bm y]}:[\bm y]\in X_{\bm A_0^{(0)}}\right\},\qquad
	X_R^{(0),[\bm y]}\coloneqq \left\{[\bm A_0^{(0)}\bm x_1+\bm y]:[\bm x_1]\in X_{\bm A_1^{(0)}}\right\}.
\]

\begin{proposition}
	\label{prop:first_layer_orbit_draft}
	The orbit of the representative layer $X_R^{(0),[\bm y_{000}]}$ is
	\[
		O_{[\bm y_{000}]} = \Theta_R^{(0)} = \left\{ X_R^{(0),[\bm y]} : [\bm y]\in X_{\bm A_0^{(0)}} \right\}.
	\]
	Its stabilizer subgroup is $K\coloneqq \langle G_2\cup G_3\cup G_4\rangle$, and a set of left coset representatives with respect to $K$ can be chosen as
	\[
		\mathcal{R}([\bm y_{000}]) = G_1 = \langle\Psi_{\bar g_\alpha},\Psi_{\bar g_\beta},\Psi_{\bar g_\gamma}\rangle,\qquad |\mathcal{R}([\bm y_{000}])|=8.
	\]
\end{proposition}

\begin{proof}
	As a computational paradigm for \cref{rem:unified_computation}, a detailed derivation is provided here. By \cref{prop:layer_permutation}, the permutation action induced by $\Psi_{\bar g}\in G_{\bm A}$ on the layer set $\Theta_R^{(0)}$ is given by $X_R^{(0),[\bm y]}\mapsto X_R^{(0),[\bm M_{\bar g}\bm y+\bm z_g]}$. In particular, when taking the representative layer $X_R^{(0),[\bm y_{000}]}$, its image $X_R^{(0),[\bm z_g]}$ under the action of $\Psi_{\bar g}$ is entirely determined by the discrete translation term $[\bm z_g] \in X_{2\bm I_3}$.

	\noindent\textbf{(1) Orbit. }
	From \cref{prop:discrete_action_formula}, we have $\bm z_g=(\bm M_g-\bm I_3)\bm s+N\bm t_g$. Since the congruence class $[\bm z_g]$ of this affine discrete translation term in the quotient group $X_{2\bm I_3}$ is independent of the choice of the representative $g$ in $\mathcal{G}$ (see \cref{lem:rep-independence}), for computational convenience, we examine the mirror subgroup $G_1=\langle\Psi_{\bar g_\alpha},\Psi_{\bar g_\beta},\Psi_{\bar g_\gamma}\rangle$ and select representatives $g$ satisfying the translation part $\bm t_g=\bm 0$ for each of its generators $\Psi_{\bar g}$. Substituting $\bm s=(\tfrac12,\tfrac12,\tfrac12)^\mathsf T$, we calculate the discrete translation terms $\bm z_g$ for these three operations as $(-1,0,0)^\mathsf T$, $(0,-1,0)^\mathsf T$, and $(0,0,-1)^\mathsf T$, respectively. As $\Psi_{\bar g}$ traverses the $8$ elements of $G_1$, the generated discrete translation terms $[\bm z_g]$ exactly span all $8$ independent congruence classes $[\bm y_{abc}]$ ($a,b,c\in\{0,1\}$) in the quotient group $X_{2\bm I_3}$. Thus, $O_{[\bm y_{000}]} = \Theta_R^{(0)}$.

	\noindent\textbf{(2) Stabilizer Subgroup. }
	First, we verify that the representative layer $X_R^{(0),[\bm y_{000}]}$ remains invariant under the action of the group $K$. Based on the same representative independence, for the generators $\Psi_{\bar g_\sigma}$ and $\Psi_{\bar g_\psi}$, we similarly choose their representatives $g$ in $\mathcal{G}$ satisfying $\bm t_g=\bm 0$, which yields $(\bm M_g-\bm I_3)\bm s=\bm 0$, and hence $\bm z_{g_\sigma}=\bm z_{g_\psi}=\bm 0$. For the body-centered translation generator $\Psi_{\bar g_\tau}$, we choose its representative such that $\bm M_{g_\tau}=\bm I_3$ accompanied by the fractional translation $\bm t_{g_\tau}=(\tfrac12,\tfrac12,\tfrac12)^\mathsf T$, yielding $\bm z_{g_\tau}=N\bm t_{g_\tau}=(N/2,N/2,N/2)^\mathsf T$. Given the condition $m\ge 4$, halving $N$ gives $N/2=2^{m-1}$, which is an even number; therefore, in the quotient group $X_{2\bm I_3}$, we have $[\bm z_{g_\tau}]=[\bm 0]\pmod{2}$. Thus, $K=\langle\Psi_{\bar g_\sigma},\Psi_{\bar g_\psi},\Psi_{\bar g_\tau}\rangle \subseteq \operatorname{Stab}_{G_{\bm A}}\bigl(X_R^{(0),[\bm y_{000}]}\bigr)$.

	It is known that the order of the discrete symmetry group in this example is $|G_{\bm A}|=96$ \cite{Aroyo2016ITC_A}, and according to (1), the orbit size is $|O_{[\bm y_{000}]}|=8$. By the orbit-stabilizer theorem, the order of its stabilizer subgroup must be $96/8=12$. On the other hand, from the group order relation, we have $|K|=|G_{\bm A}|/|G_1|=12$. Since both have the same order and satisfy the inclusion relation, $\operatorname{Stab}_{G_{\bm A}}\bigl(X_R^{(0),[\bm y_{000}]}\bigr)=K$. Finally, because the $8$ operations of $G_1$ exactly map this representative layer to $8$ distinct layers within the orbit, we can directly take $\mathcal R([\bm y_{000}])=G_1$ as a set of left coset representatives.
\end{proof}

The orbit representative set is $Y_{\mathrm{ref}}^{(0)}=\{[\bm y_{000}]\}$. Lexicographically ordering the row index sequence $X_{\bm A_0^{(0)}}^{\mathrm{lex}}=(\bm y_{000},\bm y_{001},\dots,\bm y_{111})$ and the column index sequence $\mathcal{I}^{\mathrm{lex}}=((0,0),\dots,(0,7))$. Combined with \cref{eq:core_matrix_relation_linear_action}, for any $\bm h_1\in \mathcal{D}_{\bm A_1^{(0)}}$,
\[
	\hat{\bm F}^{(0)}(\bm h_1) = \bm T_{c,r}^{(0)}\,\bm T_{s,r}^{(0)}(\bm h_1)\,\hat{\bm V}^{(0)}(\bm h_1),
\]
where $\bm T_{c,r}^{(0)}$ is an $8\times 8$ constant matrix (see Appendix A for its explicit form), $\bm T_{s,r}^{(0)}(\bm h_1)$ is a phase diagonal matrix, and $\hat{\bm V}^{(0)}$ is assembled by rearranging a single local DFT $\hat{V}_{\bm y_{000}}$ via symmetries. Therefore, this layer only requires $1$ instance of a 3D FFT of size $(N/2)^3$.

This layer of reduction generates a unique subproblem $\mathcal{P}^{(1)}$, with $\bm A^{(1)}=(N/2)\bm I_3$, $\bm s^{(1)}=(\tfrac14,\tfrac14,\tfrac14)^\mathsf{T}$, $\mathcal{G}^{(1)}=\{ g\in\mathcal{G}^{(0)} : \Psi_{\bar g}^{(0)}(X_R^{(0),[\bm y_{000}]})=X_R^{(0),[\bm y_{000}]}\}$, and the new function $f^{(1)}([\bm x_1])=f^{(0)}([2\bm x_1])$. The discrete symmetry group of this subproblem is denoted as
\[
	K^{(1)}\coloneqq G_{\bm A^{(1)}}\coloneqq \left\langle \Psi_{\bar g_\sigma}^{(1)},\,\Psi_{\bar g_\psi}^{(1)},\,\Psi_{\bar g_\tau}^{(1)}\right\rangle.
\]
Here, for the standard representative $\bm x_1=(x_1,x_2,x_3)^\mathsf T$ of any $[\bm x_1]\in X_{\bm A^{(1)}}$, the explicit discrete actions of the three generators are given by
\begin{align*}
	\Psi_{\bar g_\sigma}^{(1)}([\bm x_1]) & =[(x_2,x_3,x_1)^\mathsf T],             \\
	\Psi_{\bar g_\psi}^{(1)}([\bm x_1])   & =[(x_1,x_3,x_2)^\mathsf T],             \\
	\Psi_{\bar g_\tau}^{(1)}([\bm x_1])   & =[(x_1+N/4,x_2+N/4,x_3+N/4)^\mathsf T].
\end{align*}
Since the residual symmetry group $K^{(1)}$ is nontrivial, we can proceed to the next layer of symmetry reduction.

\subsection{Layer 2 reduction (case I): symmetry reduction under \texorpdfstring{$K^{(1)}$}{K(1)}}
\label{subsec:second_layer_draft}

Now consider the well-formed subproblem $\mathcal{P}^{(1)}$ generated in the first layer (whose well-formedness is guaranteed by \cref{prop:well_formed_inheritance}), with its discrete symmetry group being $K^{(1)}$. Since the linear parts of $\Psi_{\bar g_\sigma}^{(1)}$ and $\Psi_{\bar g_\psi}^{(1)}$ are not identity matrices, this layer still falls under Case I. We choose the decomposition
\[
    \bm A^{(1)} \coloneqq \bm A_1^{(0)} = \bm A_0^{(1)}\bm A_1^{(1)},\qquad \bm A_0^{(1)}=2\bm I_3,\quad \bm A_1^{(1)}=(N/4)\bm I_3.
\]
Under the above decomposition and grid parameters, we define
\[
	\Theta_R^{(1)}\coloneqq \left\{X_R^{(1),[\bm y]}:[\bm y]\in X_{\bm A_0^{(1)}}\right\},\qquad
	X_R^{(1),[\bm y]}\coloneqq \left\{[\bm A_0^{(1)}\bm x_2+\bm y]:[\bm x_2]\in X_{\bm A_1^{(1)}}\right\}.
\]

\begin{proposition}
	\label{prop:second_layer_orbits_draft}
	The permutation action induced by the group $K^{(1)}$ on $\Theta_R^{(1)}$ yields four orbits
	\[
		\begin{aligned}
			O_{[\bm y_{000}]} & =\left\{X_R^{(1),[\bm y_{000}]}\right\},                                                     \\
			O_{[\bm y_{011}]} & =\left\{X_R^{(1),[\bm y_{011}]},\,X_R^{(1),[\bm y_{110}]},\,X_R^{(1),[\bm y_{101}]}\right\}, \\
			O_{[\bm y_{100}]} & =\left\{X_R^{(1),[\bm y_{100}]},\,X_R^{(1),[\bm y_{010}]},\,X_R^{(1),[\bm y_{001}]}\right\}, \\
			O_{[\bm y_{111}]} & =\left\{X_R^{(1),[\bm y_{111}]}\right\}.
		\end{aligned}
	\]
	Accordingly, their stabilizer subgroups and sets of left coset representatives can be chosen as follows.
	\begin{itemize}
		\item For both $X_R^{(1),[\bm y_{000}]}$ and $X_R^{(1),[\bm y_{111}]}$, the stabilizer subgroup is $K^{(1)}$, and
		      \[
			      \mathcal R([\bm y_{000}])=\mathcal R([\bm y_{111}])=\{\mathrm{id}\}.
		      \]
		\item For both $X_R^{(1),[\bm y_{011}]}$ and $X_R^{(1),[\bm y_{100}]}$, the stabilizer subgroup is
		      \[
			      \langle \Psi_{\bar g_\psi}^{(1)},\Psi_{\bar g_\tau}^{(1)}\rangle \coloneqq H,
		      \]
		      and
		      \[
			      \mathcal R([\bm y_{011}])=\mathcal R([\bm y_{100}])
			      =\{\mathrm{id},\,\Psi_{\bar g_\sigma}^{(1)},\,(\Psi_{\bar g_\sigma}^{(1)})^2\}.
		      \]
	\end{itemize}
\end{proposition}

\begin{proof}
	The proof follows a similar idea to that of \cref{prop:first_layer_orbit_draft}. By combining the discrete action formulas of the generators $\Psi_{\bar g_\sigma}^{(1)}$ and $\Psi_{\bar g_\psi}^{(1)}$ and directly computing their induced actions on the finite quotient group $X_{2\bm I_3}$ (which has a cardinality of only $8$), one checks that their actions are equivalent to a 3-cycle permutation and a transposition of the coordinate components, respectively, while $\Psi_{\bar g_\tau}^{(1)}$ degenerates into the identity mapping. By enumerating the action trajectories of the generators on the quotient group, the four aforementioned orbits can be obtained. Furthermore, using $|K^{(1)}|=12$ and the orbit-stabilizer theorem, the order of the stabilizer subgroup corresponding to each orbit can be determined, which in turn allows for the verification and reconstruction of its corresponding generators and left coset representatives. The detailed verification is omitted; see \cref{rem:unified_computation}.
\end{proof}

Let the set of orbit representatives be $Y_{\mathrm{ref}}^{(1)}\coloneqq \{[\bm y_{000}],[\bm y_{011}],[\bm y_{100}],[\bm y_{111}]\}\subset X_{\bm A_0^{(1)}}$. Lexicographically order the sequence of orbit representatives as $Y_{\mathrm{ref}}^{(1),\mathrm{lex}}=(\bm y_{000},\bm y_{011},\bm y_{100},\bm y_{111})$. Lexicographically order the row index sequence
$X_{\bm A_0^{(1)}}^{\mathrm{lex}}=(\bm y_{000},\bm y_{001},\dots,\bm y_{111})$,
and choose the column index sequence as
\[
	\mathcal{I}^{\mathrm{lex}}=((0,0),(1,0),(1,1),(1,2),(2,0),(2,1),(2,2),(3,0)),
\]
where the second component is enumerated according to the order in the corresponding set of coset representatives $\mathcal R([\bm y])$.

Combined with \cref{eq:core_matrix_relation_linear_action}, for any $\bm h_2\in \mathcal{D}_{\bm A_1^{(1)}}$, we have
\[
	\hat{\bm F}^{(1)}(\bm h_2) = \bm T_{c,r}^{(1)}\,\bm T_{s,r}^{(1)}(\bm h_2)\,\hat{\bm V}^{(1)}(\bm h_2),
\]
where $\bm T_{c,r}^{(1)}$ is an $8\times8$ constant matrix, $\bm T_{s,r}^{(1)}(\bm h_2)$ is a phase diagonal matrix, and $\hat{\bm V}^{(1)}$ can be assembled by rearranging $4$ local DFTs $\{\hat{V}_{\bm y}:[\bm y]\in Y_{\mathrm{ref}}^{(1)}\}$ via symmetries. Therefore, this layer only requires $4$ instances of a 3D FFT of size $(N/4)^3$.

This layer of reduction generates four subproblems $\{\mathcal{P}_{[\bm y]}^{(2)}:[\bm y]\in Y_{\mathrm{ref}}^{(1)}\}$. For each subproblem, the parameters are $\bm A^{(2)}=\bm A_1^{(1)}=(N/4)\bm I_3$, $\bm s^{(2)}=(\bm A_0^{(1)})^{-1}(\bm y+\bm s^{(1)})$, $\mathcal{G}^{(2)}=\{ g\in\mathcal{G}^{(1)} : \Psi_{\bar g}^{(1)}(X_R^{(1),[\bm y]})=X_R^{(1),[\bm y]}\}$, and the corresponding new function is $f^{(2)}([\bm x_2])=f^{(1)}([2\bm x_2+\bm y])$.

The subproblems $\mathcal{P}_{[\bm y_{000}]}^{(2)}$ and $\mathcal{P}_{[\bm y_{111}]}^{(2)}$ inherit the symmetry group $K^{(1)}$, and thus can continue to invoke the reduction algorithm in this section. For the other two subproblems $\mathcal{P}_{[\bm y_{011}]}^{(2)}$ and $\mathcal{P}_{[\bm y_{100}]}^{(2)}$, they inherit the new symmetry group $H$. Denote their induced discrete symmetry group as
\[H^{(2)}\coloneqq G_{\bm A^{(2)}}\coloneqq \left\langle \Psi_{\bar g_\psi}^{(2)},\,\Psi_{\bar g_\tau}^{(2)}\right\rangle.\]
For the standard representative $\bm x_2=(x_1,x_2,x_3)^\mathsf T$ of any $[\bm x_2]\in X_{\bm A^{(2)}}$, the explicit discrete actions of these two generators are
\begin{align*}
	\Psi_{\bar g_\psi}^{(2)}([\bm x_2]) & =[(x_1,x_3,x_2)^\mathsf T],             \\
	\Psi_{\bar g_\tau}^{(2)}([\bm x_2]) & =[(x_1+N/8,x_2+N/8,x_3+N/8)^\mathsf T].
\end{align*}
Since the residual symmetry group $H^{(2)}$ is nontrivial, these two subproblems can still undergo the next layer of symmetry reduction with respect to $H^{(2)}$.

\subsection{Layer 3 reduction (case I): symmetry reduction under the \texorpdfstring{$H^{(2)}$}{H(2)} branch}
\label{subsec:third_layer_draft}

Now consider the well-formed subproblem $\mathcal{P}_{[\bm y_{011}]}^{(2)}$ generated in the second layer (its well-formedness is similarly guaranteed by \cref{prop:well_formed_inheritance}, and the case for $\mathcal{P}_{[\bm y_{100}]}^{(2)}$ is completely analogous). Its discrete symmetry group is $H^{(2)}$. Since the linear part of $\Psi_{\bar g_\psi}^{(2)}$ is not an identity matrix, this layer still belongs to Case I. We choose the decomposition
\[\bm A^{(2)}\coloneqq\bm A_1^{(1)} = \bm A_0^{(2)}\bm A_1^{(2)},\qquad \bm A_0^{(2)}=\operatorname{diag}(1,2,2),\quad \bm A_1^{(2)}=\operatorname{diag}(N/4,N/8,N/8).\]
Accordingly, under the above decomposition and grid parameters, for any $[\bm y]\in X_{\bm A_0^{(2)}}$, we define
\[\Theta_R^{(2)}\coloneqq \left\{X_R^{(2),[\bm y]}:[\bm y]\in X_{\bm A_0^{(2)}}\right\},\qquad
	X_R^{(2),[\bm y]}\coloneqq \left\{[\bm A_0^{(2)}\bm x_3+\bm y]:[\bm x_3]\in X_{\bm A_1^{(2)}}\right\}.\]

\begin{proposition}
	\label{prop:third_layer_orbits_draft}
	The permutation action induced by $H^{(2)}$ on $\Theta_R^{(2)}$ yields three orbits
	\[\begin{aligned}
			O_{[\bm y_{000}]} & =\left\{X_R^{(2),[\bm y_{000}]}\right\},                           \\
			O_{[\bm y_{001}]} & =\left\{X_R^{(2),[\bm y_{001}]},\,X_R^{(2),[\bm y_{010}]}\right\}, \\
			O_{[\bm y_{011}]} & =\left\{X_R^{(2),[\bm y_{011}]}\right\}.
		\end{aligned}\]
	Accordingly, their stabilizer subgroups and sets of left coset representatives can be chosen as follows.
	\begin{itemize}
		\item For both $X_R^{(2),[\bm y_{000}]}$ and $X_R^{(2),[\bm y_{011}]}$, the stabilizer subgroup is $H^{(2)}$, and
		      \[
			      \mathcal R([\bm y_{000}])=\mathcal R([\bm y_{011}])=\{\mathrm{id}\}.
		      \]
		\item For $X_R^{(2),[\bm y_{001}]}$, the stabilizer subgroup is
		      \[
			      \langle\Psi_{\bar g_\tau}^{(2)}\rangle,
		      \]
		      and
		      \[
			      \mathcal R([\bm y_{001}])=\{\mathrm{id},\,\Psi_{\bar g_\psi}^{(2)}\}.
		      \]
	\end{itemize}
\end{proposition}

\begin{proof}
	From the explicit discrete actions of the generators, the conclusions can be obtained via enumerative computations over the finite quotient group $X_{\bm A_0^{(2)}}$ (see \cref{rem:unified_computation} for details).
\end{proof}

Let the set of orbit representatives be $Y_{\mathrm{ref}}^{(2)}\coloneqq \{[\bm y_{000}],[\bm y_{001}],[\bm y_{011}]\}\subset X_{\bm A_0^{(2)}}$. Lexicographically order the sequence of orbit representatives as $Y_{\mathrm{ref}}^{(2),\mathrm{lex}}=(\bm y_{000},\bm y_{001},\bm y_{011})$. Lexicographically order the row index sequence $X_{\bm A_0^{(2)}}^{\mathrm{lex}}=(\bm y_{000},\bm y_{001},\bm y_{010},\bm y_{011})$,
and select the column index sequence as
\[\mathcal{I}^{\mathrm{lex}}=((0,0),(1,0),(1,1),(2,0)),\]
where the second component is enumerated according to the order in the corresponding set of coset representatives $\mathcal R([\bm y])$.

Combined with \cref{eq:core_matrix_relation_linear_action}, for any $\bm h_3\in \mathcal{D}_{\bm A_1^{(2)}}$, we have
\[\hat{\bm F}^{(2)}(\bm h_3) = \bm T_{c,r}^{(2)}\,\bm T_{s,r}^{(2)}(\bm h_3)\,\hat{\bm V}^{(2)}(\bm h_3),\]
where $\bm T_{c,r}^{(2)}$ is a $4 \times 4$ constant matrix, $\bm T_{s,r}^{(2)}(\bm h_3)$ is a phase diagonal matrix, and $\hat{\bm V}^{(2)}$ can be assembled by rearranging $3$ local DFTs $\{\hat{V}_{\bm y}:[\bm y]\in Y_{\mathrm{ref}}^{(2)}\}$ via symmetries. Therefore, this layer only requires $3$ instances of a 3D FFT of size $(N/4)\times(N/8)\times(N/8)$.

Consequently, the third layer reduction generates three subproblems $\{\mathcal{P}_{[\bm y]}^{(3)}:[\bm y]\in Y_{\mathrm{ref}}^{(2)}\}$. Among them, $\mathcal{P}_{[\bm y_{000}]}^{(3)}$ and $\mathcal{P}_{[\bm y_{011}]}^{(3)}$ continue to inherit the symmetry group $H^{(2)}$, thus allowing for the continuation of the same type of symmetry reduction. However, the residual symmetry group of $\mathcal{P}_{[\bm y_{001}]}^{(3)}$ degenerates into $\langle\Psi_{\bar g_\tau}^{(2)}\rangle$, which contains only pure translations. Therefore, the next step directly transitions to the pure translation symmetry reduction procedure of Case II.

\subsection{Layer 4 reduction (case II): pure-translation symmetry reduction}
\label{subsec:fourth_layer_draft}

Now consider the well-formed subproblem with pure translations $\mathcal{P}_{[\bm y_{001}]}^{(3)}$ generated in the third layer (its well-formedness is also guaranteed by \cref{prop:well_formed_inheritance}). Its discrete symmetry group is
\[G_{\bm A^{(3)}}=\left\langle \Psi_{\bar g_\tau}^{(3)}\right\rangle.\]
For the standard representative $\bm x_3=(x_1,x_2,x_3)^\mathsf T$ of any $[\bm x_3]\in X_{\bm A^{(3)}}$, the explicit discrete action of this generator is
\begin{equation}
	\label{eq:psi_tau_3_explicit}
	\Psi_{\bar g_\tau}^{(3)}([\bm x_3])=\bigl[(x_1+N/8,\ x_2+N/16,\ x_3+N/16)^\mathsf T\bigr].
\end{equation}
Hence, this layer falls under Case II, and the pure translation reduction framework presented in \cref{sec:translation_derivation} can be directly applied.

We adopt the following decomposition
\[\bm A^{(3)} \coloneqq \bm A_1^{(2)} = \bm A_0^{(3)}\bm A_1^{(3)},\qquad
    \bm A_0^{(3)}\coloneqq \operatorname{diag}(N/8,N/16,N/16),\quad
    \bm A_1^{(3)}\coloneqq 2\bm I_3.\]
Since we can choose $\bm t_{g_\tau}=(\tfrac12,\tfrac12,\tfrac12)^\mathsf T$ for $g_\tau=(\bm I_3,\bm t_{g_\tau})$, it follows that
\[\bm A_1^{(3)}\bm t_{g_\tau}=(1,1,1)^\mathsf T\in\mathbb Z^3,\]
thereby satisfying the translation compatibility condition in \cref{def:translation_compatibility} for $\bm A_1^{(3)}$. Accordingly, under the above decomposition, we define
\[\Theta_T^{(3)}\coloneqq \left\{X_T^{(3),[\bm y]}:[\bm y]\in X_{\bm A_1^{(3)}}\right\},\qquad
	X_T^{(3),[\bm y]}\coloneqq \left\{[\bm A_0^{(3)}\bm y+\bm x_4]:[\bm x_4]\in X_{\bm A_0^{(3)}}\right\}.\]

\begin{proposition}
	\label{prop:fourth_layer_orbits_draft}
	The permutation action induced by the group $G_{\bm A^{(3)}}=\langle \Psi_{\bar g_\tau}^{(3)}\rangle$ on $\Theta_T^{(3)}$ yields four orbits
	\[\begin{aligned}
			O_{[\bm y_{000}]} & =\left\{X_T^{(3),[\bm y_{000}]},\,X_T^{(3),[\bm y_{111}]}\right\}, \\
			O_{[\bm y_{010}]} & =\left\{X_T^{(3),[\bm y_{010}]},\,X_T^{(3),[\bm y_{101}]}\right\}, \\
			O_{[\bm y_{100}]} & =\left\{X_T^{(3),[\bm y_{100}]},\,X_T^{(3),[\bm y_{011}]}\right\}, \\
			O_{[\bm y_{110}]} & =\left\{X_T^{(3),[\bm y_{110}]},\,X_T^{(3),[\bm y_{001}]}\right\}.
		\end{aligned}\]
	Furthermore, under pure translation actions, the stabilizer subgroup is always the trivial group $\{\mathrm{id}\}$. Thus, for any orbit representative $[\bm y]$, the set of left coset representatives can be chosen as
	\[\mathcal R([\bm y])=\{\mathrm{id},\,\Psi_{\bar g_\tau}^{(3)}\}.\]
\end{proposition}

\begin{proof}
	From the explicit discrete actions of the generators, the conclusions can be derived via enumerative computations over the finite quotient group $X_{\bm A_1^{(3)}}$ (see \cref{rem:unified_computation} for details).
\end{proof}

Let the set of orbit representatives be $Y_{\mathrm{ref}}^{(3)}\coloneqq \{[\bm y_{000}],[\bm y_{010}],[\bm y_{100}],[\bm y_{110}]\}\subset X_{\bm A_1^{(3)}}$. Lexicographically order the sequence of orbit representatives as $Y_{\mathrm{ref}}^{(3),\mathrm{lex}}=(\bm y_{000},\bm y_{010},\bm y_{100},\bm y_{110})$.

Next, according to \cref{eq:Chi_t,thm:zeta-count}, it is known that $\chi^{(3)}$ depends only on the outer characteristic frequency offset $[\bm h_4] \in X_{\bm A_1^{(3)}}$, and can only take the values $0$ or $2$. For the standard representative $\bm h_4=(h_1,h_2,h_3)^\mathsf T\in\{0,1\}^3$, utilizing $\bm z_{g_\tau}=\bm A^{(3)}\bm t_{g_\tau}=(N/8,N/16,N/16)^\mathsf T$ yields
\[\chi^{(3)}([\bm h_4])
	=1+\exp\!\left(-\pi i(h_1+h_2+h_3)\right)
	= \begin{cases}
		2, & h_1+h_2+h_3 \text{ is even}, \\[1ex]
		0, & h_1+h_2+h_3 \text{ is odd}.
	\end{cases}\]
The lexicographically ordered sequence of valid characteristic frequencies is given by
\[H_{\chi}^{(3),\mathrm{lex}}=(\bm k_{000},\bm k_{011},\bm k_{101},\bm k_{110})\subset X_{\bm A_1^{(3)}}.\]

Combined with \cref{eq:core_matrix_relation_translation}, for any $\bm x_4\in \mathcal{D}_{\bm A_0^{(3)}}$, we have
\[\bm V^{(3)}(\bm x_4) = \bm T_{s,t}^{(3)}(\bm x_4)\, \bm T_{c,t}^{(3)}\, \bm F^{(3)}(\bm x_4),\]
where $\bm T_{c,t}^{(3)}$ is a $4\times4$ constant matrix, $\bm T_{s,t}^{(3)}(\bm x_4)$ is a phase diagonal matrix, and $\bm F^{(3)}(\bm x_4)$ is a column vector constructed from the function values at the subgrid points corresponding to each orbit representative. After obtaining $\bm V^{(3)}$, performing a local 3D FFT with respect to $\bm x_4$ on each of its components reconstructs $\hat{\bm F}^{(3)}(\bm h_4)$. Therefore, this subproblem only requires computing $4$ instances of a 3D FFT of size $(N/8)\times(N/16)\times(N/16)$ on the smaller index group $X_{\bm A_0^{(3)}}$, and the recursion terminates here.

\section{Numerical experiments}
\label{sec:numerical_results}
This section evaluates the computational efficiency and numerical accuracy of the SFFT/SIFFT relative to the standard FFT/IFFT through numerical experiments. Space groups $Fm\bar{3}m$, $Im\bar{3}m$, and $Ia\bar{3}d$ are selected as test cases. All experiments are executed in single-threaded mode on a Linux system equipped with an Intel(R) Xeon(R) Silver 4214 CPU (2.20GHz). The standard FFT/IFFT is implemented using the FFTW C library \cite{frigo1998fftw}.

For the three aforementioned space groups, we uniformly set $\bm{A}=N\bm{I}_3$ and $\bm{s}=(1/2,1/2,1/2)^\mathsf{T}$, and employ the recursive symmetry reduction scheme presented in \cref{sec:algorithm_example}. As $D\to\infty$, the theoretical speedups of $Im\bar{3}m$ and $Ia\bar{3}d$ under the fully recursive reduction scheme both tend to $96$. For space group $Fm\bar{3}m$, the first three layers of reduction are identical to those of $Im\bar{3}m$, but differences emerge in the pure translation reduction. The pure translation reduction for $Fm\bar{3}m$ provides an additional $4$-fold speedup, yielding a total asymptotic theoretical speedup of $192$ under the fully recursive reduction.

We emphasize that the full recursive reduction is by no means out of reach: the framework developed above guarantees that every recursively generated subproblem is well-formed and exactly invertible, so the recursion can in principle be continued down to the leaf subproblems, some of which reduce to blocks as small as $2^3$. However, in actual computations the recursive reduction is worthwhile only down to a certain depth. The purpose of the SFFT is to replace an FFT of size $D$ by FFTs of a smaller size together with a linear combination and a data reshuffling. For large $D$, this replacement reduces the cost; once the recursion has gone deep enough for $D$ to be small, however, the cost of a standard FFT of that size may well be lower than that of the smaller FFTs together with the linear combination and the reshuffling, so that a further reduction is no longer worthwhile. For this reason, the numerical tests in this paper adopt a partial recursion, executing one reduction step for both the second and third layers. Consequently, the SFFT/SIFFT performance reported in the subsequent tables corresponds to this partial reduction configuration. Under this setup, the asymptotic theoretical speedup for $Im\bar{3}m$ and $Ia\bar{3}d$ is $128/3$, whereas for $Fm\bar{3}m$, it is $256/3$.

The reductions used here take $\bm A_0^{(0)}=\bm A_0^{(1)}=2\bm I_3$, $\bm A_0^{(2)}=\operatorname{diag}(1,2,2)$ and $\bm A_1^{(3)}=2\bm I_3$, for which all entries of the constant matrices $\bm T_{c,r}^{(k)}$ and $\bm T_{c,t}^{(3)}$ equal $\pm1$ (see \cref{appendix:matrices}). Their application therefore requires additions only, not complex multiplications; the complex multiplications of the implementation come solely from the phase matrices and the leaf transforms.

Within the range $N\in\{64,128,256,512\}$, the SFFT/SIFFT and standard FFT/IFFT are compared across the three space groups, with the results detailed in \cref{tab:SFFT_SIFFT_PATIENT}. The \texttt{Error} in the table denotes the maximum relative error, defined for SFFT and SIFFT respectively as
\[
	\frac{\max_{\bm h\in\mathcal D_{\bm A}}\bigl\lvert \hat f_{\mathrm{SFFT}}(\bm h)-\hat f_{\mathrm{FFT}}(\bm h)\bigr\rvert}{\max_{\bm h\in\mathcal D_{\bm A}}\bigl\lvert \hat f_{\mathrm{FFT}}(\bm h)\bigr\rvert},
	\qquad
	\frac{\max_{\bm x\in\mathcal D_{\bm A}}\bigl\lvert f_{\mathrm{SIFFT}}(\bm x)-f_{\mathrm{IFFT}}(\bm x)\bigr\rvert}{\max_{\bm x\in\mathcal D_{\bm A}}\bigl\lvert f_{\mathrm{IFFT}}(\bm x)\bigr\rvert}.
\]
Both types of errors are on the order of machine precision, demonstrating that the SFFT/SIFFT algorithms are numerically consistent with the standard FFT/IFFT. In terms of execution time, the minimum speedups for $Im\bar{3}m$ and $Ia\bar{3}d$ are approximately $5\times$ (SFFT, $N=64$) and $12\times$ (SIFFT, $N=128$), respectively. For $Fm\bar{3}m$, the minimum speedups are approximately $10\times$ (SFFT) and $25\times$ (SIFFT), both occurring at $N=128$.

The numerical results reveal that the speedup of the SIFFT is generally higher than that of the SFFT. The arithmetic cost of the recombination step is not lower in the forward direction (see \cref{sec:complexity}), but the forward transform must additionally assemble the local transform vector $\hat{\bm V}(\bm h_1)$, which costs $(D_0-\lvert Y_{\mathrm{ref}}\rvert)D_1$ non-contiguous memory accesses, whereas the inverse transform incurs none. Since $128/3$ and $256/3$ are asymptotic values, the finite-$N$ speedups reported below are not expected to attain them.

\begin{table}[htbp]
	\centering
	\renewcommand{\arraystretch}{1.2} % Control overall row spacing
	\small
	\caption{Computational Performance Comparison between FFT/IFFT and SFFT/SIFFT}
	\label{tab:SFFT_SIFFT_PATIENT}
	\begin{tabular}{|>{\centering\arraybackslash}p{1.0cm}|%
		>{\centering\arraybackslash}p{3.2cm}|%
		>{\centering\arraybackslash}p{1.5cm}|%
		>{\centering\arraybackslash}p{1.5cm}|%
		>{\centering\arraybackslash}p{1.5cm}|
		>{\centering\arraybackslash}p{1.5cm}|}
		\hline
		 & $N$            & 64             & 128            & 256            & 512            \\
		\hline
		\multirow{8}{*}{$Im\bar{3}m$}
		 & FFT Time (ms)   & 2.86           & 25.72          & 371.13         & 3012.99        \\
		 & SFFT Time (ms)  & 0.54           & 4.95           & 47.00          & 452.23         \\
		 & {speedup}        & \textbf{5.30}  & \textbf{5.20}  & \textbf{7.90}  & \textbf{6.66}  \\
		 &  {Error} & $8.8\times10^{-18}$ & $3.5\times10^{-18}$ & $1.4\times10^{-18}$ & $2.2\times10^{-16}$       \\
		 & IFFT Time (ms)  & 2.68           & 27.94          & 293.47         & 3322.11        \\
		 & SIFFT Time (ms) & 0.20           & 2.33           & 20.77          & 201.05         \\
		 & {speedup}        & \textbf{13.40} & \textbf{12.00} & \textbf{14.13} & \textbf{16.52} \\
		 &  {Error} & $1.0\times10^{-15}$ & $9.5\times10^{-16}$ & $1.1\times10^{-15}$ & $1.2\times10^{-15}$       \\
		\hline
		\multirow{8}{*}{$Ia\bar{3}d$}
		 & FFT Time (ms)   & 2.93           & 26.05          & 292.67         & 3213.07        \\
		 & SFFT Time (ms)  & 0.54           & 4.97           & 47.79          & 455.16         \\
		 & {speedup}        & \textbf{5.43}  & \textbf{5.24}  & \textbf{6.12}  & \textbf{7.06}  \\
		 &  {Error} & $3.5\times10^{-17}$ & $3.2\times10^{-17}$ & $2.8\times10^{-17}$ & $2.2\times10^{-17}$       \\
		 & IFFT Time (ms)  & 2.79           & 27.83          & 299.75         & 3437.85        \\
		 & SIFFT Time (ms) & 0.19           & 2.32           & 20.70          & 180.95         \\
		 & {speedup}        & \textbf{14.68} & \textbf{12.00} & \textbf{14.48} & \textbf{19.00} \\
		 &  {Error} & $8.5\times10^{-16}$ & $1.0\times10^{-15}$ & $1.0\times10^{-15}$ & $1.3\times10^{-15}$       \\
		\hline
		\multirow{8}{*}{$Fm\bar{3}m$}
		 & FFT Time (ms)   & 2.69           & 25.78          & 384.02         & 3562.60        \\
		 & SFFT Time (ms)  & 0.26           & 2.55           & 24.07          & 218.98         \\
		 & {speedup}        & \textbf{10.35} & \textbf{10.11} & \textbf{15.95} & \textbf{16.27} \\
		 &  {Error} & $1.6\times10^{-17}$ & $8.4\times10^{-18}$ & $2.9\times10^{-18}$ & $2.2\times10^{-16}$       \\
		 & IFFT Time (ms)  & 2.60           & 29.31          & 301.94         & 2973.67        \\
		 & SIFFT Time (ms) & 0.09           & 1.16           & 10.22          & 100.50         \\
		 & {speedup}        & \textbf{28.89} & \textbf{25.27} & \textbf{29.54} & \textbf{29.59} \\
		 &  {Error} & $7.8\times10^{-16}$ & $1.1\times10^{-15}$ & $1.0\times10^{-15}$ & $1.2\times10^{-15}$       \\
		\hline
	\end{tabular}
\end{table}

\section{Conclusion and outlook}
    \label{sec:conclusion}
    In this paper, we have developed a unified mathematical framework for SFFT/SIFFT for periodic discrete data with finite affine symmetries. Starting from the geometric definition of crystallographic groups, we proved that such groups can be represented, in suitable coordinates, in an arithmetic form with integral linear parts, rational translational parts, and finite quotient groups. For the finite symmetry quotient, the composition series and the small generating sets of its simple factors give a hierarchical construction of all symmetry operations and of the complete orbit of a representative point. In three dimensions, solvability makes every factor cyclic of prime order, so one lifted generator suffices at each composition step. We then introduced $G$-admissible grids and induced discrete actions, converting crystallographic symmetries into permutation actions on finite index groups. Combined with compatible multidimensional Cooley--Tukey decompositions and orbit--stabilizer structures, this construction supplies the operations and representatives needed to reduce global DFTs/IDFTs to local DFTs/IDFTs together with finite-dimensional matrix recombinations.
    
    At the algorithmic level, we derived two fundamental classes of symmetry-reduction formulas, corresponding respectively to nontrivial linear parts and pure translational parts. For each case, we provided the associated matrix formulations, invertibility results, and implementable algorithms. In the nontrivial linear case, the orthogonality of the constant matrices guarantees a stable and reversible conversion between local and global spectral data. In the pure translation case, the characterization of the phase factors and of $\chi$ reveals the intrinsic origin of effective frequency selection, and shows that the asymptotic speedup factor is precisely determined by the order of the discrete symmetry group. Building on these results, we further organized linear reductions, translational reductions, and standard FFT calls into a unified recursive framework by formally defining well-formed SFFT problems and proving that well-formedness is inherited by the resulting subproblems. The space group $Im\bar{3}m$ was used as a representative example to demonstrate the complete transition from continuous symmetry to a concrete recursive implementation.
    
    The complexity analysis and numerical experiments further confirm the effectiveness of the proposed framework. Theoretically, the asymptotic speedup of SFFT-LIN/SIFFT-LIN is governed by $D_0/|Y_{\mathrm{ref}}|$, whereas in the pure translation case the asymptotic speedup equals the order of the discrete symmetry group. Partial recursive implementations for the three space groups $Fm\bar{3}m$, $Im\bar{3}m$, and $Ia\bar{3}d$ show that both SFFT and SIFFT can achieve substantial acceleration over standard FFT/IFFT algorithms while maintaining errors at the level of machine precision. At the same time, the measured performance has not yet reached the theoretical limits. The principal remaining bottleneck is the cache overhead caused by non-contiguous memory access. These observations are consistent with the theoretical framework developed in this paper, while fully realizing its potential will require further engineering optimization.
    
    Future work may proceed in two directions: implementation refinement and theoretical extension. On the implementation side, it would be valuable to develop a more complete and better optimized implementation, including automatic generation of the constant and phase matrices, optimized data layouts, and parallelization strategies, so as to further narrow the gap between theoretical complexity and practical performance. On the theoretical side, the present approach may be extended to more general nonperiodic structures and non-Euclidean geometric settings, such as Fourier analysis for quasicrystals and spherical harmonic transforms on the sphere. Such extensions would broaden the applicability of the proposed framework to a wider range of problems in scientific computing.

\appendix

\section{Constant and phase matrices for symmetry reduction}
\label{appendix:matrices}

This appendix supplements the specific expressions for the constant matrices and phase diagonal matrices that appear in the four-layer example in \cref{sec:algorithm_example}, and clarifies their row/column index sequences. The fixed orbit representative order, coset representative order, and standard representative conventions established in the main text are retained throughout the following.

\subsection{Layer 1 reduction matrices}
In the first layer of reduction, the row index sequence is given by
\begin{align}
	X_{\bm A_0^{(0)}}^{\mathrm{lex}}=(\bm k_{000},\bm k_{001},\bm k_{010},\bm k_{011},\bm k_{100},\bm k_{101},\bm k_{110},\bm k_{111}),
\end{align}
and the column index sequence is given by
\begin{align}
	\mathcal I^{\prec}=((0,0),(0,1),(0,2),(0,3),(0,4),(0,5),(0,6),(0,7)).
\end{align}
The corresponding column representative congruence classes are sequentially
\begin{align}
	[\bm y_{000}],\ [\bm y_{001}],\ [\bm y_{010}],\ [\bm y_{011}],\ [\bm y_{100}],\ [\bm y_{101}],\ [\bm y_{110}],\ [\bm y_{111}].
\end{align}
The constant reduction matrix $\bm{T}_{c,r}^{(0)} \in \mathbb{C}^{8 \times 8}$ is computed as
\begin{align}
	\bm{T}_{c,r}^{(0)}           & =
	\begin{bmatrix}
		1 & 1  & 1  & 1  & 1  & 1  & 1  & 1  \\
		1 & -1 & 1  & -1 & 1  & -1 & 1  & -1 \\
		1 & 1  & -1 & -1 & 1  & 1  & -1 & -1 \\
		1 & -1 & -1 & 1  & 1  & -1 & -1 & 1  \\
		1 & 1  & 1  & 1  & -1 & -1 & -1 & -1 \\
		1 & -1 & 1  & -1 & -1 & 1  & -1 & 1  \\
		1 & 1  & -1 & -1 & -1 & -1 & 1  & 1  \\
		1 & -1 & -1 & 1  & -1 & 1  & 1  & -1
	\end{bmatrix}.
\end{align}
The phase diagonal matrix $\bm{T}_{s,r}^{(0)}(\bm h_1) \in \mathbb{C}^{8 \times 8}$ is computed as
\begin{align}
	\bm{T}_{s,r}^{(0)}(\bm h_1) = \operatorname{diag}\Bigl(
	 & 1,\,
	e^{2\pi i\, \bm h_1^\mathsf{T}\bm{y}_{001}/N},\,
	e^{2\pi i\, \bm h_1^\mathsf{T}\bm{y}_{010}/N},\,
	e^{2\pi i\, \bm h_1^\mathsf{T}\bm{y}_{011}/N},\, \notag \\
	 & e^{2\pi i\, \bm h_1^\mathsf{T}\bm{y}_{100}/N},\,
	e^{2\pi i\, \bm h_1^\mathsf{T}\bm{y}_{101}/N},\,
	e^{2\pi i\, \bm h_1^\mathsf{T}\bm{y}_{110}/N},\,
	e^{2\pi i\, \bm h_1^\mathsf{T}\bm{y}_{111}/N}
	\Bigr).
\end{align}
The above diagonal elements correspond sequentially to the $8$ elements in the column index sequence $\mathcal I^{\mathrm{lex}}$. That is, they follow the same column order as the constant matrix $\bm{T}_{c,r}^{(0)}$, so that the two factors acting on column $v$ of \cref{eq:core_matrix_relation_linear_action} are built from one and the same coset representative $\bar g_{0v}$.

\subsection{Layer 2 reduction matrices}

In the second layer of reduction, the row index sequence is given by
\begin{align}
	X_{\bm A_0^{(1)}}^{\mathrm{lex}}=(\bm k_{000},\bm k_{001},\bm k_{010},\bm k_{011},\bm k_{100},\bm k_{101},\bm k_{110},\bm k_{111}),
\end{align}
and the column index sequence is given by
\begin{align}
	\mathcal I^{\mathrm{lex}}=((0,0),(1,0),(1,1),(1,2),(2,0),(2,1),(2,2),(3,0)).
\end{align}
The corresponding column representative congruence classes are sequentially
\begin{align}
	[\bm y_{000}],\ [\bm y_{011}],\ [\bm y_{110}],\ [\bm y_{101}],\ [\bm y_{100}],\ [\bm y_{001}],\ [\bm y_{010}],\ [\bm y_{111}].
\end{align}

For the second-layer subproblem, the constant reduction matrix $\bm{T}_{c,r}^{(1)} \in \mathbb{C}^{8 \times 8}$ is computed as
\begin{equation}
	\bm{T}_{c,r}^{(1)} =
	\begin{bmatrix}
		1 & 1  & 1  & 1  & 1  & 1  & 1  & 1  \\
		1 & -1 & 1  & -1 & 1  & -1 & 1  & -1 \\
		1 & -1 & -1 & 1  & 1  & 1  & -1 & -1 \\
		1 & 1  & -1 & -1 & 1  & -1 & -1 & 1  \\
		1 & 1  & -1 & -1 & -1 & 1  & 1  & -1 \\
		1 & -1 & -1 & 1  & -1 & -1 & 1  & 1  \\
		1 & -1 & 1  & -1 & -1 & 1  & -1 & 1  \\
		1 & 1  & 1  & 1  & -1 & -1 & -1 & -1
	\end{bmatrix}.
\end{equation}
The phase diagonal matrix $\bm{T}_{s,r}^{(1)}(\bm h_2) \in \mathbb{C}^{8 \times 8}$ is computed as
\begin{align}
	\bm{T}_{s,r}^{(1)}(\bm h_2) = \operatorname{diag}\Bigl(
	 & 1,\,
	e^{-2\pi i\, \bm h_2^\mathsf{T}\bm{y}_{011}/(N/2)},\,
	e^{-2\pi i\, \bm h_2^\mathsf{T}\bm M^{(1)}_{\bar g_{\sigma}}\bm{y}_{011}/(N/2)},\, \notag \\
	 & e^{-2\pi i\, \bm h_2^\mathsf{T}(M^{(1)}_{\bar g_{\sigma}})^2\bm{y}_{011}/(N/2)},\,
	e^{-2\pi i\, \bm h_2^\mathsf{T}\bm{y}_{100}/(N/2)},\,
	e^{-2\pi i\, \bm h_2^\mathsf{T}\bm M^{(1)}_{\bar g_{\sigma}}\bm{y}_{100}/(N/2)},\, \notag \\
	 & e^{-2\pi i\, \bm h_2^\mathsf{T}(M^{(1)}_{\bar g_{\sigma}})^2\bm{y}_{100}/(N/2)},\,
	e^{-2\pi i\, \bm h_2^\mathsf{T} \bm{y}_{111} / (N/2)}
	\Bigr)\\
	=\operatorname{diag}\Bigl(
	 & 1,\,
	e^{-2\pi i\, \bm h_2^\mathsf{T}\bm{y}_{011}/(N/2)},\,
	e^{-2\pi i\, \bm h_2^\mathsf{T}\bm{y}_{110}/(N/2)},\,
	e^{-2\pi i\, \bm h_2^\mathsf{T}\bm{y}_{101}/(N/2)}, \notag \\
	&
	 e^{-2\pi i\, \bm h_2^\mathsf{T}\bm{y}_{100}/(N/2)},\,
	e^{-2\pi i\, \bm h_2^\mathsf{T}\bm{y}_{001}/(N/2)},\,
	e^{-2\pi i\, \bm h_2^\mathsf{T}\bm{y}_{010}/(N/2)}, \notag \\
	&
	e^{-2\pi i\, \bm h_2^\mathsf{T}\bm{y}_{111}/(N/2)}\Bigr),
\end{align}
where $\bm M^{(1)}_{\bar g_{\sigma}}$ corresponds to the linear part of the generator $\Psi_{[g_\sigma]}^{(1)}$, i.e.\ the three-cycle $(x_1,x_2,x_3)\mapsto(x_2,x_3,x_1)$. The above diagonal elements correspond sequentially to the 8 elements in the column index sequence $\mathcal I^{\mathrm{lex}}$.

\subsection{Layer 3 reduction matrices}

In the third layer of reduction, the row index sequence is given by
\begin{align}
	X_{\bm A_0^{(2)}}^{\mathrm{lex}}=(\bm k_{000},\bm k_{001},\bm k_{010},\bm k_{011}),
\end{align}
and the column index sequence is given by
\begin{align}
	\mathcal I^{\mathrm{lex}}=((0,0),(1,0),(1,1),(2,0)).
\end{align}
The corresponding column representative congruence classes are sequentially
\begin{align}
	[\bm y_{000}],\ [\bm y_{001}],\ [\bm y_{010}],\ [\bm y_{011}].
\end{align}

For the third-layer subproblem, the constant rotation reduction matrix $\bm{T}_{c,r}^{(2)} \in \mathbb{C}^{4 \times 4}$ is computed as
\begin{equation}
	\bm{T}_{c,r}^{(2)} =
	\begin{bmatrix}
		1 & 1  & 1  & 1  \\
		1 & -1 & 1  & -1 \\
		1 & 1  & -1 & -1 \\
		1 & -1 & -1 & 1
	\end{bmatrix}
\end{equation}
The phase diagonal matrix $\bm{T}_{s,r}^{(2)}(\bm h_3) \in \mathbb{C}^{4 \times 4}$ is computed as
\begin{align}
	\bm{T}_{s,r}^{(2)}(\bm h_3) = \operatorname{diag}\Bigl(
	 & 1,\,
	e^{-2\pi i\, \bm h_3^\mathsf{T}(\bm A^{(2)})^{-1}\bm{y}_{001}},\, \notag           \\
	 & e^{-2\pi i\, \bm h_3^\mathsf{T}(\bm A^{(2)})^{-1}\bm M^{(2)}_{\bar g_\psi}\bm{y}_{001}},\,
	e^{-2\pi i\, \bm h_3^\mathsf{T}(\bm A^{(2)})^{-1} \bm{y}_{011}}
	\Bigr) \\ 
	=\operatorname{diag}\Bigl(
	 & 1,\,
	e^{-2\pi i\, \bm h_3^\mathsf{T}(\bm A^{(2)})^{-1}\bm{y}_{001}},\,
	e^{-2\pi i\, \bm h_3^\mathsf{T}(\bm A^{(2)})^{-1}\bm{y}_{010}},\,
	e^{-2\pi i\, \bm h_3^\mathsf{T}(\bm A^{(2)})^{-1}\bm{y}_{011}}\Bigr),
\end{align}
where $\bm A^{(2)} = \frac{N}{4}\bm{I}_3$, and $\bm M^{(2)}_{\bar g_\psi}$ corresponds to the linear part of the generator $\Psi_{[g_\psi]}^{(2)}$, i.e.\ the transposition $(x_1,x_2,x_3)\mapsto(x_1,x_3,x_2)$, so that $\bm M^{(2)}_{\bar g_\psi}\bm{y}_{001}=\bm y_{010}$. The diagonal elements correspond sequentially to the $4$ elements in the column index sequence above.

\subsection{Layer 4 (pure translation) reduction matrices}

In the fourth layer of pure translation reduction, the row index sequence is given by
\begin{align}
	H_{\chi}^{(3),\mathrm{lex}}=(\bm k_{000},\bm k_{011},\bm k_{101},\bm k_{110}),
\end{align}
and the column index sequence is given by
\begin{align}
	Y_{\mathrm{ref}}^{(3),\mathrm{lex}}=(\bm y_{000},\bm y_{010},\bm y_{100},\bm y_{110}).
\end{align}

The constant reduction matrix $\bm{T}_{c,t}^{(3)}\in\mathbb{C}^{4\times 4}$ is computed as
\begin{align}
	\bm{T}_{c,t}^{(3)} & =
	\begin{bmatrix}
		1 & 1  & 1  & 1  \\
		1 & -1 & 1  & -1 \\
		1 & 1  & -1 & -1 \\
		1 & -1 & -1 & 1
	\end{bmatrix}.
\end{align}
The phase diagonal matrix $\bm{T}_{s,t}^{(3)}(\bm x_4)\in\mathbb{C}^{4\times4}$ is computed as
\begin{align}
	\bm{T}_{s,t}^{(3)}(\bm x_4) = \operatorname{diag}\Bigl(
	 & 1,\,
	e^{-2\pi i\, \bm k_{011}^\mathsf{T}(\bm A^{(3)})^{-1}\bm x_4},\,
	e^{-2\pi i\, \bm k_{101}^\mathsf{T}(\bm A^{(3)})^{-1}\bm x_4},\,
	e^{-2\pi i\, \bm k_{110}^\mathsf{T}(\bm A^{(3)})^{-1}\bm x_4}
	\Bigr).
\end{align}
where $\bm A^{(3)}=\bm A_1^{(2)}=\operatorname{diag}(N/4,N/8,N/8)$, and the order of the aforementioned diagonal elements is consistent with the row index sequence $H_{\chi}^{(3),\mathrm{lex}}$.

\end{document}